\documentclass{amsart}

\title{On extension of overconvergent log isocrystals on log smooth varieties}

\usepackage{amsmath, amsfonts, amssymb, amsthm, graphicx,tikz-cd, mathtools, braket}
\usetikzlibrary{arrows}
\newtheorem{thm}{Theorem}[subsection]
\theoremstyle{definition}
\newtheorem{defn}[thm]{Definition}
\newtheorem{lem}[thm]{Lemma}
\newtheorem{prop}[thm]{Proposition}
\newtheorem{rem}[thm]{Remark}
\newtheorem{cor}[thm]{Corollary}

\newtheorem{situation}[thm]{Situation}

\def\N{\mathbb{N}}
\def\Z{\mathbb{Z}}

\def\R{\mathbb{R}}
\def\AA#1{\mathbb{A}_{#1}}
\def\A#1#2{\mathbb{A}_{#1,#2}}
\def\lr#1{\left(#1\right)}
\def\op#1{\!\left(#1\right)}
\def\abs#1{\left\vert #1 \right\vert}
\def\brace#1{\left\lbrace #1 \right\rbrace}
\def\angbra#1{\left\langle#1\right\rangle}
\def\brack#1{\left[#1\right]}
\def\clop#1{\left[#1\right)}
\def\LNM{\mathrm{LNM}}
\def\ULNM{\mathrm{ULNM}}
\def\id{\mathrm{id}}
\def\Ker{\mathrm{Ker}}
\def\Im{\mathrm{Im}}
\def\Coker{\mathrm{Coker}}
\def\Mat{\mathrm{Mat}}
\def\Spec{\mathrm{Spec}}
\def\Spm{\mathrm{Spm}}
\def\Spf{\mathrm{Spf}}
\def\Ext{\mathrm{Ext}}
\def\Hom{\mathrm{Hom}}
\def\O#1{\mathcal{O}_{#1}}
\def\tens#1{\mathbin{\mathop{\otimes}\displaylimits_{#1}}}
\def\fibpro#1{\mathbin{\mathop{\times}\displaylimits_{#1}}}
\def\amsum#1{\mathbin{\mathop{\oplus}\displaylimits_{#1}}}
\def\Kbar{\overline{K}}
\def\Spm{\mathrm{Spm}}
\def\dsum{\displaystyle \sum}
\def\RNN{\left[0,\infty\right)}
\def\tube#1{\left]#1\right[}
\def\ov#1{\overline{#1}}
\def\ul#1{\underline{#1}}
\def\mc#1{\mathcal{#1}}
\def\mb#1{\mathbf{#1}}
\def\mr#1{\mathrm{#1}}

\def\OK{\mathcal{V}}
\def\et{\mathrm{\acute{e}t}}

\numberwithin{equation}{thm}
\newtheorem{claim}[equation]{Claim}

\allowdisplaybreaks

\begin{document}

\author[K. Kasaura]{Kazumi Kasaura}
\address{Graduate~School~of~Mathematical~Sciences, University~of~Tokyo, 3-8-1 Komaba, Meguro-ku, Tokyo 153-8914, Japan}
\email{kasaura@ms.u-tokyo.ac.jp}

\begin{abstract}
By works of Kedlaya and Shiho, it is known that, for a smooth variety $\overline{X}$ over a field of positive characteristic and its simple normal crossing divisor $Z$, an overconvergent isocrystal on the compliment of $Z$ satisfying a certain monodromy condition can be extended to a convergent log isocrystal on $\left(\overline{X}, \mathcal{M}_Z\right)$, where $\mathcal{M}_Z$ is the log structure associated to $Z$.
We prove a generalization of this result: for a log smooth variety $\left(\overline{X},\mathcal{M}\right)$ satisfying some conditions, an overconvergent log isocrystal on the trivial locus of a direct summand of $\mathcal{M}$ satisfying a certain monodromy condition can be extended to a convergent log isocrystal on $\left(\overline{X}, \mathcal{M}\right)$.
\end{abstract}



\maketitle

\setcounter{section}{-1}
\section{Introduction}
Let $K$ be a field of characteristic $0$ complete with respect to a non-Archimedean valuation
whose residue field $k$ has a positive characteristic $p$. Let $\OK$ be the ring of integer of $K$.

Kedlaya proved in \cite {Ked} that, for a smooth variety $\ov{X}$ and an open subset $X$ of $\ov{X}$
such that $Z\coloneqq\ov{X}\setminus X$ is a simple normal crossing divisor,
an overconvergent isocrystal on $\lr{X,\ov{X}}/K$ with unipotent monodromy
can be extended to a convergent log isocrystal on $\lr{\ov{X},\mc{M}}/K$ where $\mc{M}$ is the log structure on $\ov{X}$
associated to $Z$. Shiho extended this result to the case of more general monodromy ($\Sigma$-unipotence) in \cite{Shi}.
These results are $p$-adic analogue of the theory of canonical extension of regular singular integrable connections
on algebraic varieties over $\mathbb{C}$ which is developed by Deligne in \cite{Del}.

We give an overview of Shiho's result.
Let $Z=\bigcup_{i=1}^r Z_i$ be the decomposition of $Z$ to irreducible components.
Let $Z_i'$ be the smooth locus of $Z_i$. Take a sufficiently local lift $\ov{X}\hookrightarrow P$ into a smooth $p$-adic formal scheme $P$ over $\OK$.
Then the tube $\tube{Z_i'}_P$ is isomorphic to a product of a disk $\AA{K}^1\clop{0,1}$ and the rigid space $Q_K$ associated to a locally closed smooth
$p$-adic formal subscheme $Q$ of $P$ which lifts $Z_i'$.
An overconvergent isocrystal $\mc{E}$ on $\lr{X,\ov{X}}/K$ is represented by a $\nabla$-module on some strict neighborhood $W$ of $\tube{X}_P$.
The intersection of $W$ and $\tube{Z_i'}_P$ contains a product of an annulus $\AA{K}^1\clop{\lambda,1}$ and $Q_K$ for some $0<\lambda<1$.
Let $\Sigma_i$ be a subset of $\Z_p$ satisfying some non-Liouville hypothesis.
$\mc{E}$ is $\Sigma_i$-unipotent along $Z_i$ if $\mc{E}$, restricted to the above product, has a filtration whose successive quotients
are pullbacks of $\nabla$-modules on $Q_K$ twisted by $\nabla$-modules $\lr{\O{\AA{K}^1\clop{\lambda,1}},d+\xi\cdot \id}$ on the annulus for some element $\xi$ of $\Sigma_i$.
For $\Sigma=\prod_{i=1}^r \Sigma_i$, if $\mc{E}$ has $\Sigma_i$-unipotent along $Z_i$ for each $i$,
then $\mc{E}$ can be extended to a convergent log isocrystal on $\lr{\ov{X},\mc{M}}/K$ whose exponents are contained in $\Sigma$.

Di Proietto extended this theory  in \cite{diPr2} to the case where $\ov{X}$ is the log special fiber of a proper
semistable variety over $\OK$.

In this paper, we extend the above results to the case when $\ov{X}$ is not necessarily smooth but log smooth with some log structure on $\ov{X}$
and $X$ is the trivial locus of a part of the log structure, under some assumptions:
Let $N\rightarrow \OK$ be a log structure on
$\Spf\, \OK$ which is trivial on $\Spm\,K$. Let $\lr{\ov{X},\mc{M}_0\oplus \mc{M}}$ be a log smooth variety over $\lr{\Spec\,k,N}$
satisfying some conditions. (In particular,
the morphism $N\rightarrow \mc{M}_0\oplus \mc{M}$ must factor through $\mc{M}_0\rightarrow \mc{M}_0\oplus \mc{M}$.)
Let $X$ be the trivial locus of $\mc{M}$.
Let $\mc{S}$ be a subsheaf of $\ov{\mc{M}}^{\mr{gp}}\tens{\Z} \Z_p$ satisfying some non-Liouville hypothesis.
We prove that an overconvergent log isocrystal on $\lr{X,\ov{X},\mc{M}_0\oplus\mc{M}}/\lr{\Spf\,\OK,N}$
with ``$\mc{S}$-unipotent'' monodromy can be extended to a convergent log isocrystal on $\lr{\ov{X},\mc{M}_0\oplus\mc{M}}/\lr{\Spf\,\OK,N}$.

For the definition of $\mc{S}$-unipotence and the proof of the above result, a generalized polyannulus $\A{M}{K}\op{I}$ associated to
a fine monoid $M$ which is a polyannulus in $\Spm\,K\!\angbra{M}$ plays the role of
the polyannulus $\AA{K}^n\op{I}$ defined in \cite{Ked}.

Our result is an extension of the results of Kedlaya \cite{Ked}, Shiho \cite{Shi} and Di Proietto \cite{diPr2},
except that our definition of unipotence of overconvergent log isocrystals is stronger than theirs.
In the situation of \cite{Ked}, \cite{Shi} and \cite{diPr2},
we can check the unipotence at each codimension one component. (See Proposition 4.4.4 in \cite{Ked}.)
However in our situation, we must assume the unipotence at every point. (See Remark \ref{codimension 2}.)
So some arguments in \cite{Ked} or in \cite{Shi} can be omitted in our paper.

We explain the contents of each section.
In Section 1, we prepare the basic notions and results about monoid theory and log structures.
Especially, we define the notion of semi-saturatedness of monoids in \S 1.1, which is a weaker version of saturatedness.
In \S 1.2, we define the notions of log structure and log smoothness on schemes, formal schemes and rigid spaces.

In Section 2, we prove some properties of log $\nabla$-modules on log rigid spaces.
Roughly speaking, the contents of this section is a generalization of those of Section 1 and Section 2 of \cite{Shi}.
In \S 2.1, we define the notion of polyannuli associated to fine weighted monoids and prove their basic properties.
We introduce the notion of log $\nabla$-modules with exponents in $\Sigma$ in \S 2.2
and the notion of $\Sigma$-unipotent log $\nabla$-modules in \S 2.3.
Next, we adapt key propositions in Section 2 of \cite{Shi} to our situation.
Generization proposition is proven in \S 2.4 and transfer theorem is proven in \S 2.5.

In Section 3, we develop the theory of log isocrystals and prove the main result.
In \S 3.1, we generalize the equivalence between the category of convergent log isocrystals and the category of convergent log $\nabla$-modules
(Theorem 6.4.1 in \cite{Ked}) to our situation.
We introduce the notion of log isocrystals with exponents in $\mc{S}$ in \S 3.2
and the notion of $\mc{S}$-unipotent log isocrystals in \S 3.3. But the global definition of $\mc{S}$-unipotence is
postponed to \S 3.5.
In \S 3.4, we adapt overconvergent generization proposition (Proposition 2.7 in \cite{Shi}) to our situation.
In \S 3.5, we define the notion of $\mc{S}$-unipotence of log isocrystals and prove the main result.

\subsection{Conventions}
\begin{enumerate}
\item Throughout this paper, $K$ is a field of characteristic $0$
complete with respect to a non-Archimedean valuation $\abs{\cdot}$.
$k$ is the residue field of $K$ and it is assumed to be a field of characteristic $p>0$.
$\OK$ is the ring of integers of $K$.
$\Gamma^{*}$ is the subset $\sqrt{\abs{K^{\times}}} \cup \brace{0}$ of $\R_{\geq 0}$.
In Section 3, we will assume moreover that the valuation of $K$ is discrete.
\item For a field extesion $K\subseteq K'$, we denote the scalar extension $-\tens{K}K'$ by $\lr{-}_{K'}$.
\item
In this paper, a \textit{monoid} means a commutative monoid with an identity element.
The operation of a monoid is written additively and the identity element of a monoid is denoted by $0$.
The set of units of a monoid $M$ is written as $M^*$ and $\overline{M}$ means $M/M^{*}$.
The Grothendieck group of a monoid $M$ is written as $M^{\mr{gp}}$.
\item
In this paper, for a fine monoid $M$ and $m,m' \in M$, we write $m' \leq m$ if there exists $m'' \in M$ such that $m'+m''=m$
and we write $m'<m$ if $m' \leq m$ and $m' \neq m$. If $M$ is sharp, with this binary relation $\leq$, $\lr{M, \leq}$ is a poset.
\item
  For a finitely generated abelian group $G$, the torsion subgroup of $G$ is written as $G^{\mr{tor}}$. We write $G/G^{\mr{tor}}$ as $G^{\mr{free}}$.
\item
  In this paper, $\N$ denotes the set of non-negative integers (hence $0\in \N$) and it is regarded as a monoid with
  the additive operation $+$.
  The set of positive integers is denoted by $\N_{>0}$.
\item For a complete topological ring $A$ with respect to a norm $\abs{\cdot}$,
  $A\!\angbra{X_1, \ldots, X_n}$ denotes the ring of series $\sum_{\underline{q}\in \N^n} c_{\underline{q}}X_1^{q_1}\ldots X_n^{q_n}$ such that
  the set \[\Set{\ul{q}\in \N^n|\abs{c_{\ul{q}}}>a}\] is finite for any $a\in \R_{>0}$.
  More generally, for a monoid $M$, $A\!\angbra{M}$ denotes the ring of series $\sum_{m \in M} c_{m}t^m$ such that
  the set $\Set{m\in M|\abs{c_{m}}>a}$ is finite for any $a\in \R_{>0}$. Here, $t^m$ is the element associated to $m\in M$ and
  $t^{m_1+m_2}=t^{m_1}\cdot t^{m_2}$ for any $m_1,m_2\in M$.
\item
  In this paper, a point of a rigid space means a closed point of it. A geometric point
  of a rigid space means a geometric point lying over a closed point of it.
\end{enumerate}
\section{Preliminaries}
\subsection{Monoids}

Let $M$ be a monoid. $M$ is called \textit{finitely generated} if there exists a surjective monoid homomorphism $\N^r\twoheadrightarrow M$ for some $r \in \N$.
$M$ is called \textit{integral} if the natural map $M \rightarrow M^{\mr{gp}}$ is injective, i.e., for any $a, b, c \in M$ with $a+c=b+c$, $a=b$.
$M$ is called \textit{fine} if it is finitely generated and integral.
$M$ is called \textit{sharp} if $M^*=\brace{0}$.
An integral monoid $M$ is called \textit{saturated} if for any $a \in M^{\mr{gp}}$ and $ n \in \N_{>0}$
with $na \in M$, $a \in M$.

Moreover, we define a property of monoids which is weaker than saturatedness.

\begin{defn}
  A fine monoid $M$ is called \textit{semi-saturated} if for any $a \in M^{\mr{gp}}$ and $n \in \mathbb{N}_{>0}$ with
  $na\in M$ the exists $m \in \mathbb{N}$ such that $\lr{nm+1}a \in M$.
\end{defn}

A \textit{submonoid} of $M$ is a subset $N$ which is stable under the operation $+$ and contains the identity element $0$.
For a submonoid $N \subseteq M$, we define a monoid $M/N$ by dividing $M$ by the equivalence relation $\cdot \equiv \cdot \pmod N$ which
is defined as follows: $a \equiv b \pmod N$ if there exist $c, d \in N$ such that $a+c=b+d$.
If $M$ is fine, $\lr{M/N}^{\mr{gp}}=M^{\mr{gp}}/N^{\mr{gp}}$ and $M/N$ is the image of $M$ under the canonical map $M \rightarrow M^{\mr{gp}}\rightarrow M^{\mr{gp}}/N^{\mr{gp}}$.
$N^{-1}M$ denotes the submonoid of $M^{\mr{gp}}$ generated by $M$ and $-N=\Set{-n|n\in N}$.
For two submonoids $N_1,N_2\subseteq M$, $N_1+N_2$ denotes the submonoid of $M$ generated by $N_1$ and $N_2$.

A submonoid $F$ of a monoid $M$ is called a \textit{face} if for any $a, b \in M$ with $a+b \in F$, $a,b \in F$.
For any face $F$ of $M$, $M/F$ is sharp. A face $F$ is called a \textit{facet} there exists no face $F'$ such that $F \subsetneq F' \subsetneq M$.

\begin{prop}\label{semi-saturated}
\begin{enumerate}
\item[]
\item If $M$ is saturated, it is semi-saturated.
\item If $M$ is semi-saturated and sharp, $M^{\mr{gp}}$ is torsion-free.
\item If $M$ is semi-saturated, for any submonoid $N$, $M/N$ is semi-saturated.
\item A fine monoid $M$ is semi-saturated if and only if for any face $F$ of $M$, $\lr{M/F}^{\mr{gp}}$ is torsion-free.
\end{enumerate}
\end{prop}

\begin{proof}
  \begin{enumerate}
  \item[]
  \item
    It is clear.
  \item
    Assume that for $a \in M^{\mr{gp}}$ and $n \in \N_{>0}$, $na=0$. Then there exists $m \in \N$ such that $\lr{nm+1}a =a \in M$.
    Then, since $a +\lr{n-1}a=0$ and $M$ is sharp, $a=0$.
  \item
    Take $a \in M^{\mr{gp}}$ and $n \in \N_{>0}$ such that $n\overline{a} \in M/N$ where $\overline{a}$ is the image of $a$ in $\lr{M/N}^{\mr{gp}}$.
    There exists $b \in N$ such that $na+b \in M$, so $n\lr{a+b} \in M$. Then there exists $m \in \N$ such that $\lr{mn+1}\lr{a+b} \in M$.
    This implies $\lr{mn+1}\overline{a} \in M/N$.
  \item
    The former condition implies the latter condition as a consequence of 2 and 3. Conversely, assume that for any face $F$, $\lr{M/F}^{\mr{gp}}$ is torsion-free.
    Take $a \in M^{\mr{gp}}, n \in \N_{>0}$ such that $na \in M$. Let $F\coloneqq\Set{b \in M | \exists m \in \N, b \leq nma}$. This is a face of $M$.
    Let $\overline{a}$ be the image of $a$ in $(M/F)^{\mr{gp}}$. Then $n\overline{a}=0$ by the definition of $F$. Since $\lr{M/F}^{\mr{gp}}$ is torsion-free by assumption, $\overline{a}=0$.
    Then there exists $b \in F$ such that $a+b \in F$ so we can take $m \in \N, c \in M$ such that $b+c=nma$ by the definition of $F$. Therefore
    $\lr{nm+1}a=\lr{a+b}+c \in M$.    
  \end{enumerate}
\end{proof}

\begin{rem}
  One of the simplest examples of semi-saturated but not saturated monoid is $\N\setminus\brace{1}$. Moreover, any submonoid $M \subseteq \N$ such
  that $\N \setminus M$ is finite is semi-saturated.
\end{rem}

The first result of the following lemma is almost the same as Corollary 3.1.5 of \cite{FKato} but the assumption is a bit generalized.
\begin{lem}\label{monoid_section}
Let $f: N \rightarrow M$ be a surjective morphism of fine monoids such that $M^{\mr{gp}}$ is torsion-free.
Put $\tilde{N}\coloneqq \lr{f^{\mr{gp}}}^{-1}\op{M}\subseteq N^{\mr{gp}}$. Then the induced morphism $\tilde{f}: \tilde{N}\rightarrow M$ has a section $M\rightarrow \tilde{N}$
and it gives a splitting $\tilde{N} \cong M \oplus \Ker\op{f^{\mr{gp}}}$. Moreover, if $M$ is sharp, $\lr{\Im\op{s}+N}\cap \Ker\op{f^{\mr{gp}}}=\Ker\op{f}$.
\end{lem}
\begin{proof}
  The results except the last equality is a consequence of Lemma 3.1.3 and Proposition 3.1.4 in \cite{FKato}.
  Assume that $M$ is sharp. Take $a\in M$ and $b\in N$ such that $s\op{a}+b\in \Ker\op{f^{\mr{gp}}}$. Then $a+f\op{b}=0$, so
  $a=0$ and $f\op{b}=0$. Thus $s\op{a}+b=b\in \Ker\op{f}$.
\end{proof}

We also have to define the notion of vertical homomorphisms.
\begin{defn}
  A homomorphism $f:N\rightarrow M$ of fine monoids is \textit{vertical} if
  for any $m\in M$ there exists $n\in N$ such that $m\leq f\op{n}$.
\end{defn}

\subsection{Log structures}
\def\Esp{\mb{Esp}}

In this section, let the category of spaces $\Esp$ be one of the followings:

\begin{enumerate}
\item The category of schemes over $k$.
\item The category of $p$-adic formal schemes over $\OK$.
  \item The category of rigid spaces over $K$.
\end{enumerate}
Let $q$ be $p$ in the first or the second case, $0$ in the third case.

\begin{rem}
  In \cite{GM}, Gillam and Molcho developed a general theory of log structures on categories of spaces
  satisfying some axioms. But they consider only classical topologies, so we cannot use their theory.
  Talpo and Vistoli defined the notion of log structures on more general topoi in \cite{TV},
  but their formalism is too general for our purpose. It seems to be possible to find
  appropriate axioms of spaces to explicate log structure theories on them,
  but we do not do it and consider only the categories which appear in this paper.
\end{rem}

In this section, we call an object of $\Esp$ a \textit{space}.
Let $X$ be a space. Let $X_{\et}$ be the \'etale site over $X$ and let $\O{X}$ be the structure sheaf on $X_{\et}$
(for the proof of the sheaf property in the case of rigid spaces, see Appendix A. of \cite{Diao}.)
A \textit{pre-log structure on $X$} is a pair of a sheaf of monoids $\mc{M}$ on $X_{\et}$ and a morphism of sheaves of monoids
$\alpha: \mc{M}\rightarrow \O{X}$ where $\O{X}$ is the structure sheaf which is regarded as a sheaf of monoids by the multiplication.
A pre-log structure $\lr{\mc{M},\alpha}$ is called a \textit{log structure} if $\alpha^{-1}\op{\O{X}^*}\cong \O{X}^*$ via $\alpha$.
For a pre-log structure $\lr{\mc{M},\alpha}$, we define its associated log structure as the push-out of
$\mc{M}$ and $\O{X}^*$ over $\alpha^{-1}\op{\O{X}^*}$ in the category of sheaves of monoids on $X_{\et}$.

A \textit{log space} is a pair of a space $X$ and a log structure $\lr{\mc{M},\alpha}$ on $X_{\et}$.
$\ov{\mc{M}}$ denotes $\mc{M}/\O{X}^*$.
For a rigid space $X$, $\O{X}^+$ denotes a subsheaf of power-bounded elements of $\O{X}$.
For a scheme or a formal scheme $X$, $\O{X}^+$ denotes $\O{X}$.
A chart of $\lr{\mc{M},\alpha}$ is a pair of a monoid $M$ and a monoid morphism $M\rightarrow \O{X}^+\op{X}$
such that, when it is regarded as a pre-log structure on $X$, its associated log structure is isomorphic to $\mc{M}$.
A chart $M\rightarrow \O{X}^+\op{X}$ is \textit{fine} if $M$ is fine.
$\mc{M}$ is \textit{fine} if there exists a fine chart \'etale locally on $X$.

The category of fine spaces has fiber products.

Let $\ov{x}$ be a geometric point of $X$. A chart $M\rightarrow \O{X}\op{X}$ of $\lr{\mc{M},\alpha}$
is \textit{good at $\ov{x}$} if the canonical homomorphism $M\rightarrow \ov{\mc{M}}_{\ov{x}}$ is an isomorphism.

Let $X$ be a space and $\mc{M}_1$, $\mc{M}_2$ be two log structures. We define the \textit{direct sum} $\mc{M}_1\oplus\mc{M}_2$ of
$\mc{M}_1$ and $\mc{M}_2$ as the push-out $\mc{M}_1\amsum{\O{X}^{*}}\mc{M}_2$ in the category of
sheaves of monoids on $X$. It is the direct sum in the category of
log structures on $X$.

A morphism $f:\lr{X,\mc{M}}\rightarrow \lr{Y,\mc{N}}$ of log spaces is a pair of a morphism $\underline{f}:X\rightarrow Y$ of spaces
and a morphism $f^{\#}:\underline{f}^{-1}\op{\mc{N}}\rightarrow \mc{M}$ of sheaves of monoids such that the following diagram commutes:
\[\begin{tikzcd}
\underline{f}^{-1}\op{\mc{N}}\arrow{r}\arrow{d}& \mc{M}\arrow{d}\\
\underline{f}^{-1}\op{\O{Y}}\arrow{r}& \O{X}.
\end{tikzcd}\]
$f$ is \textit{strict} if the log structure associated to $\ul{f}^{-1}\op{\mc{N}}$ is isomorphic to $\mc{M}$ via $f^{\#}$.

A chart of $f$ is a triplet of a chart $M\rightarrow \O{X}^+\op{X}$ of $\mc{M}$, a chart $N\rightarrow\O{Y}^+\op{Y}$ of $\mc{N}$ and a monoid morphism
$h: N\rightarrow M$ such that the diagram commutes:
\[\begin{tikzcd}
N\arrow{r}\arrow{d}{h}& \ul{f}^{-1}\op{\mc{N}}\arrow{d}{f^{\#}}\\
M\arrow{r}& \mc{M}.
\end{tikzcd}\]
A chart of $f$ is \textit{fine} if $N$ and $M$ are fine.
A chart of $f$ is \textit{good} at a geometric point $\ov{x}$ of $X$ if $M\rightarrow \O{X}^+\op{X}$ is good at $\ov{x}$.

For a monoid $M$, $\AA{M}$ denotes
\begin{itemize}
\item $\A{M}{k}\coloneqq \Spec\,k\!\brack{M}$ with log structure associated to $M\rightarrow k\!\brack{M}$
  if $\Esp$ is the category of schemes over $k$,
\item $\hat{\mathbb{A}}_{M,\OK}\coloneqq \Spf \,\OK\!\angbra{M}$ with log structure associated to $M\rightarrow \OK\!\angbra{M}$
  if $\Esp$ is the category of $p$-adic formal schemes over $\OK$,
\item $\A{M}{K}\!\brack{0,1}\coloneqq \Spm\, K\!\angbra{M}$ with log structure associated to $M\rightarrow K\!\angbra{M}$
  if $\Esp$ is the category of rigid spaces over $K$.
\end{itemize}

For a monoid $M$, giving a chart of a log space 
$(X,{\mathcal{M}})$ of the form $M \to {\mathcal{O}}_X^+(X)$ 
is equivalent to giving a strict morphism 
$(X,{\mathcal{M}}) \to {\mathbb{A}}_M$. 
Thus, when a morphism $X \to {\mathbb{A}}_M$ of spaces is given, 
we can associate to it a log structure ${\mathcal{M}}$ on 
$X$ endowed with a chart the form 
$M \to {\mathcal{O}}_X^+(X)$, by defining 
${\mathcal{M}}$ as the pullback of the log structure of ${\mathbb{A}}_M$ 
and considering the resulting strict morphism 
$(X,{\mathcal{M}}) \to {\mathbb{A}}_M$. 

Next, we fix a chart $N \to {\mathcal{O}}_Y^+(Y)$ of a log space 
$(Y,{\mathcal{N}})$ and a morphism $N \to M$ of monoids. 
Then, giving a chart of a morphism of log spaces 
$f: (X,{\mathcal{M}}) \to (Y, {\mathcal{N}})$ which contains 
the chart $N \to {\mathcal{O}}_Y^+(Y)$ and the morphism $N \to M$
as a part of data is equivalent to giving a strict morphism 
$(X,{\mathcal{M}}) \to {\mathbb{A}}_M \times_{{\mathbb{A}}_N} 
(Y, {\mathcal{N}})$ such that the composition of it and the 
projection ${\mathbb{A}}_M \times_{{\mathbb{A}}_N} 
(Y, {\mathcal{N}}) \to (Y, {\mathcal{N}})$ is equal to $f$. 
Thus, when a morphism $X \to {\mathbb{A}}_M \times_{{\mathbb{A}}_N} Y$
 of spaces is given, we can associate to it 
a log structure ${\mathcal{M}}$ on $X$, a morphism of log spaces 
$f: (X,{\mathcal{M}}) \to (Y,{\mathcal{N}})$ whose underlying morphism of 
spaces is equal to the composition $X \to {\mathbb{A}}_M \times_{{\mathbb{A}}_N} Y \to Y$ and a chart of $f$ which contains 
the chart $N \to {\mathcal{O}}_Y^+(Y)$ and the morphism $N \to M$ above 
as a part of data, by defining 
${\mathcal{M}}$ as the pullback of the log structure of 
${\mathbb{A}}_M \times_{{\mathbb{A}}_N} 
(Y, {\mathcal{N}})$ and considering the resulting strict morphism 
$(X,{\mathcal{M}}) \to {\mathbb{A}}_M \times_{{\mathbb{A}}_N} 
(Y, {\mathcal{N}})$. 

\begin{defn}
  Let $f: \lr{X,\mathcal{M}}\rightarrow \lr{Y,\mathcal{N}}$ be a morphism of fine log spaces.

  Assume there exists a chart $N\rightarrow \O{Y}\op{Y}$ of $\mc{N}$.
  $f$ is \textit{log smooth} (resp. \textit{log \'etale}) if, \'etale locally on $X$ and $Y$, there exists a fine chart $h:N\rightarrow M$ of $f$
  such that $h$ is injective, the torsion part of the cokernel of $h^{\mr{gp}}$ (resp. the cokernel of $h^{\mr{gp}}$) is finite group whose order is prime to $q$
  and the induced morphism $X\rightarrow Y\fibpro{\AA{N}}\AA{M}$ is \'etale.

  In general, $f$ is \textit{log smooth} (resp. \textit{log \'etale}) if there exists an (admissible) \'etale covering $\lr{Y_{\lambda}}_{\lambda\in\Lambda}$
  such that $\mc{N}|_{Y_{\lambda}}$ has a chart and $f|_{f^{-1}\op{Y_{\lambda}}}$ is log smooth (resp. log \'etale).
\end{defn}

\begin{rem}
  This definition is equivalent to the classical one in the case of schemes by Theorem 3.5 of \cite{Kat}.
  In the case of rigid spaces, this definition is the same with the one in \cite{Diao}.
\end{rem}

\begin{prop}
  A morphism $f=\lr{\ul{f},f^{\#}}:\lr{X,\mc{M}}\rightarrow \lr{Y,\mc{N}}$ of fine log spaces is log smooth (resp. log \'etale) and strict,
  then $\ul{f}$ is smooth (resp. \'etale).
\end{prop}
\begin{proof}
  This proof is the same as the proof of Theorem 6.5.6 in \cite{GM}. 
  By Definition, we may assume that there exists a fine chart
  \[\begin{tikzcd}
  N\arrow{r}{b}\arrow{d}{h}& \mc{N}\op{Y}\arrow{d}{f^{\#}}\\
  M\arrow{r}{a}& \mc{M}\op{X}
  \end{tikzcd}\]
  satisfying the condition in the definition of log smoothness (resp. log \'etaleness).

  Take a geometric point $\ov{x}$ of $X$ arbitrarily. Let $\ov{y}\coloneqq \ul{f}\op{\ov{x}}$.
  Let $F\coloneqq a^{-1}\op{\O{X,\ov{x}}^*}$, $G\coloneqq b^{-1}\op{\O{Y,\ov{y}}^*}$.
  $F$ is a face of $M$ and $G$ is a face of $N$. By Theorem 2.1.9 of \cite{Og}, $F$ and $G$ are finitely generated.
  $\AA{F^{-1}M}$  is an open subspace of $\AA{M}$ and $\AA{G^{-1}N}$ is an open subspace of $\AA{N}$.
  Shrinking $X$ to some neighborhood of $\ov{x}$ and $Y$ to some neighborhood of $\ov{y}$,
  we can replace $M$ by $F^{-1}M$ and $N$ by $G^{-1}N$. So we may assume that $F=M^*$ and $G=N^*$.
  Since $f$ is strict, $M/F=N/G$. Thus, $F/G$ is a subgroup of $M^{\mr{gp}}/N^{\mr{gp}}$, so $\AA{F}\rightarrow \AA{G}$ is smooth (resp. \'etale) by the assumption of $h$.
  Also, since $M$ is the pushout of $F\leftarrow G \rightarrow N$, we see that $\AA{M}\rightarrow \AA{N}$ is also smooth (resp. \'etale).
  Since $X\rightarrow Y\fibpro{\AA{N}}\AA{M}$ is \'etale by assumption, $X\rightarrow Y$ is smooth (resp. \'etale).
\end{proof}

\begin{lem}(cf. Lemma 3.1.1 of \cite{FKato}) \label{3.1.1.A}
Let $f: \lr{X,\mathcal{M}}\rightarrow \lr{Y,\mathcal{N}}$ be a morphism of fine log spaces.
Let $\overline{x}$ be a geometric point of $X$ and $\overline{y}\coloneqq f\op{\overline{x}}$.
Put $M\coloneqq \overline{\mathcal{M}}_{\overline{x}}$, $N\coloneqq \ov{\mc{N}}_{\ov{y}}$ and let $h: N \rightarrow M$ be the induced morphism.
Assume $h$ is injective and the cokernel of $h^{\mr{gp}}$ is torsion-free. Take an arbitrary chart $N\rightarrow \mc{N}$ on a neighborhood of $\ov{y}$
which is good at $\ov{y}$.
Then there exists a good chart of $f$ at $\ov{x}$
\[\begin{tikzcd}
      N \arrow[r] \arrow[d, "h"] & \mc{N} \arrow[d] \\
      M \arrow[r] & \mc{M}
\end{tikzcd}\]
on some neighborhood of $\ov{x}$.
\end{lem}
\begin{proof}
  It can be proven as in Lemma 3.1.1 of \cite{FKato}.

  First, we can take a fine chart chart of $f$
  \[\begin{tikzcd}
  N \arrow[r] \arrow[d, "h'"] & \mc{N} \arrow[d] \\
  M' \arrow[r] & \mc{M}
  \end{tikzcd}\]
  on some neighborhood of $\ov{x}$.
  Indeed, we can take a fine chart $M''\rightarrow \mc{M}$ on some neighborhood of $\ov{x}$.
  Let $g:\lr{M''\oplus N}^{\mr{gp}}\rightarrow \mc{M}^{\mr{gp}}_{\ov{x}}\rightarrow M^{\mr{gp}}$ be the natural surjective morphism and let $M'\coloneqq g^{-1}\op{M}$. Then, by the same argument as the proof of Lemma 2.10 of \cite{Kat}, there exists a chart $M'\rightarrow \mc{M}$ on some neighborhood of $\ov{x}$ and we can take $h'$ as the natural map $N\rightarrow M'$.

  By the assumption that $h^{\mr{gp}}$ is injective and its cockernel is torsion-free, $M^{\mr{gp}}\cong N^{\mr{gp}}\oplus \lr{M/N}^{\mr{gp}}$ and we can take a section of the natural surjective morphism $\lr{M'}^{\mr{gp}}\twoheadrightarrow M^{\mr{gp}}\twoheadrightarrow\lr{M/N}^{\mr{gp}}$. So we can take a section $s:M^{\mr{gp}}\rightarrow \lr{M'}^{\mr{gp}}$ of $\lr{M'}^{\mr{gp}}\twoheadrightarrow M^{\mr{gp}}$ such that $s\circ h^{\mr{gp}} = \lr{h'}^{\mr{gp}}$.
  Let $p:M'\rightarrow \mc{M}_{\ov{x}}\rightarrow M$ be the natural morphism and let $\tilde{M}\coloneqq \lr{p^{\mr{gp}}}^{-1}\op{M}$. Then the image of $M$ under $s$ is contained by $\tilde{M}$ and $M \xrightarrow{s} \tilde{M} \rightarrow \mc{M}^{\mr{gp}}_{\ov{x}} \rightarrow M^{\mr{gp}}$ is the inclusion $M\hookrightarrow M^{\mr{gp}}$. Thus, by the same argument as the proof of Lemma 2.10 of \cite{Kat}, there exists a chart $M\rightarrow \mc{M}$ on some neighborhood of $\ov{x}$.

\end{proof}

\begin{lem}(cf. Lemma 3.1.1 of \cite{FKato}) \label{3.1.1.B}
Let $f: \lr{X,\mathcal{M}}\rightarrow \lr{Y,\mathcal{N}}$ be a log smooth morphism of fine log spaces.
Let $\overline{x}$ be a geometric point of $X$ and $\overline{y}\coloneqq f\op{\overline{x}}$.
Put $M\coloneqq \overline{\mathcal{M}}_{\overline{x}}$, $N\coloneqq \ov{\mc{N}}_{\ov{y}}$ and let $h: N \rightarrow M$ be the induced morphism.
Assume $h$ is injective and $M^{\mr{gp}}$ is torsion-free. Take an arbitrary chart $N\rightarrow \mc{N}$ on a neighborhood of $\ov{y}$
which is good at $\ov{y}$. Then there exists a good chart of $f$ at $\ov{x}$
\[\begin{tikzcd}
      N \arrow[r] \arrow[d, "h"] & \mc{N} \arrow[d] \\
      M \arrow[r] & \mc{M}
\end{tikzcd}\]
such that the induced morphism $X\rightarrow Y \times_{\AA{N}}\AA{M}$ is smooth on a neighborhood of $\ov{x}$.
\end{lem}
\begin{proof}
  Take  a chart
  \[\begin{tikzcd}
      N \arrow[r] \arrow[d, "h'"] & \mc{N} \arrow[d] \\
      L \arrow[r] & \mc{M}.
  \end{tikzcd}\]
  on a neighborhood $U$ of $\ov{x}$ such that $h'$ is injective, the order of $\Coker\op{\lr{h'}^{\mr{gp}}}^{\mr{tor}}$ is finite and prime to $q$ and
  $U \rightarrow Y \times_{\AA{N}}\AA{L}$ is \'etale.
  There is a natural surjective morphism $s: L \rightarrow \ov{\mc{M}}_{\ov{x}}\cong M$. Let $\tilde{L}\coloneqq \lr{s^{\mr{gp}}}^{-1}\op{M}$.
  There exists a homomorphism $\tilde{L}\rightarrow \mc{M}_{\ov{x}}$, and we can take a chart $\tilde{L}\rightarrow \mc{M}$
  on some neighborhood of $\ov{x}$ by augment in the proof of Lemma 2.10 of \cite{Kat}. The morphism $\AA{\tilde{L}}\rightarrow \AA{L}$ is log-\'etale and strict on some neighborhood
  of the image of $\ov{x}$, so \'etale there.
  Hence the morphism $X\rightarrow Y\fibpro{\AA{N}}\AA{\tilde{L}}$ is \'etale on the neighborhood of $\ov{x}$.
  By Lemma \ref{monoid_section},
  $\tilde{L}\cong M \oplus \Ker\op{s^{\mr{gp}}}$. Then
  \[\begin{tikzcd}
      N \arrow[rr] \arrow[d, "h"] & & \mc{N} \arrow[d] \\
      M \arrow[r, hook] & \tilde{L} \arrow[r] & \mc{M}
  \end{tikzcd}\]
  is a good chart of $f$ at $\ov{x}$. Since $s^{\mr{gp}}\circ \lr{h'}^{\mr{gp}}=h^{\mr{gp}}$ and $h^{\mr{gp}}$ is injective, $\Ker\op{s^{\mr{gp}}} \rightarrow \Coker\op{\lr{h'}^{\mr{gp}}}$
  is injective, so $\Ker\op{s^{\mr{gp}}}^{\mr{tor}}$ is finite and its order is prime to $q$. Thus the map $\AA{\tilde{L}}\rightarrow \AA{M}$ induced by the inclusion $M\hookrightarrow  M \oplus \Ker\op{s^{\mr{gp}}}\cong\tilde{L}$ is log-smooth and strict on some neighborhood of the image of $\ov{x}$, so smooth there. Therefore $X \rightarrow Y \times_{\AA{N}}\AA{M}$ is smooth on some neighborhood of $\ov{x}$.
\end{proof}

We define the notion of differential modules on log spaces.
\begin{defn}
  Let $f:\lr{X,\mc{M},\alpha}\rightarrow \lr{Y,\mc{N},\beta}$ be a morphism of log spaces.
  The log differential module is defined as follows:
  \[\Omega^{\mr{log},1}_{\lr{X,\mc{M}}/\lr{Y,\mc{N}}}\coloneqq \lr{{\Omega^1_{X/Y}\oplus \lr{O_X\tens{}\mc{M}^{\mr{gp}}}}}/R\]
  where $R$ is the sub-$\O{X}$-module locally generated by
  \begin{itemize}
  \item $\lr{d\alpha\op{a},0}-\lr{0,\alpha\op{a}\tens{}a}$ for each section $a\in \mc{M}$,
  \item $\lr{0,1\tens{}a}$ for each section $a\in \Im\op{f^{\#}}$.
  \end{itemize}
\end{defn}
We denote the element $\lr{0,1\tens{}a}$ by $d\log a$ or $d\log \alpha\op{a}$ for each $a\in \mc{M}^{\mr{gp}}$.

When $Y=\Spec\,k$, $\Spf\,\OK$ or $\Spm\,K$ and $\mc{N}$ is trivial, $\Omega^{\mr{log},1}_{\lr{X,\mc{M}}/\lr{Y,\mc{N}}}$ may be simply denoted by $\Omega^{\mr{log},1}_{\lr{X,\mc{M}}}$. 

\section{Log $\nabla$-modules}
\subsection{Polyannuli}

Let $M$ be a fine monoid.
The rigid analytic space $\A{M}{K}$ is defined as $\lr{\Spec \, K\!\brack{M}}^{\mr{an}}$.
We regard $\A{M}{K}$ as a log rigid space by the log structure associated to $M \rightarrow \Gamma\op{\A{M}{K}, \O{\A{M}{K}}}$.
For $m \in M$, we write $t^m$ as the image of $m$ by map $M \rightarrow \Gamma\op{\A{M}{K}, \O{\A{M}{K}}}$.
If $M$ is sharp, the point of $\A{M}{K}$ defined by $t^m=0$ for all $m\in M\setminus\brace{0}$ is called the \textit{vertex} and denoted by $\mb{0}$.

\begin{defn}\label{weighted_monoid}
  A \textit{weighted monoid} is a monoid $M$ equipped with a monoid homomorphism
  $h: M \rightarrow \N$ (which is called \textit{weighting} of $M$) such that $h^{-1}\op{0}=M^{*}$.
\end{defn}

\begin{rem}
  By Lemma 2.2.2 of \cite{Og}, for any fine monoid $M$, there always exists a monoid homomorphism $h$
  as in Definition \ref{weighted_monoid}.
\end{rem}

Note that for any weighted monoid $M$, $h$ can be extended
to the group homomorphism $M^{\mr{gp}} \rightarrow \Z$ which is also denoted by $h$.

\begin{defn}(Definition 3.1.1 of \cite{Ked})
  A subinterval of $\left[0, +\infty\right)$ is called \textit{aligned} if
any endpoint at which it is closed is contained in $\Gamma^*$.
  A subinterval of $\left[0, +\infty\right)$ is called \textit{quasi-open} if
  it is of the form $\left[0,a\right)$ with $a>0$ or of the form of $\lr{a,b}$ with $0<a<b$.
\end{defn}

\begin{defn}(a generalization of 3.1.2 of \cite{Ked})
  Let $M$ be a fine weighted monoid. For an aligned subinterval $I$ of $\RNN$ we define
  the { \it polyannuli} $\A{M}{K}\op{I}$ as follows:
  \begin{enumerate}
\item If $I=\brack{a,b}$ for $a,b \in \Gamma^{*}$ with $0 \leq a \leq b$,
  \[ \A{M}{K}\op{I}=\Set{x \in \A{M}{K} | \forall m \in M, a^{h\op{m}} \leq \abs{t^m\op{x}} \leq b^{h\op{m}} }. \]
\item In general,
\[ \A{M}{K}\op{I}=\bigcup_{\brack{a,b}\subseteq I}\A{M}{K}\!\brack{a,b}.\]
  \end{enumerate}
  We write $\A{M}{K}\!\brack{a,b}$ or $\A{M}{K}\!\clop{a,b}$ etc. instead of $\A{M}{K}\op{\brack{a,b}}$ or $\A{M}{K}\op{\clop{a,b}}$ etc.
\end{defn}

\begin{rem}
It is enough to check the condition $a^{h\op{m}} \leq \abs{t^m\op{x}} \leq b^{h\op{m}}$ in the above definition of $\A{M}{K}\!\brack{a,b}$
for a set of (finite) generators of $M$ instead of all $m \in M$. So $\A{M}{K}\!\brack{a,b}$ is an affinoid open subset
of $\A{M}{K}$ for any $a\leq b, 0<b$ and $\A{M}{K}\op{I}$ is an admissible open subset for any $I\neq \brack{0,0}$.
\end{rem}

\begin{rem}
  For $I=\brack{0,0}$, if $M$ is sharp, $\A{M}{K}\!\brack{0,0}$ is a single point (the vertex) $\Spm\,K$ with a log structure defined by $\alpha:M\rightarrow K$
  such that $\alpha\op{m}$ is $1$ if $m=0$ and $0$ otherwise. Thus
  \[\Omega^{\mr{log},1}_{\A{M}{K}\!\brack{0,0}}=M^{\mr{gp}}\tens{\Z}K.\]
\end{rem}

\begin{rem}
If $M=\N^n$ and $h: \N^n\rightarrow \N$ is the morphism which takes the sum of all entries, $\A{K}{M}\op{I}$
is equal to $\mathbb{A}_{K}^n\op{I}$ in \cite{Ked}.
\end{rem}

\begin{prop}\label{GammaAMK}
Let $M$ be a fine sharp weighted monoid.
\begin{enumerate}
\item For $0<b \in \Gamma^{*}$,
\[\Gamma\op{\A{M}{K}[0,b],\O{\A{M}{K}[0,b]}}
=\Set{\sum_{m \in M} c_mt^m | c_m \in K, \lim_{h\op{m} \rightarrow \infty}
\abs{c_m}b^{h\op{m}} = 0}.\]
\item For $m \in M^{\mr{gp}}$, we write
\begin{align*}
h^{+}\op{m}&=\min\Set{h\op{m^{+}} | m^{+},m^{-} \in M, m=m^{+}-m^{-}}, \\
h^{-}\op{m}&=\min\Set{h\op{m^{-}} | m^{+},m^{-} \in M, m=m^{+}-m^{-}} =h^{+}\op{m}-h\op{m},\\
\abs{h}\op{m}&=h^{+}\op{m}+h^{-}\op{m}.
\end{align*}
Then for $a,b \in \Gamma^{*}$, $0 < a \leq b$,
\[\Gamma\op{\A{M}{K}[a,b],\O{\A{M}{K}[a,b]}}
=\Set{\sum_{m \in M^{\mr{gp}}} c_mt^m | c_m \in K, \lim_{\abs{h}\op{m} \rightarrow \infty}
\abs{c_m}a^{-h^{-}\op{m}}b^{h^{+}\op{m}} = 0}.\]
\end{enumerate}
\end{prop}
\begin{proof}
  Let $m_1, \cdots, m_n$ be a set of generators of $M$.
  \begin{enumerate}
  \item
    The region $\A{M}{K}\!\brack{0,b}$ is defined by $\frac{\abs{t^{m_i}}}{b^{h\op{m_i}}}\leq 1$ in $\A{M}{K}$. Hence
    \begin{align*}
      &\Gamma\op{\A{M}{K}\!\brack{0,b},\O{\A{M}{K}\!\brack{0,b}}}\\
      &=K\angbra{\frac{t^{m_1}}{b^{h\op{m_1}}}, \ldots, \frac{t^{m_n}}{b^{h\op{m_n}}}}\\
      &=\Set{\sum_{\mathbf{k} \in \N^n} c'_{\mathbf{k}}t^{k_1m_1+\cdots+k_nm_n}
        | c'_{\mathbf{k}} \in K, \lim_{\abs{\mathbf{k}} \rightarrow \infty}b^{k_1h\op{m_1}+\cdots+k_nh\op{m_n}}\abs{c'_{\mathbf{k}}} = 0} \\
      &=\Set{\sum_{m \in M} c_mt^m | c_m \in K, \lim_{h\op{m} \rightarrow \infty}\abs{c_m}b^{h\op{m}} = 0}.
    \end{align*}
    At the last equality, we take $m=k_1m_1+\cdots +k_nm_n$ and $c_m=\sum_{m=k_1m_1+\cdots +k_nm_n}c'_{\mathbf{k}}$.
    Note that when $\abs{\mathbf{k}}\rightarrow \infty$, $h\op{k_1m_1+\cdots +k_nm_n}\rightarrow \infty$ and vice versa.
  \item
    The region $\A{M}{K}\!\brack{a,b}$ is defined by $\frac{\abs{t^{m_i}}}{a^{h\op{m_i}}}\geq 1, \frac{\abs{t^{m_i}}}{b^{h\op{m_i}}}\leq 1$ in $\A{M}{K}$. Hence
    \begin{align*}
      & \Gamma\op{\A{M}{K}\!\brack{a,b},\O{\A{M}{K}\!\brack{a,b}}}\\
      &=K\angbra{\frac{a^{h\op{m_1}}}{t^{m_1}}, \ldots, \frac{a^{h\op{m_n}}}{t^{m_n}},
        \frac{t^{m_1}}{b^{h\op{m_1}}}, \ldots, \frac{t^{m_n}}{b^{h\op{m_n}}}}\\
      &=\Set{\sum_{\mathbf{k},\mathbf{k'} \in \N^n} c'_{\mathbf{k},\mathbf{k'}}t^{\sum_{i=1}^nk_im_i-\sum_{i=1}^nk'_im_i}
        | c'_{\mathbf{k},\mathbf{k'}} \in K, \lim_{\abs{\mathbf{k}}+\abs{\mathbf{k'}} \rightarrow \infty} \frac{a^{\sum_{i=1}^nk'_ih\op{m_i}}}
        {b^{\sum_{i=1}^nk_ih\op{m_i}}}\abs{c'_{\mathbf{k},\mathbf{k'}}} = 0} \\
      &=\Set{\sum_{m^{+},m^{-} \in M} c''_{m^{+},m^{-}}t^{m^{+}-m^{-}}
        | c''_{m^{+},m^{-}} \in K, \lim_{h\op{m^{+}}+h\op{m^{-}} \rightarrow \infty}\abs{c''_{m^{+},m^{-}}}a^{-h\op{m^{-}}}b^{h\op{m^{+}}} = 0} \\
      &=\Set{\sum_{m \in M^{\mr{gp}}} c_mt^m | c_m \in K, \lim_{\abs{h}\op{m} \rightarrow \infty}\abs{c_m}a^{-h^{-}\op{m}}b^{h^{+}\op{m}} = 0}.
    \end{align*}
    At the third equality, we take $m^{+}=\sum_{i=1}^nk_im_i, m^{-}=\sum_{i=1}^nk'_im_i$, and
    $c''_{m^{+},m^{-}}=\sum_{m^{+}=k_1m_1+\cdots +k_nm_n}\sum_{m^{-}=k'_1m_1+\cdots +k'_nm_n}c'_{\mathbf{k},\mathbf{k'}}$.
    At the last equality, we take $c_m=\sum_{m=m^{+}-m^{-}} c''_{m^{+},m^{-}}$. Note that
    $a^{-h\op{m^{-}}}b^{h\op{m^{+}}} \geq a^{-h^{-}\op{m}}b^{h^{+}\op{m}}$ for any $m=m^{+}-m^{-}$.
  \end{enumerate}
\end{proof}

\begin{cor}
  $\A{M}{K}\!\brack{0,1}=\Spm\,K\!\angbra{M}$ and $\A{M}{K}\!\brack{1,1}=\Spm\,K\!\angbra{M^{\mr{gp}}}$
  for any fine sharp weighted monoid $M$.
\end{cor}

\begin{rem}
  Let $M$ be a fine sharp weighted monoid.
  Let $M^{\mr{sat}}\coloneqq\Set{m\in M^{\mr{gp}} | \exists n \in \N_{>0}, nm\in M}$.
  Then $h$ can be extended to a map $M^{\mr{sat}}\rightarrow \N$ which is also denoted by $h$ and $M^{\mr{sat}}$ with $h$ is also a weighted monoid.
  We prove that, if $0 \notin I$, $\A{M}{K}\op{I}=\A{M^{\mr{sat}}}{K}\op{I}$.
  Let $m_1, \ldots, m_g$ be a set of generators of $M^{\mr{sat}}$ ($M^{\mr{sat}}$ is fine by Corollary 2.2.5 of \cite{Og}.) For each $m_i$,
  take $m'_i \in M$ such that $m_i+m'_i\in M$ and $n_i\in \N_{>0}$ such that $n_im_i\in M$.
  Let $s\coloneqq \lr{n_1-1}m'_1+ \cdots +\lr{n_g-1}m'_g$. Since $nm_i+\lr{n_i-1}m'_i\in M$ for any $n\in \N$,
  $m+s \in M$ for all $m\in M^{\mr{sat}}$.
  Let $h^{\mr{sat},+}, h^{\mr{sat},-}$ be the $h^{+}, h^{-}$ in Proposition \ref{GammaAMK}
  with $M$ replaced by $M^{\mr{sat}}$. Then $h^{\mr{sat},+}\op{m}\leq h^{+}\op{m}\leq h^{\mr{sat},+}\op{m}+h\op{s}$ and 
  $h^{\mr{sat},-}\op{m}\leq h^{-}\op{m}\leq h^{\mr{sat},-}\op{m}+h\op{s}$. So by Proposition \ref{GammaAMK}, $\A{M}{K}\!\brack{a,b}=\A{M^{\mr{sat}}}{K}\!\brack{a,b}$ for any $a,b \in \Gamma^{*}$ with $0<a\leq b$.
  By definition, for any aligned interval $I\subseteq\lr{0,\infty}$, $\A{M}{K}\op{I}=\A{M^{\mr{sat}}}{K}\op{I}$.
\end{rem}

\begin{prop}\label{polyannuli_are_basis}
  Let $M$ be a fine sharp weighted monoid.
  For any admissible open subset $U\subseteq \A{M}{K}$ containing the vertex, $\A{M}{K}\!\brack{0,a}\subseteq U$ for some $0<a \in \Gamma^{*}$.
\end{prop}
\begin{proof}
  We may assume that \[U=\Set{x \in \A{M}{K}\!\brack{0,1} | \abs{f_1\op{x}} \leq \abs{g\op{x}},\ldots,\abs{f_n\op{x}} \leq \abs{g\op{x}}}\] for some
  $f_1, \ldots, f_n, g\in K\!\angbra{M}$ such that $\lr{f_1,\ldots,f_n,g}=K\!\angbra{M}$.

  Let $f_i=\sum_{m\in M} c_{i,m}t^m$ for $1 \leq i \leq n$ and $g=\sum_{m\in M} d_mt^m$ which converge on $\A{M}{K}\!\brack{0,1}$. By the assumption that the vertex is contained in $U$,
  $\abs{c_{i,0}} \leq \abs{d_0}$ for all $1\leq i \leq n$.
  If $\abs{d_0}=0$, $\abs{c_{i,0}}=0$ for all $1\leq i \leq n$ and it contradicts the assumption of $f_1,\ldots,f_n,g$. Thus $\abs{d_0}>0$.
  By the convergence of $f_1,\ldots,f_n$ and $g$, we can take $0<a\in \Gamma^{*}$ such that, for all $m \in M\setminus\brace{0}$,
  $\abs{c_{i,m}a^{h\op{m}}} \leq \abs{d_0}$ for all $1\leq i \leq n$ and $\abs{d_ma^{h\op{m}}} < \abs{d_0}$.
  Then for any point $x$ in $\A{M}{K}\!\brack{0,a}$, $\abs{f_i\op{x}}\leq \abs{d_0} = \abs{g\op{x}}$ for all $1\leq i \leq n$,
  so $x\in U$.
\end{proof}

\subsection{Log $\nabla$-modules and exponents}

In this subsection, we define the notion of exponents of log $\nabla$-modules and prove some basic properties.

\begin{lem}\label{lem_for_def_exponents}
  Let $E$ be a finite dimensional vector space over $\Kbar$
  and $\Omega$ a finitely generated abelian group.
  Let $\nabla: E \rightarrow E\tens{\Z} \Omega$ be a $\ov{K}$-linear morphism such that the composition
  \[E\xrightarrow{\nabla}E\tens{\Z}\Omega\xrightarrow{\nabla\tens{}\id}
  E\tens{\Z}\Omega\tens{\Z}\Omega\rightarrow E\tens{\Z}\bigwedge_{\Z}^2\Omega\]
  is zero. Then there exist $\xi_1,\ldots,\xi_n\in \Omega\tens{\Z} \Kbar$
  and a decomposition $E=\bigoplus_{i=1}^n E_i$ such that $E_i\neq 0$, $\nabla\op{E_i}\subseteq E_i\tens{\Z} \Omega$
  and $\nabla-\xi_i\cdot \id$ is nilpotent on $E_i$,
  which means that for any map $\phi: \Omega\rightarrow \Z$ of abelian groups, the composition map
  \[E_i\xrightarrow{\nabla-\xi_i\cdot \id} E_i \tens{\Z} \Omega \xrightarrow{\id\tens{}\phi} E_i\]
  is nilpotent.
\end{lem}
\begin{proof}
  Take some $\Omega^{\mr{free}}\cong \Z^r$ and let $\mr{pr}_i: \Omega \rightarrow \Z$ be the composition of
  the natural map $\Omega\rightarrow \Omega^{\mr{free}}$ and the $i$-th projection.
  Put $\partial_i\coloneqq \lr{\id_E\tens{} \mr{pr}_i} \circ \nabla$.
  By the assumption, $\partial_1,\ldots,\partial_r$ commute with each other.
  So, when we take the Jordan decomposition $E=\bigoplus_{i=1}^{n_1} E_{1,i}$ with respect to $\partial_1$, then
  $\partial_2,\ldots,\partial_r$ act each $E_{1,i}$ and we can take the Jordan decomposition of each $E_{1,i}$ with respect to $\partial_2$
  and so on. If we repeat this process up to $\partial_r$, we reach the decomposition satisfying the above condition.
\end{proof}

\def\wt#1{\widetilde{#1}}
For a fine log rigid space $\lr{X,\mc{M}}$, the \textit{trivial locus of $\mc{M}$} is $\Set{x \in X | \ov{\mc{M}}_{\ov{x}}=0}$.
A log $\nabla$-module on $\lr{X,\mc{M}}$ is
a locally free $\O{X}$-module $E$ equipped with an integrable connection $\nabla:E\rightarrow E\tens{}\Omega^{\mr{log},1}_{\lr{X,\mc{M}}}$.
We define residues and exponents of log $\nabla$-modules on log rigid spaces in the same way as \cite{Og2}.

\begin{defn}\label{exponents}(cf. Definition 2.1.1 of \cite{Og2})
  \def\x{\overline{x}}
  Let $\lr{X,\mathcal{M}}$ be a fine log rigid space.
  Let $\lr{E,\nabla}$ be a log $\nabla$-module on $\lr{X,\mc{M}}$.
  \begin{enumerate}
  \item
    Let $\overline{x}$ be a geometric point lying over $x\in X$. Let $M \coloneqq \ov{\mc{M}}_{\overline{x}}$.
    Let $\Omega^{\mr{log},1}_{\lr{X,\mc{M}},\overline{x}}\rightarrow M^{\mr{gp}}\tens{\Z}\Kbar$ be the canonical surjective map.

    Let $E\op{\overline{x}}\coloneqq E_{\overline{x}}/m_{\x}E_{\x}$ where $m_{\x}$ is the maximal ideal at $\x$.
    The \textit{residue} of $E$ at $\overline{x}$ is the unique $\ov{K}$-linear map $\rho_{\overline{x}}$ such that the following diagram commute:
    \[\begin{tikzcd}
      E \arrow[r, "\nabla"] \arrow[d,two heads] & E \tens{} \Omega^{\mr{log},1}_{\lr{X,\mc{M}}} \arrow[d, two heads]\\
      E\op{\x} \arrow[r, "\rho_{\x}"] & E\op{\x}\tens{\Z} M^{\mr{gp}}
    \end{tikzcd}\]
    
    Since $E\op{\x}$ is finite dimensional over $\Kbar$,
    for some $\xi_1,\ldots \xi_n \in M^{\mr{gp}} \tens {\Z} \Kbar$,
    $\lr{E\op{x}, \rho_{\overline{x}}}$ has a decomposition as in Lemma \ref{lem_for_def_exponents}.
    $\xi_1,\ldots \xi_n$ are called \textit{exponents} of $E$ at $\overline{x}$.
  \item
    Let $\mathcal{S}\subseteq \ov{\mathcal{M}}^{\mr{gp}}\tens{\Z}\Kbar$ be a subsheaf. If for any geometric point
    $\overline{x}$, the exponents
    of $E$ at $\overline{x}$ are contained in
    $\mathcal{S}_{\overline{x}}$, it is said that the all exponents of $E$ are contained in $\mathcal{S}$.
  \end{enumerate}
\end{defn}

\begin{rem}\label{exponents_rem}
  It is clear by linear algebra that, in the situation
  of Definition \ref{exponents}, exponents of
  any subquotient of $E$ is a subset of the exponents of $E$ at any geometric point.
\end{rem}

For a subsheaf $\mathcal{S} \subseteq \ov{\mathcal{M}}^{\mr{gp}} \tens{\Z} \overline{K}$,
we define $\LNM_{\lr{X,\mc{M}},\mathcal{S}}$ as the category
of log $\nabla$-modules on $\lr{X,\mc{M}}$ such that all the exponents are contained in $\mathcal{S}$.
We will prove that $\LNM_{\lr{X,\mc{M}},\mathcal{S}}$ is an abelian category under some assumptions in this subsection.

\def\wtgp#1#2{\ov{\mc{M}_{#1}}^{\mr{gp}}\tens{\Z}#2}

For a rigid space $Y$ and a weighted monoid $M$,
$Y\times \A{M}{K}\op{I}$ naturally has a log structure induced by $M$.
For a subset $\Sigma \subseteq M^{\mr{gp}} \tens{\Z} \overline{K}$,
we denote the image of $\Sigma$ under the natural map $M^{\mr{gp}}\tens{\Z}\Kbar\rightarrow \ov{\mc{M}}^{\mr{gp}}\tens{\Z}\Kbar$,
where $\mc{M}$ is the log structure of $Y\times \A{M}{K}\op{I}$, also by $\Sigma$. 
$\LNM_{Y \times \A{M}{K}\op{I},\Sigma}$ is the category
of log $\nabla$-modules on $Y \times \A{M}{K}\op{I}$ such that all the exponents are contained in $\Sigma$.

For $Y\times \A{M}{K}\op{I}$,
\[\Omega^{\mr{log},1}_{Y\times \A{M}{K}\op{I}}=\lr{\Omega^{1}_{Y} \tens{K} \O{\A{M}{K}\op{I}}} \oplus \lr{\Omega^{\mr{log},1}_{\A{M}{K}\op{I}} \tens{K}\O{Y}}.\]
We write $d=d^Y+d^{\mr{log}}$
corresponding to this decomposition where $d$ is the derivation on $Y\times \A{M}{K}\op{I}$.
For a log $\nabla$-module $\lr{E,\nabla}$ on $Y\times \A{M}{K}\op{I}$, we write
$\nabla=\nabla^Y+\nabla^{\mr{log}}$ corresponding to this decomposition.

We relate residue and exponents in the above definition with those in the sense of \cite{Shi}.

The next lemma and its proof are essentially the same as
Proposition-Definition 1.24 in \cite{Shi2}.

\begin{lem}\label{exponents_for_R}
  \def\y{\overline{y}}
  Let $Y$ be a smooth connected rigid space
  and $M$ a fine sharp weighted monoid.   
  Let $\lr{E,\nabla}$ be a log $\nabla$-module on $Y\times \A{M}{K}\!\brack{0,0}$.
  Let $\rho$ be the composition of $\nabla$ and the projection $E \tens{\O{Y}} \Omega^{\mr{log},1}_{Y \times \A{M}{K}\!\brack{0,0}}\rightarrow E\tens{\Z}M^{\mr{gp}}$. 
  Then there exist $\xi_1,\ldots,\xi_n \in M^{\mr{gp}}\tens{\Z}\Kbar$ and a decomposition
  $E_{\Kbar}=\bigoplus_{i=1}^nE_i$ such that $E_i\neq 0$, $\rho_{\Kbar}\op{E_i}\subseteq E_i\tens{\Z}M^{\mr{gp}}$
  and $\rho_{\Kbar} - \xi_i \cdot \id$ is nilpotent on $E_i$.

  Moreover, for any finite extension $K'$ of $K$ such that $\xi_1,\ldots,\xi_n \in M^{\mr{gp}}\tens{\Z}K'$, $E_{K'}$ admits such a decomposition.

  For any geometric point $\y \in Y$, the exponents of $E$ at $\y$ in the sense of \ref{exponents} are $\xi_1,\ldots,\xi_n$.
  Especially, if $E$ is an object of $\LNM_{Y\times \A{M}{K}\!\brack{0,0},\Sigma}$ then $\xi_1,\ldots,\xi_n \in \Sigma$.
\end{lem}
\begin{proof}
  \def\res{\mathrm{res}}
  Let $e\coloneqq \mathrm{rk}\,E$.
  Take some $\lr{M^{\mr{gp}}}^{\mr{free}}\cong \Z^r$ and let $\res_j\coloneqq \lr{\id\tens{}\mr{pr}_j}\circ \rho: E \rightarrow E$ where $\mr{pr}_j$ is the projection
  $\Z^r\rightarrow \Z$ to the $j$-th factor.

  We decompose $E_{\ov{K}}$ with respect to $\res_1$.
  For $\xi\in\Kbar$, we define a subspace $H\op{\xi}$ of $E_{\Kbar}$ by
  \[\Gamma\op{V,H\op{\xi}}=\Set{\mathbf{v} \in \Gamma\op{V,E_{\Kbar}} | \exists k>0, \lr{\lr{\res_1}_{\Kbar}-\xi\cdot \id}^k\op{\mathbf{v}}=0}.\]
  Then there exists a natural map
  \begin{equation}\label{HxitoE}
    \bigoplus_{\xi\in \Kbar}H\op{\xi}\rightarrow E_{\Kbar}
  \end{equation}
  which is injective by the definition.
  
  Take an affinoid connected subspace $U=\Spm\,R\subseteq Y$ such that $E|_{U}$ is free.
  Let $Q_1\op{t}\in R\!\brack{t}$ be the characteristic polynomial of $\res_1$.
  Let $F\coloneqq \bigwedge^e E|_{U}\tens{K}K\!\brack{t}$ and
  $\psi\coloneqq \bigwedge^e \lr{t-\res_1}: F\rightarrow F$. Then $\psi$ is the multiplication by $Q_1\op{t}$.
  Let $\nabla_F:F\rightarrow F\tens{} \Omega^{1}_{Y}$ be the morphism induced by $\nabla$.
  By the integrability of $\nabla$, the following diagram commutes:
  \[\begin{tikzcd}[column sep=large]
  F \arrow{r}[name=U]{\psi} \arrow{d}[name=R]{\nabla_F}
  & F\arrow{d}[name=R]{\nabla_F}\\
  F \tens{} \Omega^{1}_{Y} \arrow{r}[name=U]{\psi\tens{}\id} &F \tens{} \Omega^{1}_{Y}.
  \end{tikzcd}\]
  Thus the coefficients of $Q_1\op{t}$ are killed by $d_Y$, the derivation on $Y$.

  Take a point $y\in U$ and let $K\op{y}$ be the residue field at $y$ which is a finite extension of $K$.
  Then $\Ker\op{d_y}\cong K\op{y}$ where $d_y$ is the derivation on $\mathcal{O}_{Y,y}$.
  Since $U$ is integral,
  $R\rightarrow \O{Y,y}$ is injective and so we can regard $Q_1\op{t}$ as an element of $K\op{y}\!\brack{t}\subseteq\overline{K}\!\brack{t}$.
  
  Let $\xi_{1,1},\ldots \xi_{1,n_1}\in \Kbar$ be the roots of $Q_1\op{t}$.
  Since $Q_1\op{\lr{\res_1}_{\Kbar}}=0$ on $U$, $\lr{E|_U}_{\Kbar} = \bigoplus_{i=1}^{n_1}H\op{\xi_{1,i}}|_{U}$.
  So the map (\ref{HxitoE}) is surjective on $U$.
  Because $U$ is arbitrary, (\ref{HxitoE}) is surjective on $X$.

  Since any $H\op{\xi}$ is a direct summand of $E_{\Kbar}$, it is locally free and $\mathrm{rk}\, H\op{\xi}$ is constant on $Y$
  because $Y$ is connected. Especially, $H\op{\xi}\neq 0$ if and only if $\xi\in \brace{\xi_{1,1},\ldots \xi_{1,n_1}}$,
  which means that this set is independent of the choice of $U$.
  Let $E_{1,i}\coloneqq H\op{\xi_{1,i}}$,
  then $E_{\Kbar} = \bigoplus_{j=1}^{n_1} E_{1,i}$ on $Y$.
  By the integrability of $\nabla$,
  each $E_{1,i}$ is a log $\nabla$-module on $Y \times \A{M}{K}\!\brack{0,0}$.
  
  We can take such decomposition of each $E_{1,i}$ with respect to $\res_2$, and so on. If we repeat this process to $\res_r$,
  we reach the decomposition satisfying the required condition.

  If $\xi_{1},\ldots \xi_{n}\in M^{\mr{gp}}\tens{\Z}K'$, then we get elements of $K'$ as the roots of the characteristic polynomial
  on each step, so we can decompose $E_{K'}$ instead of $E_{\Kbar}$.

  For any geometric point $\ov{y} \in Y$, let $E_{i}\op{\ov{y}}\coloneqq \lr{E_i}_{\ov{y}}/m_{\ov{y}}\lr{E_i}_{\ov{y}}$ where $m_{\ov{y}}$ is the
  maximal ideal at $\ov{y}$, then we have $E\op{\ov{y}}=\bigoplus_{i=1}^nE_{i}\op{\ov{y}}$
  and $\rho\tens{}k\op{\ov{y}}-\xi_i\cdot \id$ acts on $E_{i}\op{y}$ nilpotently.
  Since any $E_i$ is a direct summand of $E_{\Kbar}$, it is locally free and $\mathrm{rk}\, E_i$ is constant on $Y$,
  so $E_{i}\op{\ov{y}}\neq 0$. This means the last assertion.
\end{proof}

\begin{defn}
  $\rho$ (resp. $\xi_1,\ldots, \xi_n$) in Lemma \ref{exponents_for_R}
  is called the \textit{residue of $\lr{E,\nabla}$} (resp. \textit{exponents of $\lr{E,\nabla}$}.)
\end{defn}

Let $\phi: \lr{M^{\mr{gp}}}^{\mr{free}} \hookrightarrow \Z^r$ be an injective homomorphism.
$\phi$ induces $\phi\tens{}\id_K: M^{\mr{gp}} \tens{\Z} K \hookrightarrow K^r$.
Then we have an inclusion
\begin{equation}\label{omegalog}
  \Omega^{\mr{log},1}_{\A{M}{K}\op{I}}=\mathcal{O}_{\A{M}{K}\op{I}} \tens{\mathbb{Z}} M^{\mr{gp}}
  \subseteq \bigoplus ^r_{i=1} \mathcal{O}_{\A{M}{K}\op{I}} d\log t_i
\end{equation}
induced by $\phi\tens{}\id_K$, where $\{d\log t_i\}$ is the canonical basis of $K^r$.

\begin{defn}\label{defNID1}
Let $S \subseteq \Kbar$ be a subset. $\Sigma \subseteq M^{\mr{gp}}\otimes_{\Z}{\overline K} $ is called \textit{locally ($S$-D)}
if there exists a homomorphism $\phi: \lr{M^{\mr{gp}}}^{\mr{free}} \hookrightarrow \Z^r$ such that for any two elements
$\lr{\alpha_i}, \lr{\beta_i} \in \lr{\phi\tens{}\id_{\Kbar}}\op{\Sigma}$ and for any $1 \leq i \leq r$,
$\alpha_i-\beta_i \in S$.
\end{defn}

\begin{rem}\label{rem_after_NID}
In Definition \ref{defNID1}, we can also assume that $\phi\tens{}\id_K$ is isomorphic, hence the inclusion (\ref{omegalog}) is an equality.
Indeed, let $\mathbf{e}_1, \ldots, \mathbf{e}_r$ be the standard basis of $\Z^r$.
If $\Z \mathbf{e}_i \cap \Im \, \phi= \brace{0}$ for some $i$,
we can take $\mr{pr}_{\hat{i}}\circ \phi$ instead of $\phi$,
where $\mr{pr}_{\hat{i}}: \Z^r \rightarrow \Z^{r-1}$ is the projection
removing the $i$-th entry. Repeating this process, we arrive at the situation where
$\Z \mathbf{e}_i \cap \Im \, \phi \neq \brace{0}$ for any $i$ or $r=0$.
Then the rank of $\lr{M^{\mr{gp}}}^{\mr{free}}$ is equal to
$r$ and $\phi\tens{}\id_K$ is isomorphic.
\end{rem}

\begin{defn}
  For $\alpha \in \Kbar$, we define its type $\mathrm{type}\op{\alpha}$ as the convergent radius of
  the series $\sum_{s \in \N, s \neq \alpha}\lr{s-\alpha}^{-1}t^s$. If $\mathrm{type}\op{\alpha}<1$ or $\mathrm{type}\op{-\alpha}<1$,
  $\alpha$ is called a \textit{$p$-adic Liouville number}. We put:
  \begin{align*}
  \mr{NI}&=\Kbar\setminus\lr{\Z\setminus\brace{0}},\\
  \mr{PT}&=\Set{\alpha \in \Kbar|\textrm{type}\op{\alpha} > 0},\\
  \mr{NL}&=\Set{\alpha \in \Kbar|\text{ $\alpha$ is not a $p$-adic Liouville number}}.
  \end{align*}
  We omit the hyphen in the notation when we write (NI-D), (PT-D) or (NL-D).
\end{defn}

\begin{defn}
  Let $Y$ be a rigid space,
  $M$ a fine sharp weighted monoid, $\Sigma \subseteq M^{\mr{gp}}\tens{\Z}\Kbar$ a subset and $I \subseteq \RNN$ an aligned subinterval.
  Since we have the equalities
  \begin{align*}
    \Omega^{\mr{log},1}_{Y\times \A{M}{K}\op{I}}
    &=\lr{\Omega^{1}_{Y} \tens{K} \O{\A{M}{K}\op{I}}} \oplus \lr{\Omega^{\mr{log},1}_{\A{M}{K}\op{I}} \tens{K}\O{Y}}\\
    &=\lr{\Omega^{1}_{Y} \tens{K} \O{\A{M}{K}\op{I}}} \oplus \lr{M^{\mr{gp}} \tens{\Z}\O{Y\times\A{M}{K}\op{I}}}\\
    &=\lr{\Omega^{1}_{Y}\oplus M^{\mr{gp}}\tens{\Z}\O{Y}} \tens{K} \O{\A{M}{K}\op{I}}\\
    &=\Omega^{\mr{log},1}_{Y\times \A{M}{K}\!\brack{0,0}} \tens{K} \O{\A{M}{K}\op{I}},
  \end{align*}
  there exists a natural inclusion
  \[\Omega^{\mr{log},1}_{Y\times \A{M}{K}\!\brack{0,0}}\hookrightarrow\Omega^{\mr{log},1}_{Y\times \A{M}{K}\op{I}}.\]
  The functor
  $\mathcal{U}_{I}:\LNM_{Y\times \A{M}{K}\!\brack{0,0},\Sigma} \rightarrow \LNM_{Y\times \A{M}{K}\op{I},\Sigma}$ is defined as follows:
  $\mathcal{U}_I$ sends $\lr{E,\nabla} \in \LNM_{Y\times \A{M}{K}\!\brack{0,0},\Sigma}$ to $E \tens{K} \O{\A{M}{K}\op{I}}$
  endowed with the connection
  \[\mathbf{v}\tens{} t^m \mapsto \lr{\nabla\op{\mathbf{v}}+\mathbf{v}d\log t^m}\tens{}t^m\]
  where $\mathbf{v}$ is a section of $\mathcal{O}_Y$ and $m \in M^{\mr{gp}}$. 
\end{defn}

The following lemma is a key proposition of this subsection.

\begin{lem}\label{1.9}(cf. Lemma 1.9 of \cite{Shi})
  Let $Y=\Spm \, R$ be a smooth connected affinoid space,
  $M$ a fine sharp weighted monoid and $\Sigma \subseteq M^{\mr{gp}}\otimes_{\Z}{\overline K}$ be an locally ($\lr{\mr{NI} \cap  \mr{PT}}$-D) subset.
  Let $a \in \lr{0, \infty}\cap \Gamma^{*}$ and $\lr{E,\nabla}$ an object of $\LNM_{Y \times \A{M}{K}\!\brack{0,a}, \Sigma}$
  such that $E|_{Y\times \mb{0}}$ is free. Then there exists $b \in (0,a] \cap \Gamma^{*}$ such that $E|_{Y \times \A{M}{K}\!\brack{0,b}}$
  is in the essential image of $\mathcal{U}_{\brack{0,b}}: \LNM_{Y \times \A{M}{K}\!\brack{0,0},\Sigma} \rightarrow \LNM_{Y\times \A{M}{K}\!\brack{0,b},\Sigma}$.
\end{lem}
\begin{proof}
As the proof of Lemma 1.9 of \cite{Shi}, we can retake $a$ and assume that $E$ is free.
Let $\mathbf{e}= \lr{\mathbf{e}_1, \ldots,\mathbf{e}_n}$ be a basis of $E$.

Take $\phi: M \rightarrow \Z^r$ such that $\phi\tens{}\id_K$ is isomorphic with respect to which the condition of
($\lr{\mr{NI} \cap  \mr{PT}}$-D) is satisfied.
Let $d\log t_1, \ldots, d\log t_r$ be the induced basis of $\Omega^{\mr{log},1}_{\A{M}{K}\!\brack{0,a}}$.
Write $\nabla^{\mr{log}}\op{\mb{v}}=\sum_{i=1}^{r} \partial_{i}\op{\mb{v}} d\log t_{i}$ and let
$A^{i}=\sum_{m \in M} A^{i}_{m} t^m \in \Mat_n\op{R\tens{K}\O{}\lr{\A{K}{M}\!\brack{0,a}}}$
be the matrix representation of $\partial_{i}$ with respect to $\mathbf{e}$.

We construct matrices $B=\sum_{m \in M}B_{m}t^{m}$ and $B'=\sum_{m \in M}B'_{m}t^{m}$ satisfying the following equations for all $1 \leq i \leq r$.

\begin{equation*}
A^iB+\partial_{i}B=BA^{i}_{0},
\end{equation*}
\begin{equation*}
BB'=I.
\end{equation*}

These are equivalent to the equations for all $m \in M$

\begin{equation}\label{**}
A^{i}_{0}B_{m}-B_{m}A^{i}_{0}+m_{i}B_{m}=-\sum_{\substack{m'+m''=m\\m',m'' \neq 0}}A_{m'}B_{m''},
\end{equation}
\begin{equation}\label{**2}
\sum_{m'+m''=m}B_{m'}B'_{m''}=\begin{cases}
I & m=0, \\
O & m\neq 0.
\end{cases}
\end{equation}
Here, $m_{i}$ is the $i$-th entry of $\phi\op{m}$.

By Lemma \ref{exponents_for_R}, we can take the exponents of $E|_{Y\times \A{M}{K}\!\brack{0,0}}$, which we denote by $\xi_{1},\ldots,\xi_{l}\in M^{\mr{gp}}\tens{\Z}\Kbar$.
For $1 \leq i \leq r$, we denote the $i$-th entry of $\lr{\phi\tens{}\id_{\Kbar}}\op{\xi_j}$ by $\xi_{i,j} \in \Kbar$.
By the proof of Lemma \ref{exponents_for_R}, $\brace{\xi_{i,j}}_j$ are the eigenvalues of $A^i_0$.

To construct $B$, we first show the claim which is proven in the proof of Lemma 1.9 in \cite{Shi}.

\begin{claim}\label{BmC}
Let $g : \mathrm{Mat}_n\op{R} \rightarrow \mathrm{Mat}_n\op{R}$ be the map defined as $X \mapsto A^{i}_0 X -X A^i_0$.
Then for any $s \in \Z \setminus \brace{0}$, $g+s\cdot \id$ is invertible and there exist constants
$C\gg0, e\gg0$ independent of $s$ such that
\[\abs{\lr{g+s\cdot \id}^{-1}} \leq C\lr{ \max \brace{\max_{j,j'}\abs{
\xi_{i,j}-\xi_{i,j'}+s}^{-1}, 1}}^e,\]
where $\abs{\cdot}$ is the operator norm.
\end{claim}
\begin{proof}
To prove this, we may enlarge $K$ so that $\brace{\xi_{i,j}}_{i,j} \subseteq K$.

Let $R^n=\bigoplus_{j=1}^{l} E_j$ be the decomposition of $R^n$, where $E_j$ is the $R$-submodule on which $A^i_0-\xi_{i,j}\cdot \id$ is nilpotent.
Then
\[\Mat_n\op{R}=\bigoplus_{j,j'=1}^{l} \Hom_R\op{E_{j'},E_{j}}\]
and $g$ acts on each $\Hom_R\op{E_{j'},E_{j}}$ because $A^{i}_0$ acts on each $E_j$.

Define $g_1$ and $g_2$ as follows:

\begin{equation*}
\begin{array}{cccc}
g_1: & \displaystyle \bigoplus_{j,j'=1}^{l} \Hom_R\op{E_{j'},E_{j}} &
\rightarrow & \displaystyle \bigoplus_{j,j'=1}^{l} \Hom_R\op{E_{j'},E_{j}}, \\
& \rotatebox{90}{$\in$} & & \rotatebox{90}{$\in$} \\
& \lr{X_{j,j'}}_{j,j'} & \mapsto& \lr{\lr{\xi_{i,j}-\xi_{i,j'}}X_{j,j'}}_{j,j'} \\
g_2: & \displaystyle \bigoplus_{j,j'=1}^{l} \Hom_R\op{E_{j'},E_{j}} &
\rightarrow & \displaystyle \bigoplus_{j,j'=1}^{l} \Hom_R\op{E_{j'},E_{j}}. \\
& \rotatebox{90}{$\in$} & & \rotatebox{90}{$\in$} \\
& \lr{X_{j,j'}}_{j,j'} & \mapsto& \lr{\lr{A^i_0-\xi_{i,j}}X_{j,j'}
-X_{j,j'}\lr{A^i_0-\xi_{i,j'}}}_{j,j'}
\end{array}
\end{equation*}
Then $g=g_1+g_2$. $g_1 + s\cdot \id$ is invertible for any $s \in \Z
\setminus \brace{0}$ because $\xi_{i,j}-\xi_{i,j'}+s\neq 0$
for any $j,j'$ by the assumption. $g_2$ is nilpotent because
$A^i_0-\xi_{i,j}$ acts $E_j$ nilpotently. So $g+s\cdot \id$ is invertible.

We can take a constant $c\gg0$ independently of $s$ such that
$\abs{\lr{g_1+s \cdot \id}^{-1}} \leq c \max_{j,j'}\abs{\xi_{i,j}-\xi_{i,j'}+s}^{-1}$,
$\abs{g_2} \leq c$.

Take $e\gg0$ such that $g_2^e=0$, then $\lr{g+s\cdot \id}^{-1}=\sum_{k=0}^{e-1}\lr{-1}^k \lr{g_1+s\cdot \id}^{-k-1}g_2^k$. So
\begin{align*}
\abs{\lr{g+s\cdot \id}^{-1}} &\leq
\max_{0\leq k <e} \abs{g_1+s\cdot \id}^{-k-1}\abs{g_2}^k \\
&\leq c^{2e-1} \lr{ \max \brace{\max_{j,j'}\abs{\xi_{i,j}-\xi_{i,j'}+s}^{-1}, 1}}^e.
\end{align*}
Hence we obtain the required inequality by taking $C=c^{2e-1}$.
\end{proof}

We also prove the following claim.

\begin{claim}\label{c1}
Fix $m^0 \in M, 1 \leq i_0,i_1 \leq r$. Assume:
\begin{itemize}
\item (\ref{**}) and (\ref{**2}) are true when $i=i_0$, $m \leq m^0$.
\item (\ref{**}) is true when $i=i_1$, $m<m^0$.
\item $m^0_{i_0} \neq 0$.
\end{itemize}
Then (\ref{**}) is true when $i=i_1$, $m=m^0$.
\end{claim}
\begin{proof}
Let $J$ be the $R$-submodule of $R\!\brack{\brack{M}}$ generated
by $\Set{ t^{m} | m \nleq m^0 }$. It is an ideal.
The first assumption means:
\[A^{i_0}B+\partial_{i_0}B\equiv BA_0^{i_0} \pmod J,\]
\[BB'\equiv I \pmod J.\]

By the integrability of $\lr{E,\nabla}$
\[\partial_{i_0}\partial_{i_1}\op{\mathbf{e}B}=\partial_{i_1}\partial_{i_0}\op{\mathbf{e}B}.\]

The left hand side is
\begin{align*}
\partial_{i_1}\partial_{i_0}\op{\mathbf{e}B}
&=\partial_{i_1}\op{\mathbf{e}\lr{ A^{i_0}B+\partial_{i_0}B}} \\
&\equiv \partial_{i_1}\op{\mathbf{e}BA^{i_0}_0} \pmod J \\
&=\mathbf{e}\lr{ A^{i_1}B+\partial_{i_1}B} A^{i_0}_0\\
&\equiv \mathbf{e}B\lr{ B' A^{i_1}B+B'\partial_{i_1}B} A^{i_0}_0 \pmod J.
\end{align*}

The right hand side is
\begin{align*}
\partial_{i_0}\partial_{i_1}\op{\mathbf{e}B}
&=\partial_{i_0}\op{ \mathbf{e}\lr{ A^{i_1}B+\partial_{i_1}B}}\\
&\equiv \partial_{i_0}\op{ \mathbf{e}B\lr{ B'A^{i_1}B+B'\partial_{i_1}B}} \pmod J\\
&\equiv \mathbf{e}B\lr{ A_0^{i_0}+\partial_{i_0} }\lr{ B' A^{i_1}B+B'\partial_{i_1}B } \pmod J.
\end{align*}

Let $A'=B'A^{i_1}B+B'\partial_{i_1}B$, then
\begin{equation}\label{A'}
A_0^{i_0}A'+\partial_{i_0} A' \equiv A'A^{i_0}_{0} \pmod J.
\end{equation}

By the comparison of the coefficients of $t^{m^0}$ in (\ref{A'})
\begin{equation}\label{XXX}
  A^{i_0}_{0} A'_{m^0} - A'_{m^0}A^{i_0}_{0}+m^0_{i_0} A'_{m^0}=O.
\end{equation}
Here, $A'_{m^0}$ is the coefficient of $t^{m^0}$ in $A'$.

By (\ref{XXX}), Claim \ref{BmC} and the assumption that $m^0_{i_0} \neq 0$, $A'_{m^0}=0$. By this and the second assumption,
$A'\equiv A^{i_1}_{0} \pmod J$, so
$ A^{i_1}B+\partial_{i_1}B \equiv B A^{i_1}_{0} \pmod J$.
Therefore (\ref{**}) is true when $i=i_1$, $m=m^0$.

\end{proof}

Let $h$ be the weighting of $M$.
We construct $B_m$ and $B'_m$ inductively with respect to $h\op{m}$.

First, we let $B_0=B'_0=I$.

Take $m \in M, m \neq 0$. Take $1 \leq i_0 \leq r$ such that $m_{i_0} \neq 0$
(this is always possible because $M \cap \lr{M^{\mr{gp}}}^{\mr{tor}}=\brace{0}$ by the assumption of $M$.)
We assume that $B_{m'}, B'_{m'}$ are constructed for all $m'$ with $h\op{m'}<h\op{m}$, especially for all $m'<m$.

By Claim \ref{BmC} we can take $B_m$
satisfying (\ref{**}) for $i=i_0$, and can easily take $B'_m$ satisfying (\ref{**2}).
Then by Claim \ref{c1}, (\ref{**}) is true for any $i$.

Next, we prove that $B=\sum_{m\in M}B_mt^m$ converges on $\A{M}{K}\!\brack{0,b}$ for some $b>0$.

For all $m \in M$ and $i$ such that $m_i \neq 0$, let

\[z_{m,i}=\max \left\lbrace \max_{j,j'}\left\vert \xi_{i,j}-\xi_{i,j'}-m_i \right\vert^{-1}
, 1 \right\rbrace.\]

For $m\in M$, let

\[Z_{m}=\max_{m^1 < m^2 < \cdots < m^s=m} \prod_{k=1}^s \min_{i, m^{k}_i \neq 0}
z_{m^k,i}.\]

Take $C\gg0, e\gg0$ which satisfy the inequality of Claim \ref{BmC}.
We can also assume that $\left\vert A^{i}_m\right\vert \leq Ca^{-h\op{m}}$ for all $i$ and $m \in M$
because all $A^{i}$'s converge in $\A{M}{K}\!\brack{0,a}$.

We prove inductively that

\begin{equation}\label{B_m}
\abs{B_m}\leq Z_m^eC^{2h\op{m}}a^{-h\op{m}}.
\end{equation}

The case $m=0$ is clear.

Assume this is true for all $m'<m$. By Claim \ref{BmC} and (\ref{**})
\begin{align*}
\vert B_m \vert & \leq 
\min_{i,m_i\neq 0} C z_{m,i}^e \left \vert \sum_{\substack{m'+m''=m \\ m',m''\neq 0}}
A^{i}_{m'}B_{m''}\right \vert \\
& \leq 
\min_{i,m_i\neq 0} C z_{m,i}^e \max_{\substack{m'+m''=m \\ m',m''\neq 0}}
\left \vert A^{i}_{m'} \right \vert \left \vert B_{m''}\right \vert \\
& \leq 
\min_{i,m_i\neq 0} C z_{m,i}^e \max_{\substack{m'+m''=m \\ m',m''\neq 0}}
C a^{-h\op{m'}} \cdot Z_{m''}^e C^{2h\op{m''}}a^{-h\op{m''}} \\
& = 
\max_{\substack{m'+m''=m \\ m',m''\neq 0}}\lr{ \min_{i,m_i\neq 0} z_{m,i} Z_{m''}} ^e C^{2h\op{m''}+2}a^{-h\op{m}} \\
& \leq 
Z_{m}^e C^{2h\op{m}}a^{-h\op{m}}. \\
\end{align*}

We must estimate $Z_{m}$.

\begin{claim}\label{c2}
There exist constants $C_1, C_2 > 0$
such that for any $m^1 < m^2 < \cdots < m^s$
and $x>0$ with $\abs{\Set{1\leq k\leq s| \log \min_{i,m^k_i\neq 0} z_{m^k,i} \geq x }} > rn^2$,
$h\op{m^s}-h\op{m^1} \geq C_1e^{C_2x}$.
\end{claim}
\begin{proof}
By the pigeonhole principle, there exist $m^{k} < m^{k'}$ and $i, j, j'$ such that
$m^{k}_i, m^{k'}_i \neq 0$ and that
$\abs{\xi_{i,j}-\xi_{i,j'}-m^{k}_i }, \abs{\xi_{i,j}-\xi_{i,j'}-m^{k'}_i } \leq e^{-x}$.
There exists $m' \in M$ such that $m^{k} +m' = m^{k'}$. Because $m'_i=m^{k'}_i-m^{k}_i$,
$\left\vert m'_i \right\vert \leq e^{-x}$.

On the other hand, there exists a constant $c_1>0$ which depends only on the valuation of $K$ such that
for any $a \in \Z$, $\log \abs{a} \geq - c_1 \log \abs{a}_{\R}$,
where $\abs{\cdot}_{\R}$ is the usual absolute value of $\R$. There also exists a constant $c_2>0$ which depends only on $h$ such that
for any $1 \leq i \leq r$ and $m\in M$, $\left \vert m_i\right \vert_{\mathbb{R}} \leq c_2h\op{m}$.

Then
\[x \leq -\log \abs{m'_i} \leq c_1 \log \abs{m'_i}_{\mathbb{R}} \leq  c_1\lr{ \log c_2+\log h\op{m'}}.\]

Therefore, if we put $C_1\coloneqq c_2^{-1}, C_2\coloneqq c_1^{-1}$, then $h\op{m^s}-h\op{m^1} \geq h\op{m'} \geq C_1e^{C_2x}$.
\end{proof}

By Claim \ref{c2}, for $m^1<m^2 < \cdots < m^s=m$ and $x>0$,

\[\abs{ \Set{1\leq k \leq s| \log \min_{i,m^k_i\neq 0} z_{m^k,i} \geq x }}\leq\lr{ rn^2+1}\lr{ 1 + \frac{h\op{m}}{C_1} e^{-C_2x} }. \]
Therefore, if we put
\[v_m\coloneqq\max_{m'< m}\log\op{\min_{i,m'_i\neq 0} z_{m',i}},\]
$\log Z_m$ is estimated as follows:
\begin{align*}
\log Z_m &= \max_{m^1< \cdots <  m^s=m}\sum_{j=1}^s \log\op{\min_{i,m^j_i\neq 0} z_{m^j,i} } \\
&\leq  \int_{0}^{v_m}\lr{ rn^2+1 }\lr{ 1+\frac{h\op{m}}{C_1}e^{-C_2x} } dx \\
&\leq\lr{ rn^2+1 }\lr{ v_m +\int_{0}^{\infty} \frac{h\op{m}}{C_1}e^{-C_2x} dx }\\
&=\lr{ rn^2+1 }\lr{ v_m + \frac{h\op{m}}{C_1C_2}}.
\end{align*}

By the assumption about exponents, $\xi_{i,j}-\xi_{i,j'}$ has positive type for any
$i,j,j'$. Hence
\[ \log \op{\min_{i,m'_i\neq 0} z_{m',i}} \in O\op{\max_{i} \abs{m_i}_{\mathbb{R}}} \subseteq O\op{h\op{m}},\]
and so $v_m \in O\op{h\op{m}}$.

Therefore $\log Z_m \in O\op{h\op{m}}$. Thus, by (\ref{B_m}), $B$ converges on $Y \times \A{M}{K}\!\brack{0,b}$
for some $ 0 < b \ll a$.

Let $\mathbf{f}=\mathbf{e}B$ as a basis of $E|_{Y \times \A{M}{K}\!\brack{0,b}}$.
Then the matrix representation of $\partial_i$
with respect to $\mathbf{f}$ is $A^i_{0} \in \Mat_{n}\op{R}$.
Thus the $R$-module generated by $\mathbf{f}$ is stable under $\partial_i$.

Finally, we show that the $R$-module generated by $\mathbf{f}$ is stable under $\nabla^Y$.
Let $\omega_1,\ldots, \omega_{r'}$ be a local basis of $\Omega^{1}_{Y}$. We write
$\nabla^Y=\sum_{i'=1}^{r'} \partial'_{i'}\omega_{i'}$ and let $D^{i'}=\sum_{m\in M} D^{i'}_m t^m$
be the matrix representation of $\partial'_i$ with respect to $\mathbf{f}$.

Let $m\in M\setminus\brace{0}$. Take $1 \leq i \leq r$ such that $m_i\neq 0$. By the integrability of $\nabla$,
$\partial_i\partial'_{i'}\op{\mathbf{f}}=\partial'_{i'}\partial_{i}\op{\mathbf{f}}$.
So we have $A^{i}_0D^{i'}+\partial^{i}D^{i'}=D^{i'}A^{i}_0$,
which means
\[A^{i}_0D^{i'}_m+m_iD^{i'}_m-D^{i'}_mA^{i}_0=O.\]
By Claim \ref{BmC}, $D^{i'}_m=O$.

Therefore $D^{i'}=D^{i'}_0$ and $\nabla^Y$ is also
defined on the $R$-module generated by $\mathbf{f}$.

So the $R$-module generated by $\mathbf{f}$ equipped with $\nabla_Y$
and $\partial_i$ for $1\leq i \leq r$ defines an object of $\LNM_{Y \times \A{M}{K}\!\brack{0,0},\Sigma}$, whose image under $\mathcal{U}_{\brack{0,b}}$ is equal to $E_{Y \times \A{M}{K}\!\brack{0,b}}$.
\end{proof}

Next, we define a global version of the notion of ($S$-D).

\begin{defn}\label{global S-D}
  Let $S$ be a subset of $\Kbar$ and  $M$ a fine monoid. A subset $\Sigma \subseteq M^{\mr{gp}} \tens{\Z} \Kbar$
  is called \textit{($S$-D)} if, for all face $F$ of $M$, the image of $\Sigma$ in $\lr{M/F}^{\mr{gp}} \tens{\Z} \Kbar$
  is locally ($S$-D).
\end{defn}

\begin{lem}
  Let $M$ be a semi-saturated monoid. For any facet $F$ of $M$, $\lr{M/F}^{\mr{gp}}\cong\Z$.
\end{lem}
\begin{proof}
  By the definition of facet, The dimension of $M/F$ in the sense of (5.4) of \cite{Kat2} is $1$.
  By Proposition (5.5) of \cite{Kat2}, the rank of $\lr{M/F}^{\mr{gp}}$ is $1$.
  Since $M$ is semi-saturated, $\lr{M/F}^{\mr{gp}}$ is torsion-free. So $\lr{M/F}^{\mr{gp}}\cong\Z$. 
\end{proof}

\begin{prop}\label{propSD}
  In the situation of Definition \ref{global S-D},
  assume that $M$ is semi-saturated and that, for all $s \in \Kbar$ and $n \in \N_{>0}$ such that $ns \in S $, $s \in S$.
  Then $\Sigma$ is ($S$-D) if and only if for all facet $F$ of $M$ and for any two elements $\alpha, \beta$
  in the image of $\Sigma$ in $\lr{M/F}^{\mr{gp}} \tens{\Z} \Kbar = \Kbar$, $\alpha-\beta$ is contained in $S$.
\end{prop}
\begin{proof}
  If $\Sigma$ is ($S$-D), for all facet $F$ of $M$, there exists a map $\phi: \lr{M/F}^{\mr{gp}}=\Z \hookrightarrow \Z$
  satisfying condition in Definition \ref{defNID1}. By retaking $-\phi$ as $\phi$ if necessary, we may assume that $\phi\op{1}>0$. By the assumption about $S$,
  we may assume that $\phi=\id$. Then $\Sigma$ satisfies the latter condition.

  Conversely, assume that $\Sigma$ satisfies the latter condition. For any face $F$ of $M$, the map
  \[\lr{M/F}^{\mr{gp}} \rightarrow \prod_{F\subseteq F'} \lr{M/F'}^{\mr{gp}},\]
  where $F'$ runs through all facets of $M$ containing $F$, satisfies the condition of being locally ($S$-D).
\end{proof}

\begin{rem}\label{MtoN}
  By Proposition \ref{propSD} and Remark \ref{rem_after_NID},
  if $S$ satisfies the condition of \ref{propSD}, $M$ is semi-saturated and $\Sigma$ is ($S$-D),
  we can take $\phi$ in Definition \ref{defNID1} such that
  $\phi\op{M}\subseteq \N^r$ and $\phi\tens{}\id_{K}$ is isomorphic.
\end{rem}

\begin{rem}
  NI, PT and NL satisfy the assumption of Proposition \ref{propSD} because for any $\alpha \in \Kbar, n \in \N_{>0}$
  \[\mathrm{type}\op{n\alpha} \leq \mathrm{type}\op{\alpha}^{\frac{1}{n}}.\]
\end{rem}

\def\M{\mathcal{M}}
\def\x{\overline{x}}
\begin{defn}\label{def_of_global_S-D}
  Let $S$ be a subset of $\Kbar$,
  $\lr{X,\mc{M}}$ a fine log rigid space
  and $\mathcal{S} \subseteq \mathcal{\overline M}^{\mr{gp}} \otimes_{\Z} \overline{K}$ a subsheaf.
  $\mathcal{S}$ is called \textit{($S$-D)} if,
  for any geometric point $\overline{x}$ of $X$, it admits a good chart $M \rightarrow \M$
  on some neighborhood of $\overline{x}$
  and an ($S$-D) subset $\Sigma \subseteq M^{\mr{gp}}\tens{\Z} \Kbar$
  such that $\mathcal{S}$ is the image of $\Sigma$
  under the induced map $M^{\mr{gp}}\tens{\Z}\Kbar \rightarrow \overline{\M}^{\mr{gp}}\tens{\Z}\Kbar$.
\end{defn}

\begin{rem}\label{S-D_remark}
  The definition of ($S$-D)-ness is independent of the choice of a good chart.
  Indeed, let $\alpha: M \rightarrow \mc{M}$ and $\beta: M \rightarrow \mc{M}$ be two good charts at $\ov{x}$ on some neighborhood of $\ov{x}$.
  $M\xrightarrow{\alpha} \mc{M}\rightarrow \ov{\mc{M}}\rightarrow \ov{\mc{M}}_{\ov{x}}=M$ and $M\xrightarrow{\beta}\mc{M}\rightarrow \ov{\mc{M}}\rightarrow \ov{\mc{M}}_{\ov{x}}=M$ coincide, so $M\xrightarrow{\alpha} \mc{M}\rightarrow \ov{\mc{M}}$ and $M\xrightarrow{\beta} \mc{M}\rightarrow \ov{\mc{M}}$ coincide after shrinking the neighborhood of $\ov{x}$. Then the two maps $M^{\mr{gp}}\tens{\Z}\Kbar \rightarrow \overline{\M}^{\mr{gp}}\tens{\Z}\Kbar$ induced by $\alpha$ and $\beta$ coincide.
  
\end{rem}

\begin{rem}\label{YM is S-D}
  Let $S$ be a subset of $\Kbar$.
  For any rigid space $Y$, a semi-saturated weighted monoid $M$, an aligned subinterval $I \subseteq \RNN$ and
  an ($S$-D) subset $\Sigma \subseteq M^{\mr{gp}}\tens{\Z} \Kbar$, the subsheaf on $Y \times \A{M}{K}\op{I}$ defined by $\Sigma$
  is ($S$-D). Indeed, for any geometric point $\lr{\overline{y},\overline{x}}$ of $Y \times \A{M}{K}\op{I}$, put
  $F\coloneqq \Set{m\in M| t^m\op{\overline{x}}\neq 0}$. Then $M/F=\overline{\mathcal{M}}_{\lr{\overline{y},\overline{x}}}$
  where $\mathcal{M}$ is the log structure induced by $M$. $\A{M}{K}\op{I}\cap\A{F^{-1}M}{K}$ is an open subset of $\A{M}{K}\op{I}$
  containing the image of $\ov{x}$.
  By Lemma \ref{monoid_section}, the map $F^{-1}M\rightarrow M/F$ has
  a section $M/F\rightarrow F^{-1}M$, which gives a good chart at $\lr{\overline{y},\overline{x}}$ on $Y \times \lr{\A{M}{K}\op{I}\cap\A{F^{-1}M}{K}}$.
  Let $\Sigma_F$ be the image of $\Sigma$ under the map $M^{\mr{gp}}\tens{\Z}\Kbar\twoheadrightarrow \lr{M/F}^{\mr{gp}}\tens{\Z}\Kbar$,
  which is also ($S$-D) by definition. The subsheaf defined by $\Sigma$ is the image of $\Sigma_F$ under the map
  $\lr{M/F}^{\mr{gp}}\tens{\Z}\Kbar\rightarrow \ov{\mc{M}}^{\mr{gp}}\tens{\Z}\Kbar$ on $Y\times \lr{\A{M}{K}\op{I}\cap\A{F^{-1}M}{K}}$.
\end{rem}

\begin{prop}\label{LNM_is_abelian}(cf. Proposition 1.11 in \cite{Shi})
  Let $\lr{X,\mc{M}}$ be a log smooth rigid space over $K$
  such that $\overline{\mathcal{M}}^{\mr{gp}}_{\overline{x}}$ is torsion-free at any geometric point $\overline{x}$ of $X$.
  Let $\mathcal{S} \subseteq \overline{\mathcal{M}}^{\mr{gp}} \otimes_{\Z} \overline{K}$ be an ($\lr{\mr{NI}\cap\mr{PT}}$-D) subsheaf.
  Then $\LNM_{\lr{X,\mc{M}},\mathcal{S}}$ is an abelian category.
\end{prop}

\begin{proof}
The proof is almost the same as the proof of Proposition 1.11 in \cite{Shi}.

Let $f : \lr{E, \nabla_E} \rightarrow \lr{F, \nabla_F}$ be a morphism in $\LNM_{\lr{X,\mc{M}},\mc{S}}$.
It is enough to show that $\Ker\op{f}$ and $\Coker\op{f}$ are locally free because of Remark \ref{exponents_rem}.

By Lemma 3.2.14 of \cite{Ked}, it is enough to show very locally, i.e.,
it is enough to show local freeness for some open neighborhood of any point.
Take $x \in X$ and a geometric point $\overline{x}$ lying over $x$.
Let $M\coloneqq\overline{\mathcal{M}}_{\overline{x}}$.
By Lemma \ref{3.1.1.B}, there exist an \'etale neighborhood $U \rightarrow X$ of $\overline{x}$ and a good chart
$M \rightarrow \M|_U$ at $\ov{x}$ such that the induced morphism
$U \rightarrow \A{M}{K}\!\brack{0,1}$ is smooth.
We may assume that this morphism has a decomposition $U \rightarrow \mathbb{A}^n_K \times \A{M}{K}\!\brack{0,1} \rightarrow \A{M}{K}\!\brack{0,1}$ for some $n$
where the first morphism is \'etale and the second morphism is the second projection.
By the assumption about $\mathcal{S}$ and Remark \ref{S-D_remark}, we may also assume that there exists an ($\lr{\mr{NI}\cap\mr{PT}}$-D) subset $\Sigma \subseteq M^{\mr{gp}}\tens{\Z}\Kbar$ such that $\mathcal{S}|_{U}$ is induced by it.

We may replace $K$ by its finite extension and assume that
$\overline{x}$ maps to a $K$-rational point of $U$.
By Lemma 3.1.5 of \cite{JP}, on some open neighborhood of the image of $\overline{x}$ in $U$,
both $U \rightarrow X$ and $U \rightarrow \mathbb{A}^n_K \times\A{M}{K}\!\brack{0,1}$
are open immersions. Take some $h: M\rightarrow \N$ and regarded $M$ as a weighted monoid by $h$,
then there exists an open neighborhood of $\overline{x}$ which is
isomorphic to $Y \times \A{M}{K}\!\brack{0,a}$ for some smooth affinoid space $Y$
and $a>0$ by Proposition \ref{polyannuli_are_basis}.
So we may assume that $\lr{X,\mc{M}}=Y\times \A{M}{K}\!\brack{0,a}$ and that $\mathcal{S}=\Sigma$.

We may assume that $E$ and $F$ are free.
Then, by Lemma \ref{1.9}, we may assume that $E$ and $F$ are in the image of $\mathcal{U}_{\brack{0,a}}$.

Take some morphism $M \rightarrow\N^r$ as in Remark \ref{MtoN}
and denote by $\partial_i$ the composition of $\nabla^{\mr{log}}$ and the map induced by the projection
$M^{\mr{gp}}\rightarrow\Z^r\rightarrow \Z$ to the $i$-th factor.
Let $E'$, $F'$ be the objects of $\LNM_{Y\times \A{M}{K}\!\brack{0,0},\Sigma}$ such that
$\mathcal{U}_{\brack{0,a}}\op{E'}=E, \mathcal{U}_{\brack{0,a}}\op{F'}=F$.
After we replace $K$ by its finite extension if necessary, we can take a basis $\mathbf{e}=\op{\mathbf{e}_1, \ldots, \mathbf{e}_l}$
of $E'$ such that
$\mathbf{e}_1$ is an eigenvector of all $\partial_i$'s.

It is enough to show that $f\op{\mathbf{e}_1} \in F'$
because of the induction of the ranks of $E$ and $F$. (See the proof of Proposition 3.2.14 in \cite{Ked}
or the proof of Proposition 1.11 in \cite{Shi}.)

There exists $\lr{\xi_1, \ldots, \xi_r} \in \Sigma$
such that $\partial_i\op{\mathbf{e}_1}=\xi_i \mathbf{e}_1$.
Write $f\op{\mathbf{e}_1} = \sum_{m \in M} \mathbf{a}_m t^m$ where $\mathbf{a}_m \in F'$.

We will prove that for any $1 \leq i \leq r$, if $m_i \neq 0$ then $\mathbf{a}_m=0$.

We may assume that $i=1$. After we replace $K$ by its finite extension if necessary, we can take a basis $\mathbf{f}=\lr{\mathbf{f}_1, \ldots, \mathbf{f}_n}$ of $F'$
such that the matrix representation of $\partial_1$ with respect to it is the Jordan standard form.
Then there exist $\lr{\eta_1, \ldots, \eta_r} \in \Sigma$ and
$0=n_0 < n_1 < \cdots < n_k=n$ such that
\[\partial_1\mathbf{f}_{n_i}=\eta_i\mathbf{f}_{n_i} \ \lr{1\leq i \leq k},\]
\[\partial_1\mathbf{f}_{j}=\eta_i\mathbf{f}_{j} +\mathbf{f}_{j+1} \ \lr{1\leq i < k,\  n_i < j <n_{i+1}}.\]

Let $\mathbf{a}_m=\sum_{i=1}^r a_{m,i}\mathbf{f}_i$.
Then by the commutativity of $f$ and the connections
\[\lr{f\otimes \mathrm{id}_{\Omega^{\mr{log},1}_{\lr{X,\mc{M}}}}}
\lr{\nabla_E\op{\mathbf{e}}}=\nabla_F\op{f\op{\mathbf{e}_1}}\]
and comparing the coefficients of $d\log t_1$ on both hand sides, we obtain the equalities:
\[\lr{\xi_1-\eta_{i+1}}a_{m,n_i+1}=m_1 a_{m,n_i+1} \ \lr{0 \leq i < k},\]
\[\lr{\xi_1-\eta_{i+1}}a_{m,j}=m_1 a_{m,j} +a_{m,j-1} \ \lr{0 \leq i < k,\  m_i+2 \leq j \leq m_{i+1}}.\]

By the assumption for $\Sigma$, $\xi_1-\eta_{i+1}$ is not a non-zero integer.
So if $m_1 \neq 0$, $a_{m,n_i+1}=0$ because of the first equality.
Then by the second equality, we see that $a_{m,j}=0$ for all $j$.
\end{proof}

\begin{cor}\label{cor_Y}
For any smooth rigid space $Y$, any fine sharp semi-saturated monoid $M$,
a subset $\Sigma \subseteq M^{\mr{gp}}\otimes_{\Z} \overline{K}$ which is ($\lr{\mr{NI}\cap\mr{NL}}$-D)
and any aligned subinterval $I \subseteq \clop{0,\infty}$,
$\LNM_{Y \times \A{M}{K}\op{I},\Sigma}$ is an abelian category.
\end{cor}
\begin{proof}
  By Remark \ref{YM is S-D} and Proposition \ref{LNM_is_abelian}.
\end{proof}

\subsection{Unipotence}
In this subsection, we define the notion of $\Sigma$-unipotence of log $\nabla$-modules and prove some basic properties.

\begin{defn}(cf. 1.3 of \cite{Shi})
Let $Y$ be a smooth rigid space, $M$ a fine sharp semi-saturated monoid, $I\subseteq \clop{0,\infty}$ an aligned subinterval, and $\Sigma \subseteq M^{\mr{gp}} \otimes _{\mathbb{Z}} \Kbar$ be an ($\lr{\mr{NI}\cap\mr{NL}}$-D) subset. Let $\pi_1$ and $\pi_2$ be the first and second projections of $Y\times \A{M}{K}\op{I}$.
\begin{enumerate}
\item An object $\lr{E, \nabla}$ of $\LNM_{Y\times \A{M}{K}\op{I},\Sigma}$
is \textit{$\Sigma$-constant} if
it is isomorphic to $\pi_1^*E_Y \tens{} \pi_2^*C_{\xi}$ for some $\nabla$-module $E_Y$ on $Y$
and $\xi \in \Sigma \cap \lr{M^{\mr{gp}}\tens{\Z} K}$, 
where $C_{\xi}=\lr{\mathcal{O}_{A_{M,K}\op{I}}, d + \xi\cdot \id}$.
Note that $M^{\mr{gp}}\tens{\Z} K$ is injected to $\Omega^{\mr{log},1}_{\A{M}{K}\op{I}}$ and $d+\xi\cdot \id$ is the map $f \in \O{\A{M}{K}\op{I}} \mapsto df + f\xi$. 
\item An object $\lr{E, \nabla}$ of $\LNM_{Y\times \A{M}{K}\op{I},\Sigma}$
is \textit{$\Sigma$-unipotent} if
it admits a filtration by subobjects $0=E_0 \subseteq E_1 \subseteq \cdots \subseteq E_n=E$ in $\LNM_{Y\times \A{M}{K}\op{I},\Sigma}$
such that any successive quotient is $\Sigma$-constant.
We denote by $\ULNM_{Y \times \A{M}{K}\op{I},\Sigma}$ the full subcategory of $\LNM_{Y\times \A{M}{K}\op{I},\Sigma}$
whose objects are $\Sigma$-unipotent objects.
\item An object $\lr{E, \nabla}$ of $\LNM_{Y\times \A{M}{K}\op{I},\Sigma}$ is
\textit{potentially $\Sigma$-constant} (resp. \textit{potentially $\Sigma$-unipotent})
if there exists a finite extension $K'$ of $K$ such that
$E_{K'}$ is $\Sigma$-constant (resp. $\Sigma$-unipotent).
We denote by $\ULNM'_{Y\times \A{M}{K}\op{I},\Sigma}$ the full subcategory of $\LNM_{Y \times \A{M}{K}\op{I},\Sigma}$
whose objects are potentially $\Sigma$-unipotent objects.
\end{enumerate}
\end{defn}

\begin{rem}\label{rem_for_00}
  In the case $I=\brack{0,0}$, an object of $\LNM_{Y\times \A{M}{K}\!\brack{0,0},\Sigma}$ is $\Sigma$-unipotent
  if and only if its all exponents are contained in $\lr{M^{\mr{gp}}\tens{\Z} K}$, because of Lemma \ref{exponents_for_R}.
  In particular,
  $\LNM_{Y \times \A{M}{K}\!\brack{0,0},\Sigma}=\ULNM'_{Y\times \A{M}{K}\!\brack{0,0},\Sigma}$ (cf. Remark 1.13 of \cite{Shi}.)
\end{rem}

\begin{rem}\label{unipotence/Z}
  In the case $0\notin I$, for any $\xi_1,\xi_2\in M^{\mr{gp}}\tens{\Z}K$ such that $\xi_1-\xi_2\in M^{\mr{gp}}$,
  $C_{\xi_1}\cong C_{\xi_2}$ on $\A{M}{K}\op{I}$. Indeed, let $m\coloneqq\xi_1-\xi_2\in M^{\mr{gp}}$, then
  the morphism $f\in \O{\A{M}{K}\op{I}} \mapsto ft^m \in \O{\A{M}{K}\op{I}}$ gives an isomorphism $C_{\xi_1}\cong C_{\xi_2}$ as log $\nabla$-modules.
  So $\Sigma$-constantness and $\Sigma$-unipotence of log $\nabla$-modules are only dependent on the image of $\Sigma$ in $M^{\mr{gp}}\tens{\Z}\lr{\Kbar/\Z}$.
\end{rem}

The next proposition and its proof are analogues of
3.3.2 of \cite{Ked} and 1.14 of \cite{Shi}.

\def\UI{\mathcal{U}_I}
\begin{prop}\label{UI_Ext_Isom}
  Let $Y$ be a smooth rigid space,
  $M$ a fine sharp semi-saturated  monoid,
  $I \subseteq \clop{ 0, \infty }$ a quasi-open subinterval of positive length,
  and $\Sigma \subseteq M^{\mr{gp}}\tens{\Z}\Kbar$ an ($\lr{\mr{NI}\cap\mr{NL}}$-D) subset.
  Then the morphisms
  \[\Ext^i\op{E,E'} \rightarrow \Ext^i\op{\UI\op{E},\UI\op{E'}}\]
  induced by $\UI: \LNM_{Y \times \A{M}{K}\!\brack{0,0},\Sigma} \rightarrow \LNM_{Y\times \A{M}{K}\op{I},\Sigma}$
  are isomorphic for all $E,E' \in \LNM_{Y \times \A{M}{K}\!\brack{0,0},\Sigma}$ and $i\geq 0$.
\end{prop}
\begin{proof}
Let $F=E^{\vee}\tens{}E'$. By 3.3.1 of \cite{Ked}, we have isomorphisms
\[\Ext^i\op{E,E'} = H^i\op{Y, F \tens{}\Omega^{\mr{log}, \bullet}_{Y \times \A{M}{K}\!\brack{0,0}}},\]
\[\Ext^i\op{\UI\op{E},\UI\op{E'}} = H^i\op{Y \times \A{M}{K}\op{I},
\UI\op{F}\tens{}\Omega^{\mr{log}, \bullet}_{Y \times \A{M}{K}\op{I}}}\]
for any $i \geq 0$. 

Note that
\[\UI\op{F}\tens{}\Omega^{\mr{log},1}_{Y \times \A{M}{K}\op{I}}=F \tens{}\Omega^{\mr{log},1}_{Y \times \A{M}{K}\!\brack{0,0}} \tens{K} \O{\A{M}{K}\op{I}}.\]

We may assume that $Y$ is affinoid because, in general case, the both cohomologies can be calculated by making an admissible affinoid cover of $Y$ and using spectral sequences
provided by corresponding \v{C}ech complexes.

By considering Katz-Oda type spectral sequence for
$Y\times \A{M}{K}\!\brack{0,0}\rightarrow Y\rightarrow \Spm\,K$, we have the spectral sequence
\begin{align*}
  E_2^{p,q}&=H^p\op{\Gamma\op{Y,\Omega^{\bullet}_{Y}}\tens{}H^q\op{Y,F\tens{}\Omega^{\mr{log}, \bullet}_{Y\times \A{M}{K}\!\brack{0,0}/Y}}}\\
  &\Rightarrow H^{p+q}\op{Y, F \tens{}\Omega^{\mr{log}, \bullet}_{Y \times \A{M}{K}\!\brack{0,0}}}.
\end{align*}

On the other hand, 
by considering Katz-Oda type spectral sequence for
$Y\times \A{M}{K}\op{I}\rightarrow Y\rightarrow \Spm\,K$, we have the spectral sequence
\begin{align*}
  E_2^{p,q}&=H^p\op{\Gamma\op{Y,\Omega^{\bullet}_{Y}}\tens{}H^q\op{Y\times \A{M}{K}\op{I},\mc{U}_I\op{F}\tens{}\Omega^{\mr{log}, \bullet}_{Y\times \A{M}{K}\op{I}/Y}}}\\
  &\Rightarrow H^{p+q}\op{Y\times \A{M}{K}\op{I},\mc{U}_I\op{F} \tens{}\Omega^{\mr{log}, \bullet}_{Y \times \A{M}{K}\op{I}}}.
\end{align*}

So we have to prove that the map
\begin{equation}\label{HiHi}
H^i\op{Y, F \tens{}\Omega^{\mr{log}, \bullet}_{Y\times \A{M}{K}\!\brack{0,0}/Y}}
\rightarrow H^i\op{Y \times \A{M}{K}\op{I}, \UI\op{F}\tens{}\Omega^{\mr{log}, \bullet}_{Y\times\A{M}{K}\op{I}/Y}}
\end{equation}
induced by the natural map of complexes
\[g_1: \Gamma\op{Y, F \tens{}\Omega^{\mr{log}, \bullet}_{Y \times \A{M}{K}\!\brack{0,0}/Y}}
\rightarrow \Gamma\op{Y \times \A{M}{K}\op{I}, \UI\op{F}\tens{}\Omega^{\mr{log}, \bullet}_{Y \times \A{M}{K}\op{I}/Y}}\]
is isomorphic.

We may enlarge $K$ and assume that $E$ and $E'$ are $\Sigma$-unipotent.
Applying Five Lemma to (\ref{HiHi}), we may also assume that $E$ and $E'$ are $\Sigma$-constant.
We can write $E=E_Y\tens{}C_{\xi}, E'=E'_Y \tens{}C_{\xi'}$ for some $\xi, \xi' \in \Sigma$
and then $F=E_Y^{\vee}\tens{}E'_Y\tens{}C_{\xi'-\xi}$.
Let
\[g_2: \Gamma\op{Y \times \A{M}{K}\op{I}, \UI\op{F}\tens{}\Omega^{\mr{log}, \bullet}_{Y \times \A{M}{K}\op{I}/Y}}
\rightarrow  \Gamma\op{Y, F \tens{}\Omega^{\mr{log}, \bullet}_{Y \times \A{M}{K}\!\brack{0,0}/Y}}\]
be the map induced by the `taking constant coefficient' map
\begin{equation*}
\begin{array}{ccc}
\Gamma\op{\A{M}{K}\op{I},\O{\A{M}{K}\op{I}}} &\rightarrow & K. \\
\rotatebox{90}{$\in$} & & \rotatebox{90}{$\in$}\\
\dsum_{m \in M^{\mr{gp}}} a_m t^m & \mapsto & a_0
\end{array}
\end{equation*}
Then $g_2 \circ g_1$ is the identity. We show that $g_1\circ g_2$ is homotopic
to the identity by constructing a homotopy $\varphi$.

Take $\phi: M \hookrightarrow \N^r$ as in Remark \ref{MtoN}, which induces
\[\Omega^{\mr{log},1}_{\A{M}{K}\op{I}}= \bigoplus_{i=1}^r \O{\A{M}{K}\op{I}}d\log t_i.\]

For $m\in M^{\mr{gp}}\setminus \brace{0}$ and $1 \leq i_1 < \cdots < i_k \leq r$,
let $l=l\op{m}$ be the least integer such that $m_l \neq 0$,
and we define
\begin{equation*}
\varphi\op{t^m \bigwedge_{j=1}^k d\log t_{i_j}}= \begin{cases}
\ \displaystyle \frac{\lr{-1}^{s-1}}{m_l+\xi'_l-\xi_l} t^m \bigwedge_{\substack{1 \leq j \leq k \\ j \neq s}} d \log t_{i_j}
& \text{if }i_s = l \text{ for some $s$,}\\
0 & \text{otherwise.}
\end{cases}
\end{equation*}
Note that by the assumption for $\Sigma$, $m_l+\xi'_l-\xi_l \neq 0$.
We also define $\varphi\op{t^0\bigwedge_{j=1}^k d\log t_{i_j}}=0$.

We will show that $\varphi$ can be extended naturally to the map of graded modules,
\[\varphi: \Gamma\op{Y \times \A{M}{K}\op{I}, \UI\op{F}\tens{}\Omega^{\mr{log}, \bullet}_{Y \times\A{M}{K}\op{I}}}
\rightarrow \Gamma\op{Y \times \A{M}{K}\op{I}, \UI\op{F}\tens{}\Omega^{\mr{log}, \bullet-1}_{Y \times \A{M}{K}\op{I}}}.\]

To show this, it is enough to show that, if $\sum_{m \in M^{\mr{gp}}} c_m t^m$ converges on $Y \times \A{M}{L}\op{I}$,
\begin{equation}\label{seriesconverge}
\sum_{\substack{m \in M^{\mr{gp}}\setminus\brace{0}
\\ l\op{m}=l}} \frac{c_m}{m_l +\xi'_l-\xi_l} t^m
\end{equation}
also converges on $Y \times \A{M}{L}\op{I}$ for any $l$.

Let $\brack{a,b} \subseteq I$ be an aligned closed subinterval. By the assumption,
$\abs{c_m} a^{-h^{-}\op{m}}b^{h^{+}\op{m}} \rightarrow 0$ when $\abs{h}\op{m}\rightarrow \infty$.
Since $\xi'_l-\xi_l$ is not a $p$-adic Liouville number, for any $s<1$,
\[\frac{s^{\abs{m_l}_{\R}}}{\abs{m_l+\xi'_l-\xi_l}}\rightarrow 0\] when $\abs{m_l}_{\R}\rightarrow \infty$.
Take a constant $C$ such that $\abs{m_l}_{\R} \leq C\abs{h}\op{m}$ for any $m\in M^{\mr{gp}}$.
Then
\[\abs{\frac{c_m}{m_l +\xi'_l-\xi_l}} a^{-h^{-}\op{m}}b^{h^{+}\op{m}}s^{C\abs{h}\op{m}} \rightarrow 0.\]
So the series (\ref{seriesconverge}) converges on $\A{M}{K}\!\brack{a/s^C, bs^C}$.
Since $I$ is quasi-open, (\ref{seriesconverge}) converges on $\A{M}{K}\op{I}$.

To finish the proof, We show that $\varphi$ gives a homotopy between $g_1 \circ g_2$
and the identity.

Take $m \neq 0$ and $1 \leq i_1 < \cdots < i_k \leq r$,
and let $l=l\op{m}$. If $i_s = l$ for some $s$,
\begin{align*}
& \nabla_{F} \varphi\op{t^m \bigwedge_{j=1}^k d\log t_{i_j}}
= \nabla_{F}\op{\frac{\lr{-1}^{s-1}}{m_l+\xi'_l-\xi_l} t^m \bigwedge_{j \neq s} d \log t_{i_j}}\\
& = t^m \bigwedge_{j=1}^k d\log t_{i_j} + \lr{-1}^{s-1} \sum_{i' \notin \brace{i_j}}
\frac{m_{i'}+\xi'_{i'}-\xi_{i'}}{m_{l}+\xi'_{l}-\xi_{l}}t^m d\log t_{i'} \wedge 
\bigwedge_{j \neq s} d \log t_{i_j},
\end{align*}
\begin{align*}
\varphi\nabla_{F}\op{t^m \bigwedge_{j=1}^k d\log t_{i_j}}
&= \varphi\op{\sum_{i' \notin \brace{i_j}}
\lr{m_{i'}+\xi'_{i'}-\xi_{i'}}t^m d\log t_{i'} \wedge \bigwedge_{j=1}^k d\log t_{i_j}}\\
&= \lr{-1}^{s}\sum_{i' \notin \brace{i_j}}
\frac{m_{i'}+\xi'_{i'}-\xi_{i'}}{m_{l}+\xi'_{l}-\xi_{l}}t^m d\log t_{i'} \wedge 
\bigwedge_{j \neq s} d \log t_{i_j}.\\
\end{align*}
Otherwise, $\varphi\op{t^m \bigwedge_{j=1}^k d\log t_{i_j}}=0$ and
\begin{align*}
\varphi\nabla_{F}\op{t^m \bigwedge_{j=1}^k d\log t_{i_j}}
&=\varphi\op{\sum_{i' \notin \brace{i_j}}
\lr{m_{i'}+\xi'_{i'}-\xi_{i'}}t^m d\log t_{i'} \wedge \bigwedge_{j=1}^k d\log t_{i_j}}\\
&=t^m \bigwedge_{j=1}^k d\log t_{i_j}.
\end{align*}
Also, for $m=0$, $\nabla_{F}\varphi\op{t^0 \bigwedge_{j=1}^k d\log t_{i_j}}=\varphi\nabla_{F}\op{t^0 \bigwedge_{j=1}^k d\log t_{i_j}}=0$.

So $\nabla_{F}\varphi+\varphi\nabla_{F}=\id-g_1\circ g_2$.
\end{proof}

\def\AZ{Y \times \A{M}{K}\!\brack{0,0}}
\def\AI{Y \times \A{M}{K}\op{I}}

The following corollary is the analogue of Corollary 1.15 and Corollary 1.16 in \cite{Shi}.

\begin{cor}\label{UULNM}
  Let $Y$, $M$, $\Sigma$ be as in \ref{UI_Ext_Isom} and $I \subseteq \clop{0,\infty}$ an aligned subinterval
  of positive length. Then the functors
  \begin{align*}
    \mathcal{U}_{I}: \ULNM_{Y \times \A{M}{K}\!\brack{0,0},\Sigma} &\rightarrow \ULNM_{Y \times \A{M}{K}\op{I},\Sigma}\\
    \mathcal{U}_{I}: \ULNM'_{Y \times \A{M}{K}\!\brack{0,0},\Sigma} &\rightarrow \ULNM'_{Y \times \A{M}{K}\op{I},\Sigma}\\
  \end{align*}
  are fully-faithful. Moreover, if $I$ is quasi-open, these functors are equivalences of categories.
\end{cor}
\begin{proof}
  The proof is the same as the proofs of Corollary 1.15 and Corollary 1.16 in \cite{Shi}.
\end{proof}

To prove Proposition \ref{1.17} below, we show the following lemma, which is a generalization of Lemma $3.2.19$ in \cite{Ked}.

\begin{lem}\label{Lem3.2.19}
  Let $Y$ be an affinoid space, $M$ a fine sharp monoid such that $M^{\mr{gp}}$ is torsion-free,
  $I \subseteq \RNN$ an aligned closed subinterval and $E$
  a log $\nabla$-module on $\AI$. Take $\mathbf{e}_1, \ldots, \mathbf{e}_n \in \Gamma\op{\AI,E}$
  which are linearly independent over $\O{Y}$.
  Assume that there exists $\xi \in M^{\mr{gp}}\tens{\Z} \Kbar$ such that $\nabla^{\mr{log}}_{E}\op{\mathbf{e}_i}=\xi\mathbf{e}_i$ for each $i$.
  Then $\mathbf{e}_1, \ldots, \mathbf{e}_n$ are linearly independent over $\O{\AI}$.
\end{lem}
\begin{proof}
  The case when $I=\brack{0,0}$ is clear. Assume $I \neq [0,0]$. Making $I$ smaller if necessary, we may assume that $0 \notin I$.
  Take some $M^{\mr{gp}} \cong \Z^r$. Write $\brace{d\log t_i}$ as the induced basis of $\Omega^{\mr{log},1}_{\A{M}{K}\op{I}}$
  and $d^{\mr{log}}=\sum_{i=1}^r \partial_i d\log t_i$, $\xi=\sum_{i=1}^{r} \xi_i d\log t_i$. For $l>0$, define an operator $D_l$ on $\O{\AI}$ by
  \[D_l=\prod_{i=1}^r\prod_{\substack{-l \leq j \leq l \\ j \neq 0}} \frac{1}{j}\lr{\partial_i-j}\]
  (cf. the proof of 3.4.1 in \cite{Ked}).

  We prove that for all $f=\sum_{m\in M^{\mr{gp}}} f_mt^m \in \O{\AI}$, $D_l\op{f}$ converges to $f_0$ when $l\rightarrow \infty$.
  Let \[M^{\mr{gp}}_l\coloneqq \Set{m\in M^{\mr{gp}}\setminus\brace{0}| m_i< -l\text{ or }m_i=0\text{ or }m_i>l\text{ for any $1 \leq i\leq r$}}.\]
  Then
  \[D_l\op{f}-f_0=\sum_{m\in M^{\mr{gp}}_l} \prod_{i=1}^r\prod_{\substack{-l \leq j \leq l \\ j \neq 0}} \frac{m_i-j}{j}f_mt^m.\]
  For any $m$, $\prod_{i=1}^r\prod_{\substack{-l \leq j \leq l \\ j \neq 0}} \frac{m_i-j}{j}$ is an integer because it is a product of
  binomial coefficients. Since $\min_{m \in M^{\mr{gp}}_l}\abs{h}\op{m}\rightarrow \infty$ when $l \rightarrow \infty$,
  $\sup_{m\in M^{\mr{gp}}_l} \abs{f_mt^m} \rightarrow 0$ on $\A{M}{K}\op{I}$. Therefore $\lr{D_l\op{f}-f_0}\rightarrow 0$ on $\A{M}{K}\op{I}$.
  
  Assume $\sum_{k=1}^n c_k\mathbf{e}_k=0$ for some $c_1, \ldots, c_n \in \O{\AI}$. Applying $\nabla_{E}$, we have
  \[\sum_{k=1}^n \lr{\partial_i\op{c_k}\mathbf{e}_k+ c_k\xi_i\mathbf{e}_k}=0,\]
  so $\sum_{k=1}^n \partial_i\op{c_k}\mathbf{e}_k=0$ for each $i$.
  Therefore $\sum_{k=1}^n D_{l}\op{c_k}\mathbf{e}_k=0$ for any $l$.

  If $c_k\neq 0$ for some $k$, we may assume $c_{k,0} \neq 0$, where $c_{k,0}$ is the constant term of $c_k$,
  by multiplying $t^m$ for some $m \in M^{\mr{gp}}$ to all $c_k$.
  $D_{l}\op{c_k}$ converges to $c_{k,0}$, so $\sum_{k=1}^n c_{k,0} \mathbf{e}_k=0$. This contradicts the
  assumption that $\mathbf{e}_1, \ldots, \mathbf{e}_n$ are linearly independent over $\O{Y}$.
\end{proof}

\begin{prop}\label{1.17}(cf. Proposition 1.17 of \cite{Shi})
  Let $Y$, $M$, $\Sigma$ be as in Proposition \ref{UI_Ext_Isom} and $I \subseteq \clop{0,\infty}$ be a quasi-open subinterval or
  a closed aligned subinterval of positive length. Any subquotient of an object of $\ULNM_{Y\times\A{M}{K}\op{I},\Sigma}$
  (resp. $\ULNM'_{Y\times\A{M}{K}\op{I},\Sigma}$) in $\LNM_{Y\times\A{M}{K}\op{I},\Sigma}$ is an object of $\ULNM_{Y\times\A{M}{K}\op{I},\Sigma}$ (resp. $\ULNM'_{Y\times\A{M}{K}\op{I},\Sigma}$.) In particular, $\ULNM_{Y\times\A{M}{K}\op{I},\Sigma}$
  and $\ULNM'_{Y\times\A{M}{K}\op{I},\Sigma}$ are abelian subcategories of $\LNM_{Y\times\A{M}{K}\op{I},\Sigma}$.
\end{prop}
\begin{proof}
  It is enough to show the case of $\ULNM_{Y\times\A{M}{K}\op{I},\Sigma}$.
  
  As in the proof of 1.17 in \cite{Shi}, we may assume that $I$ is a closed aligned subinterval and it is enough to show that
  any quotient of a $\Sigma$-constant module is also $\Sigma$-constant
  or that any subobject of a $\Sigma$-constant module is also $\Sigma$-constant.

  First, we prove this in the case when $Y=\Spm\,K$.
  Let $E$ be a $\Sigma$-constant log $\nabla$-module on $Y\times\A{M}{K}\op{I}$ and $f:E \twoheadrightarrow E' $ be a surjection in $\LNM_{Y\times\A{M}{K}\op{I},\Sigma}$.
  There exists an element $\xi \in \Sigma$ such that $E=F\tens{K}\O{\A{M}{K\op{I}}}$, where $F=\brace{\mathbf{e}\in E| \nabla^{\mr{log}}_E\op{\mathbf{e}}=\xi \mathbf{e}}$.
  Let $F'$ be the image of $F$ under $f$. The map $F'\tens{K}\O{\A{M}{K}\op{I}} \rightarrow E'$ is surjective by the surjectivity of $f$ and
  injective by Lemma \ref{Lem3.2.19}. So $E'$ is $\Sigma$-constant.

  Next we prove it for general $Y$.
  Let $E$ be a $\Sigma$-constant log $\nabla$-module on $Y\times\A{M}{K}\op{I}$ and $f:E' \hookrightarrow E $ be an injection in $\LNM_{Y\times\A{M}{K}\op{I},\Sigma}$.
  
  It is enough to show that $E'$ is $\Sigma$-constant after replacing $K$ by its finite Galois extension $K'$.
  Indeed, if $E'_{K'}$ is $\Sigma$-constant, there exists a $\Sigma$-constant object $E'_{K',0}$ of $\LNM_{Y\times \A{M}{K'}\!\brack{0,0}}$
  such that $\mc{U}_I\op{E'_{K',0}}=E'_{K'}$. For each $\sigma \in \mathrm{Gal}\op{K'/K}$, there exists a natural isomorphism
  $\iota_{\sigma}:\sigma^*E'_{K'}\xrightarrow{\sim} E'_{K'}$. By Proposition \ref{UULNM}, there exists a unique morphism
  $\iota_{\sigma,0}:\sigma^*E'_{K',0}\xrightarrow{\sim}E'_{K',0}$ such that $\mc{U}_I\op{\iota_{\sigma,0}}=\iota_{\sigma}$.
  So there exists a $\Sigma$-constant object $E'_{0}$ of $\LNM_{Y\times \A{M}{K}\!\brack{0,0}}$ such that $\lr{E'_0}_{K'}=E'_{K',0}$ by Galois descent.
  Since $\mc{U}_I\op{E'_0}=E'$, $E'$ is $\Sigma$-constant.
  
  Thus we may assume that $\A{M}{K}\op{I}$ has some $K$-rational point $x:\Spm\,K\hookrightarrow \A{M}{K}\op{I}$.
  We denote the morphism $Y\hookrightarrow Y\times \A{M}{K}\op{I}$ induced by $x$ also by $x$.
  Let $\pi_1$ and $\pi_2$ be the first and second projections of $Y\times \A{M}{K}\op{I}$.
  Take $\xi\in \Sigma$ such that $E\tens{}\pi_2^*C_{-\xi}$ is $\brace{0}$-constant.
  Put \[E''\coloneqq \pi_1^*\Im\op{x^*\lr{E'\tens{}\pi_2^*C_{-\xi}}\rightarrow x^*\lr{E\tens{}\pi_2^*C_{-\xi}}},\]
  then
  $E''$ is a $\brace{0}$-constant log $\nabla$-module and
  $E''\tens{}\pi_2^*C_{\xi}$ is a subobject of
  $E$. We have to show $E''\tens{}\pi_2^*C_{\xi}=E'$. Since it is enough to show it very locally on $Y$,
  we may assume that $Y=\AA{K}^n\!\brack{0,p^{-m}}$. Let $L$ be the completion of fraction field of $\O{Y}$ with respect to the spectrum norm.
  It is enough to show $E''\tens{}\pi_2^*C_{\xi}=E'$ after scalar extension to $L$. The equality $E''_L\tens{}\pi_2^*C_{\xi}=E'_L$ is true on $\A{M}{L}\op{I}$ since we have already proved that $E'_{L}$ is $\Sigma$-constant.
  So $E'$ is $\Sigma$-constant.
  
\end{proof}

\def\M{\mathcal{M}}

\subsection{Generization propositions}
In this section, we adapt the
Proposition 2.5 of \cite{Shi} called generization to our situation.

Let $Y$ be a smooth rigid space, $M$ a fine sharp weighted monoid, $I\subseteq\RNN$ an aligned subinterval.
For a log $\nabla$-module $E$ on $Y\times \A{M}{K}\op{I}$ and $\xi\in M^{\mr{gp}}\tens{\Z}\Kbar$, we put:
\[H^0_{\xi}\op{Y\times \A{M}{K}\op{I},E}
\coloneqq\Set{\mathbf{v}\in \Gamma\op{Y\times \A{M}{K}\op{I},E}| \nabla^{\mr{log}}\op{\mathbf{v}}=\xi\mathbf{v}}.\]

\begin{lem}\label{2.3}(cf. Lemma 2.3 of \cite{Shi})
  Let $Y=\Spm \,R$ be a smooth affinoid space,
  $\Sigma\subseteq M^{\mr{gp}}\tens{\Z}\Z_p$ an ($\lr{\mr{NI}\cap\mr{NL}}$-D) subset, $I\subseteq\RNN$ a quasi-open subinterval of positive length.
  Let $L\supseteq R$ be one of the following:
  \begin{enumerate}
  \item an affinoid algebra over $K$ such that $\Spm\,L$ is smooth
    and the supremum norm of $L$ restricts to the supremum norm of $R$.
  \item a complete field with respect to a multiplicative norm which restricts to the supremum norm on $R$.
  \end{enumerate}

  For the first case, we put $\A{M}{L}\op{I}\coloneqq \Spm\,L \times \A{M}{K}\op{I}$. For the second case $\A{M}{L}\op{I}$ means
  the polyannulus over $L$ as usual.
  
  Let $E$ be an object of $\LNM_{Y\times \A{M}{K}\op{I},\Sigma}$ such that the induced object $F\coloneqq E_L \in \LNM_{\A{M}{L}\op{I},\Sigma}$
  is $\Sigma$-unipotent. Then for any aligned closed subinterval $[b,c]\subseteq I$, there exists an element $\xi\in \Sigma$
  such thatly valued
  $H^0_{\xi}\lr{Y\times \A{M}{K}\!\brack{b,c},E}\neq 0$.
\end{lem}
\begin{proof}
  \def\res{\mathrm{res}}
  This proof is almost same as the proof of Lemma 2.3 of \cite{Shi}.
  
  By Proposition \ref{UULNM},  we can take $W \in \ULNM_{\A{M}{L}\!\brack{0,0},\Sigma}$ such that $F=\mathcal{U}_I\op{W}$.

  Let $\rho$ be the residue of $W$ and
  $\xi_1,\ldots,\xi_n$ the exponents of $W$. Let $W=\bigoplus_{k=1}^nW_k$ be the decomposition such that
  $\rho-\xi_k\cdot\id$ acts on $W_i$ nilpotently. Let $\psi_k: W \rightarrow W_k$ be the projection.
  We also put
  \[W'_k\coloneqq \Set{\mathbf{w}\in W| \rho\op{\mathbf{w}}=\xi_k\mathbf{w}}\]

  Take a homomorphism $M\hookrightarrow \N^r$ as in Remark \ref{MtoN}.
  Let $\partial_i$ be composition of $\nabla^{\mr{log}}$ and the $i$-th projection of $\Z^r$.
  Let $\res_i$ be the composition of $\rho$ and the $i$-th projection of $\Z^r$.
  We denote the $i$-th entry of $\xi_k$ by $\xi_{i,k}$.
  We can take polynomials $Q_1,\ldots,Q_r$ such that $Q_i$ divides the minimal polynomial of $\res_i$ and
  the image of $\prod_{i=1}^rQ_i\op{\res_i}$ is not zero and contained in $W'_1$.

  By the definition of $W_k$, we can take $q\in \N$ such that $\lr{\res_i-\xi_{i,k}\cdot\id}^q\op{W_k}=0$ for all $1\leq i \leq r$ and $1 \leq k \leq n$.

  For $l\in \N$, We define the operator $D_l$ on $E$ as follows:
  \[D_l\coloneqq \prod_{i=1}^r\lr{Q_i\op{\partial_i}\lr{\prod_{k=1}^n\prod_{j=1}^l \frac{j-\lr{\partial_i-\xi_{i,k}}}{j-\lr{\xi_{i,1}-\xi_{i,k}}}
      \cdot\frac{j+\lr{\partial_i-\xi_{i,k}}}{j+\lr{\xi_{i,1}-\xi_{i,k}}}}^q}.\]

  \def\v{\mathbf{v}}
  Take two intervals $\brack{b,c}\subseteq\brack{d,e}\subseteq I$ such that $c<e$ and if $b>0$, then $d<b$ (if $b=0$, then $d=0$).
  We prove that for any $\v\in \Gamma\op{Y\times \A{M}{K}\!\brack{d,e}}$, $D_l\op{\v}$ converges to an element of
  $H^0_{\xi_1}\op{Y\times \A{M}{K}\!\brack{b,c},E}$ when $l \rightarrow \infty$.

  Put $\v=\sum_{m\in M^{\mr{gp}}}\v_mt^m$ where $\v_m \in W$. Let $\v_{m,k}\coloneqq\psi_{k}\op{\v_m}$.
  \begin{align*}
    D_l\op{\v}&=\sum_{k=1}^nD_l\op{\sum_{m\in M^{\mr{gp}}}\v_{m,k}t^m}\\
    &=\sum_{k=1}^n\sum_{m\in M^{\mr{gp}}}t^m\lr{\prod_{i=1}^rQ_i\op{\res_i+m_i}\prod_{k=1}^n\prod_{j=1}^l
    \frac{j-\lr{\res_i+m_i-\xi_{i,k}}}{j-\lr{\xi_{i,1}-\xi_{i,k}}}
    \cdot\frac{j+\lr{\res_i+m_i-\xi_{i,k}}}{j+\lr{\xi_{i,1}-\xi_{i,k}}}}^q\op{\v_{m,k}}\\
  \end{align*}
  For the term of $t^0$, since $\prod_{i=1}^rQ_i\op{\res_i}\op{\v}\in W'_1$,
  \[\prod_{i=1}^rQ_i\op{\res_i}\lr{\prod_{k=1}^n\prod_{j=1}^l
    \frac{j-\lr{\res_i-\xi_{i,k}}}{j-\lr{\xi_{i,1}-\xi_{i,k}}}
    \cdot\frac{j+\lr{\res_i-\xi_{i,k}}}{j+\lr{\xi_{i,1}-\xi_{i,k}}}}^q\op{\v_{0,k}}=\prod_{i=1}^rQ_i\op{\res_i}\op{\v_{0,k}}\]
  For the term of $t^m$ such that $0<\abs{m_i}_{\R}\leq l$ for some $i$, it is equal to zero since
  this expression contains $\lr{\res_i-\xi_{k,i}}^q\op{\v_{m,k}}$ for any $k$.

  For other terms, the coefficient of $t^m$ is written as
  \[\sum_{\alpha_1,\ldots,\alpha_r=0}^{q-1}\prod_{i=1}^rc^{m_i,l}_{i,\alpha_i}\lr{\res_i-\xi_{i,k}}^{\alpha_i}\op{\v_{m,k}}\]
  where $c^{m_i,l}_{i,\alpha_i}$ are some constants.
  Let
  \[M^{\mr{gp}}_l\coloneqq \Set{m\in M^{\mr{gp}}\setminus\brace{0}| m_i< -l\text{ or }m_i=0\text{ or }m_i>l\text{ for any $1 \leq i\leq r$}}.\]
  We write
  \[D_l\op{\v}-\prod_{i=1}^rQ_i\op{\res_i}\op{\v_{0}}=\sum_{k=1}^n\sum_{\alpha_1,\ldots,\alpha_r=0}^{q-1}\lr{\sum_{m\in M^{\mr{gp}}_l}
  \prod_{i=1}^rc^{m_i,l}_{i,\alpha_i}t^m\lr{\res_i-\xi_{i,k}}^{\alpha_i}\op{\v_{m,k}}}\]

  By \textit{Claim 3} in the proof of Lemma 2.3 in \cite{Shi}, it is proven that
  for any $\delta>1$, $\abs{c^{j,l}_{i,\alpha}}\leq \lr{\mathrm{const}}\delta^{\abs{j}_{\R}}$.

  We prove that $\sum_{m\in M^{\mr{gp}}}\prod_{i=1}^nc^{m_i,l}_{i,\alpha_i}t^m\lr{\res_i-\xi_{i,k}}^{\alpha_i}\op{\v_{m,k}}$
  converges in $Y\times \A{M}{K}\!\brack{b,c}$ for any $k,\alpha$ which means that $D_l\op{\v}-\prod_{i=1}^rQ_i\op{\res_i}\op{\v_{0}}$
  converges to zero.

  Take $C>0$ such that $\sum_{i=1}^r\abs{m_i}_{\R}\leq C \abs{h}\op{m}$.
  Take $\delta>1$ such that $d\leq \delta^{-C} b$, $\delta^C  c\leq e$.
  Then
  \begin{align*}
    & \abs{\prod_{i=1}^rc^{m_i,l}_{i,\alpha_i}\lr{\res_i-\xi_{i,k}}^{\alpha_i}\op{\v_{m,k}}}b^{-h^{-}\op{m}}c^{h^+\op{m}}\\
    &\leq \lr{\mathrm{const}}\delta^{\sum{\abs{m_i}_{\R}}} b^{-h^{-}\op{m}}c^{h^+\op{m}} \abs{\v_{m,k}}\\
    &\leq \lr{\mathrm{const}}\delta^{C\abs{h}\op{m}} b^{-h^{-}\op{m}}c^{h^+\op{m}} \abs{\v_{m,k}}\\
    &\leq \lr{\mathrm{const}} \lr{\delta^{-C}b}^{-h^{-}\op{m}}\lr{\delta^Cc}^{h^+\op{m}} \abs{\v_{m,k}}
  \end{align*}
  By the assumption that $\v=\sum_{m\in M^{\mr{gp}}} \v_mt^n$ converges in $Y\times \A{M}{K}\!\brack{d,e}$,
  this converges to zero when $\abs{h}\op{m}\rightarrow \infty$.

  So $D_l\op{\v}$ converges to \[f\op{\v}\coloneqq\prod_{i=1}^rQ_i\op{\res_i}\op{\v_{0}}\in H^0_{\xi_1}\op{Y\times \A{M}{K}\!\brack{b,c},E}.\]
  Because $F=\mathcal{U}_I\op{W}$, the image of $\Gamma\op{Y\times \A{M}{L}\!\brack{b,c},F}$ under the map \[\v=\sum_{m\in M^{\mr{gp}}}\v_mt^m\mapsto\prod_{i=1}^rQ_i\op{\res_i}\op{\v_0}\] is $\prod_{i=1}^rQ_i\op{\res_i}\op{W}$.
  Since $\Gamma\op{Y\times \A{M}{K}\!\brack{d,e},E}\tens{R}L$ is dense in $\Gamma\op{Y\times \A{M}{L}\!\brack{b,c},F}$,
  $\Im\op{f}\tens{R}L$ is dense in $\prod_{i=1}^rQ_i\op{\res_i}\op{W}$, so $f\neq 0$. Therefore $H^0_{\xi_1}\op{Y\times \A{M}{K}\!\brack{b,c},E}\neq 0$.
\end{proof}

\begin{prop}\label{generization}(generization, cf. Proposition 2.4 in \cite{Shi})
  In the condition of Lemma \ref{2.3}, we also assume that $0\notin I$.
  If $F\coloneqq E_{L} \in \LNM_{\A{M}{L}\op{I},\Sigma}$ is $\Sigma$-unipontent,
  $E$ is also $\Sigma$-unipotent.
\end{prop}
\begin{proof}
  Let $\brack{b,c}\subseteq I$ be a closed aligned subinterval of positive length.
  By Lemma \ref{2.3}, there exists $\xi \in \Sigma$ such that
  $H_E\coloneqq H^0_{\xi}\op{Y\times \A{M}{K}\!\brack{b,c},E}\neq 0$.

  Let $\pi_Y:Y\times \A{M}{K}\op{I}\rightarrow Y$ be the projection.
  $H_E$ is a finitely generated $R$-module and $\pi_Y^*\op{H_E}\rightarrow E$ is injective on $Y\times \A{M}{K}\!\brack{b,c}$ by the proof of \textit{Claim 1} of Proposition 2.4 in \cite{Shi}.
  Also, by the proof of \textit{Claim 2} of Proposition 2.4 in \cite{Shi}, $H_E$ is a $\nabla$-module on $Y$,
  so it is locally free. Moreover, $H_E$ admits a map
  $H_E\rightarrow H_E\tens{\Z}M^{\mr{gp}}$ induced by $\nabla^{\mr{log}}$ of $E$.
  So $H_E$ can be regarded as an object of $\LNM_{Y\times \A{M}{K}\!\brack{0,0},\Sigma}$.
  Put $G\coloneqq \mc{U}_{\brack{b,c}}\op{H_E}$.
  Then $G$ is a $\brace{\xi}$-constant object of $\LNM_{Y\times \A{M}{K}\!\brack{b,c},\Sigma}$ and
  $G=\pi_{Y}^*\op{H_E}\rightarrow E$ is injective.  
  Then, by the induction on the rank of $E$, we see that $E$ is $\Sigma$-unipotent
  on $Y\times \A{M}{K}\!\brack{b,c}$ for any aligned closed subinterval $\!\brack{b,c}\subseteq I$.
  So $E$ is $\Sigma$-unipotent on $Y\times \A{M}{K}\op{I}$.
  
\end{proof}

\subsection{Transfer theorem}
In this section, we adapt the
Proposition 2.5 of \cite{Shi} called the transfer theorem to our situation.
To do this, we define the notion of log-convergence.

\begin{defn}(Definition 2.4.1 of \cite{Ked})
  Let $X$ be an affinoid space and $E$ a coherent $\O{X}$-module.
  For $\eta>0$, multi-indexed series $\brace{\mb{v}_{\mb{k}}}_{\mb{k} \in \N^n}$ of elements of $\Gamma\op{X,E}$
  is \textit{$\eta$-null} if for any multi-indexed series $\brace{c_{\mb{k}}}_{\mb{k}\in \N^n}$ of elements of $K$ such that
  $\abs{c_{\mb{k}}}\leq \eta^{\abs{\mb{k}}}$, $\brace{c_{\mb{k}}\mb{v}_{\mb{k}}}_{\mb{k}\in \N^n}$ converges to zero when $\abs{\mb{k}}\rightarrow \infty$.
\end{defn}

\begin{defn}(cf. Definition 2.9 in \cite{Shi})\label{log convergence}
  Let $Y=\Spm\,R$ be a smooth connected affinoid space, $M$ a fine sharp semi-saturated weighted monoid and $a$ an element of $\lr{0,1}\cap\Gamma^*$.
  Let $\phi: M\hookrightarrow \Z^r$
  be a homomorphism which induces an isomorphism $M^{\mr{gp}}\tens{\Z}\Kbar \cong \Kbar^r$.
  Let $\partial_1,\ldots,\partial_r$ be the dual basis of the basis of $\Omega^{\mr{log},1}_{\A{M}{K}}$ induced by $\phi$.
  A log $\nabla$-module $E$ on $Y\times \A{M}{K}\!\clop{0,a}$ is 
  called \textit{log-convergent with respect to $\phi$} if for any $a'\in \lr{0,a}\cap\Gamma^*$, $\eta\in\lr{0,1}$ and
  $\mathbf{v}\in \Gamma\op{Y\times \A{M}{K}\!\brack{0,a'},E}$
  \begin{equation}\label{log-conv-seq}
    \brace{\frac{1}{k_1!k_2!\cdots k_{r}!}\lr{\prod_{i=1}^r\prod_{j=0}^{k_i-1}\lr{\partial_i-j}}\op{\mathbf{v}}}_{k_1,k_2,\ldots,k_r\in \N}
  \end{equation}
  is $\eta$-null on $Y\times \A{M}{K}\!\brack{0,a'}$.
\end{defn}

\begin{rem}
  Any subquotient of a log-convergent object is also log-convergent.
\end{rem}

\begin{rem}\label{log-congergence-generator}
  By the same calculation as that after Remerk 2.10 of \cite{Shi}, to check log-convergence, it is enough to check the $\eta$-nullity of (\ref{log-conv-seq}) for a set of generators $\mathbf{v}$ of $\Gamma\op{Y\times \A{M}{K}\!\brack{0,a'},E}$.

  Indeed, by induction of $k$, one can check the following equation:
  \[\prod_{j=0}^{k-1}\lr{\partial_i-j}\op{f\mathbf{v}}=\sum_{k'=0}^{k}\binom{k}{k'}\lr{\prod_{j=0}^{k'-1}\lr{\partial_i-j}\op{f}}\lr{\prod_{j=0}^{k-k'-1}\lr{\partial_i-j}\op{\mathbf{v}}}\]
  for $1 \leq i \leq r$, $\mathbf{v}\in \Gamma\op{Y\times \A{M}{K}\!\brack{0,a'},E}$ and $f \in \Gamma\op{Y\times \A{M}{K}\!\brack{0,a'},\mathcal{O}}$.
  Let $P_{\mathbf{k}}\coloneqq\frac{1}{k_1!k_2!\cdots k_{r}!}\lr{\prod_{i=1}^r\prod_{j=0}^{k_i-1}\lr{\partial_i-j}}$ for $\mathbf{k}=\lr{k_1,\ldots,k_r}$, then
  \[P_{\mathbf{k}}\op{f\mathbf{v}}=\sum_{\mathbf{k}'\leq\mathbf{k}}P_{\mathbf{k}'}\op{f}P_{\mathbf{k}-\mathbf{k}'}\op{\mathbf{v}}.\]

  For $m \in M$, $P_{\mathbf{k}}\op{t^m}=\prod_{i=1}^r\binom{m_i}{k_i}t^m$ where $\phi\op{m}=\lr{m_1,\ldots,m_r}$. So $\abs{P_{\mathbf{k}}\op{f}}\leq \abs{f}$.
  Thus,
  \[\abs{P_{\mathbf{k}}\op{f\mathbf{v}}}\leq\max_{\mathbf{k}'\leq\mathbf{k}}\abs{f}\abs{P_{\mathbf{k}-\mathbf{k}'}\op{\mathbf{v}}}.\]
  Therefore, when $\mathbf{v}=\sum_{s=0}^gf_s\mathbf{v}_s$
  \begin{align*}
    \abs{P_{\mathbf{k}}\op{\mathbf{v}}}\eta^{\abs{\mb{k}}}&\leq\max_{s,\mathbf{k}'\leq\mathbf{k}}\abs{f_s}\abs{P_{\mathbf{k}-\mathbf{k}'}\op{\mathbf{v_s}}}\eta^{\abs{\mathbf{k}}}\\
    &\leq \max_{s,\mathbf{k}'\leq\mathbf{k}}\lr{\abs{f_s}\eta^{\abs{\mathbf{k}'}}}\lr{\abs{P_{\mathbf{k}-\mathbf{k}'}\op{\mathbf{v_s}}}\eta^{\abs{\mathbf{k}-\mathbf{k}'}}}
  \end{align*}
  for $\eta<1$.
\end{rem}

To prove the transfer theorem (Proposition \ref{transfer}), we prove the following
lemma which is  an analogue of Lemma 3.1.6 of \cite{Ked}.

\begin{lem}\label{3.1.6}
  Let $R$ be an affinoid algebra endowed with a norm $\abs{\cdot}$ and let $M$ be a fine weighted monoid. For $a\in \Gamma^*\cap \R_{>0}$ and a formal power series $f=\sum_{m\in M}c_mt^m$ with $c_m \in R$, we put
  \[\abs{f}_{a}\coloneqq\sup_{m \in M}\brace{\abs{c_m}a^{h\op{m}}}.\]
  Then for $a,b\in \Gamma^*\cap \R_{>0}$ and $c\in \brack{0,1}$,
  \[\abs{f}_{a^{c}b^{1-c}}\leq \abs{f}_{a}^c\abs{f}_{b}^{1-c}.\]
\end{lem}
\begin{proof}
  \begin{align*}
    \abs{f}_{a^{c}b^{1-c}} &= \sup_{m\in M}\brace{\abs{c_m}a^{ch\op{m}}b^{\lr{1-c}h\op{m}}}\\
    &= \sup_{m\in M}\brace{\lr{\abs{c_m}a^{h\op{m}}}^{c}\lr{\abs{c_m}b^{h\op{m}}}^{1-c}}\\
    &\leq \sup_{m\in M}\brace{\abs{c_m}a^{h\op{m}}}^{c}\sup_{m\in M}\brace{\abs{c_m}b^{h\op{m}}}^{1-c}\\
    &= \abs{f}_a^c\abs{f}_{b}^{1-c}.
  \end{align*}
\end{proof}

\begin{prop}\label{transfer}(transfer theorem, cf. Proposition 2.12 in \cite{Shi})
  Let $Y=\Spm\,R$ be a smooth connected affinoid space, $M$ a fine sharp semi-saturated weighted monoid.
  Let $\Sigma\subseteq M^{\mr{gp}}\tens{\Z} \Z_p$ be an ($(\mr{NI}\cap\mr{NL})$-D) subset.
  Take $\phi:M\rightarrow \Z^r$ as in Definition \ref{log convergence}.
  For any an object $E$ of
  $\LNM_{Y\times \A{M}{K}\!\clop{0,1},\Sigma}$, if $E$ is log-convergent with respect to $\phi$, then $E$ is $\Sigma$-unipotent.
\end{prop}
\begin{proof}
  \def\res{\mr{res}}
  \def\v{\mb{v}}
  By induction on the rank of $E$, it is enough to show that for any log-convergent object $E$ of $\LNM_{Y\times \A{M}{K}\!\brack{0,a},\Sigma}$
  there exists a $\Sigma$-constant subobject of $E$ in $\LNM_{Y\times \A{M}{K}\!\brack{0,a'},\Sigma}$ for any $0<a'<a<1$.
  
  Let $W\coloneqq E|_{Y\times \A{M}{K}\!\brack{0,0}}$. We define $\xi_i\in \Z_p$, $\res_i$, $\psi_k$, $e\in \N$, $Q_i$, $D_l$ as in the proof of Proposition \ref{2.3} from this $W$.
  
  As in the proof of Proposition 2.12 in \cite{Shi},
  we want to show that
  for any $0<a'<a<1$ and $\mathbf{v}\in \Gamma\op{Y\times \A{M}{K}\!\brack{0,a},E}$, $D_l\op{\mathbf{v}}$ converges on $\A{M}{K}\!\brack{0,a'}$ when $l \rightarrow \infty$.

  First, we show that $D_{l}\op{\mathbf{v}}$ is $\rho$-null on $\A{M}{L}\!\brack{0,a}$ for any $\rho\in \lr{0,1}$. 
  Put \[P_l\op{y}\coloneqq \prod_{j=1}^l\frac{j-1-y}{j}\] and \[c_{ikl}\coloneqq \prod_{j=1}^l\frac{j}{j-\lr{\xi_{1,k}-\xi_{i,k}}}\cdot\frac{j}{j+\lr{\xi_{1,k}-\xi_{i,k}}}.\]
  Note that $P_l\op{\xi}\in \Z_p$ for any $\xi \in \Z_p$.
  Then,
  \begin{align*}
    &\prod_{j=1}^l\frac{j-\lr{\partial_i-\xi_{i,k}}}{j-\lr{\xi_{1,k}-\xi_{i,k}}}\cdot\frac{j+\lr{\partial_i-\xi_{i,k}}}{j+\lr{\xi_{1,k}-\xi_{i,k}}}\\
    &=c_{ikl}P_l\op{\partial_i-\xi_{i,k}-1}P_l\op{-\partial_i+\xi_{i,k}-1}\\
    &=c_{ikl}\lr{\sum_{q=0}^lP_{l-q}\op{-\xi_{i,k}-1}P_q\op{\partial_i}}\lr{\sum_{q=0}^lP_{l-q}\op{\xi_{i,k}-1}P_q\op{-\partial_i}}.
    \end{align*}
    $P_{q}\op{-y}$ is written as a linear combination of $P_{q}\op{y}, P_{q-1}\op{y}, \ldots, P_{0}\op{y}$
    with integer coefficients.
    $P_{q_1}\op{y}P_{q_2}\op{y}$ is written as a linear combination of $P_{q_1+q_2}\op{y}$, $P_{q_1+q_2-1}\op{y}$, $\ldots$, $P_{0}\op{y}$
    with integer coefficients. 
    So we can write
  \[D_l=\lr{\prod_{i=1}^rQ_i\op{\partial_i}\prod_{k=1}^nc_{ikl}}\lr{\sum_{q_1,\ldots,q_r=0}^{2nle}\lr{\mr{const.}}P_{q_1}\op{\partial_1}\cdots P_{q_r}\op{\partial_r}}.\]
  where the constants are contained in $\Z_p$. Thus
  \[\abs{D_l\op{\mb{v}}}\leq\lr{\prod_{i=1}^r\prod_{k=1}^n\abs{c_{ikl}}}\max_{0\leq q_1,\ldots,q_r\leq 2nlp}\abs{P_{q_1}\op{\partial_1}\cdots P_{q_r}\op{\partial_r}\op{\prod_{i=1}^rQ_i\op{\partial_i}\op{\mb{v}}}}.\]
  By the assumption on $\Sigma$ and log-convergence, this is $\rho$-null for any $\rho\in\lr{0,1}$. 
  
  We can take $b\in \left(0,a\right]$ such that $E|_{Y\times \A{M}{K}\!\brack{0,b}}=\mc{U}_{\brack{0,b}}\op{W}$. Indeed, if we take a finite admissible affinoid covering $Y=\bigcup U_{i}$ such that $W|_{U_{i}}$ is free, by Lemma \ref{1.9}, we may assume $E|_{U_i\times \A{M}{K}\!\brack{0,b}}=\mc{U}_{\brack{0,b}}\op{W|_{U_i}}$, so $E|_{Y\times\A{M}{K}\!\brack{0,b}}=\mc{U}_{\brack{0,b}}\op{W}$ by Corollary \ref{UULNM}.
  Then $\lr{D_{l+1}\op{\mathbf{v}}-D_{l}\op{\mathbf{v}}}$ is $\eta$-null on $Y\times\A{M}{K}[0,b]$ for some $\eta>1$ by the \textit{claim 2} in the proof of Lemma 2.3 of \cite{Shi}.

  Because $E$ is locally free, it is a direct summand of a free module $F$.
  We fix a basis of $F$.
  Applying Lemma \ref{3.1.6} to each component with respect to the basis of $F$, $\lr{D_{l+1}\op{\mathbf{v}}-D_{l}\op{\mathbf{v}}}$ is $\rho^{1-c}\eta^c$-null on $Y\times\A{M}{K}[0,a'']$ where $c\in \lr{0,1}$ and $a''\coloneqq a^{1-c}b^c$. Take $c$ such that $\rho^{1-c}\eta^c=1$, then $D_l\op{\mathbf{v}}$ converges on $Y\times\A{M}{K}[0,a'']$.
  If we take $\rho<1$ arbitrarily close to $1$, $c$ can be arbitrarily close to $0$ and $a''$ can be arbitrarily close to $a$.

  So, $f\op{\mb{v}}\coloneqq \lim_l D_l\op{\mb{v}}$ exists on $Y\times \A{M}{K}\!\brack{0,a'}$ for any $0<a'<a$ and it is equal
  to $\prod_{i=1}^r Q_i\op{\mr{res}_i}\op{\mb{v}}$ on $Y\times \A{M}{K}\!\brack{0,b}$ by the calculation in the proof of Proposition \ref{2.3}.
  Thus $f\op{\mb{v}}$ is an element of $H^0_{\xi_1}\op{Y\times \A{M}{K}\!\brack{0,a'},E}$ and 
  $f\neq 0$ by the argument similar to the proof of Proposition \ref{2.3}.
  Hence $H_E\coloneqq H^0_{\xi_1}\op{Y\times \A{M}{K}\!\brack{0,a'},E}$ is non-zero.

  By the same argument as the proof of Proposition 2.12 in \cite{Shi},
  we see that $\pi_{a'}^*H_E$ is $\Sigma$-constant subobject of $E$ where $\pi_{a'}: Y\times \A{M}{K}\!\brack{0,a'}\rightarrow Y$ is the projection.
\end{proof}

\section{Log isocrystals}
In this section, we assume that $\OK$ is a discrete valuation ring.

Let $N$ be a fine monoid and $\alpha: N \rightarrow \OK$ a monoid homomorphism.
$\lr{\Spf\,\OK, N}$ denotes $\Spf\,\OK$ with the log structure induced by $\alpha$.
We assume that $\alpha$ is a good chart at the unique point.
We also assume that the image of $\alpha$ does not contain $0$, i.e., $\lr{\Spf\,\OK,N}$ induces a trivial
log structure on $\Spm\,K$.
$\lr{\Spec\,k, N}$ denotes $\Spec\,k$ with the log structure induced by $\alpha_k:N \rightarrow \OK\rightarrow k$.

\def\OKa{\lr{\Spf\,\OK, N}}
\def\ka{\lr{\Spec\,k, N}}
\def\tA#1{\hat{\mathbb{A}}_{{#1},{\OK}}}
\def\tAA{\hat{\mathbb{A}}_{\OK}}
In this section, a log variety over $\ka$ means a log scheme over $\ka$ such that the underlying scheme is separated and of finite type over $k$.

For a monoid $M$, $\A{M}{k}$ denotes $\Spec\,k\!\brack{M}$ as a log scheme over $k$ with the log structure defined by the natural map $M\rightarrow k\!\brack{M}$. $\tA{M}$ denotes $\Spf\,\OK\!\angbra{M}$ as a log formal scheme over $\OK$ with the log structure defined by the natural map $M\rightarrow \OK\!\angbra{M}$.

\subsection{Convergent log isocrystals and convergent log $\nabla$-modules}
Let $\lr{\ov{X},\mc{M}}$ be a fine log variety over $\ka$ and $X$ an open subscheme of $\ov{X}$.
In this subsection, we prove that, locally on $\ov{X}$, convergent log isocrystals correspond to convergent log $\nabla$-modules.

First, we recall the definition of log tubes and sheaves of overconvergent sections.

\def\E{\mathcal{E}}
\begin{defn}\label{log tube}
  Let $\lr{\ov{X},\mc{M}}$ be a fine log variety over $\ka$ and $i:\lr{\ov{X},\mc{M}}\hookrightarrow \lr{P,\mc{L}}$ a closed immersion into
  a fine log formal scheme which is separated and topologically of finite type over $\OKa$. By Proposition-Definition 2.10. of \cite{Shi6}, there exists a unique homeomorphic exact closed immersion $i^{\mathrm{ex}}:\lr{\ov{X},\mc{M}}\hookrightarrow \lr{P^{\mathrm{ex}},\mc{L}^\mathrm{ex}}$ such that $i$ factors through $i^{\mathrm{ex}}$ and $i^{\mathrm{ex}}$ is universal among such immersions, i.e., for any homeomorphic exact closed immersion $i':\lr{\ov{X},\mc{M}}\hookrightarrow \lr{P',\mc{L}'}$ such that $i$ factors through $i'$, it is uniquely decomposed as $i'=j\circ i^{\mathrm{ex}}$ where $j:\lr{P',\mc{L}'} \rightarrow \lr{P^{\mathrm{ex}},\mc{L}^\mathrm{ex}}$.

  Let $\mathrm{sp}: P^{\mathrm{ex}}_K \rightarrow \ov{X}$ be the specialization map. For an open subscheme $X\subseteq \ov{X}$,
  The \textit{log tube} $\tube{X}^{\mr{log}}_P$ is defined to be the inverse image of $X$ under $\mathrm{sp}$. In particular, $\tube{\ov{X}}^{\mr{log}}_P=P^{\mathrm{ex}}_K$.
\end{defn}

\begin{rem}(cf. Definition 6.1.4 of \cite{Ked})
  Let $\lr{\ov{X},\mc{M}}$ and $i:\lr{\ov{X},\mc{M}}\hookrightarrow \lr{P,\mc{L}}$ be as in Definition \ref{log tube}.
  Assume $i$ has a factorization
  \[\lr{\ov{X},\mc{M}}\xhookrightarrow{i'} \lr{P',\mc{L}'}\xrightarrow{f'}\lr{P,\mc{L}}\]
  such that $i'$ is an exact closed immersion and $f'$ is formally log \'etale. Then $P^{\mathrm{ex}}$ is the completion of $P'$ along $\ov{X}$. So $\tube{X}^{\mr{log}}_P=\tube{X}_{P'}$ for any open subscheme $X\subseteq \ov{X}$.
\end{rem}

\begin{defn}
  Let $\lr{\ov{X},\mc{M}}$ and $i:\lr{\ov{X},\mc{M}}\hookrightarrow \lr{P,\mc{L}}$ be as in Definition \ref{log tube}.
  Let $X\subseteq \ov{X}$ be an open subscheme.
  We put
  \[j^{\dagger}\O{\tube{\ov{X}}^{\mr{log}}_{P}}\coloneqq \lim_{\substack{\longrightarrow\\ V}} \iota_{V*}\iota_{V}^{-1}\O{\tube{\ov{X}}^{\mr{log}}_P}\]
  Where $V$ runs through strict neighborhoods of $\tube{X}^{\mr{log}}_{P}$ in $\tube{\ov{X}}^{\mr{log}}_{P}$ and $\iota_V: V\hookrightarrow \tube{\ov{X}}^{\mr{log}}_{P}$ is the open immersion, and call this sheaf the sheaf of overconvergent sections.
\end{defn}

We recall the definition of overconvergent log isocrystals in \S 4 of \cite{Shi5} or \S 10 of \cite{diPr2}.

\begin{defn}\label{log isocrystal}(cf. Definition 6.1.7 of \cite{Ked})
  Let $\lr{\ov{X},\mc{M}}$ be a fine log variety over $\ka$ and $X\subseteq \ov{X}$ an open subscheme.
  Assume there exists a commutative diagram:
  \[\begin{tikzcd}
  \lr{\ov{X},\mc{M}}\arrow[hook]{r}{i} \arrow{d}& \lr{P,\mc{L}} \arrow{d}\\
  \ka\arrow[hook]{r}& \OKa
  \end{tikzcd}\]
  where $i$ is a closed immersion and $\lr{P,\mc{L}}$ is a $p$-adic fine log formal scheme separated, topologically of finite type and formally log smooth over $\OKa$.

  Let $\lr{P\op{i},\mc{L}\op{i}}$ be the ($i+1$)-fold fiber product of $\lr{P,\mc{L}}$ over $\OKa$.
  We denote the projections by $\pi_0,\pi_1: \lr{P\op{1},\mc{L}\op{1}} \rightarrow \lr{P,\mc{L}}$,
  $\pi_{0,1},\pi_{1,2},\pi_{0,2}: \lr{P\op{2},\mc{L}\op{2}} \rightarrow \lr{P\op{1},\mc{L}\op{1}}$
  and the diagonal map by $\Delta: \lr{P,\mc{L}} \rightarrow \lr{P\op{1},\mc{L}\op{1}}$.

  A \textit{overconvergent log isocrystal on $\lr{X,\ov{X},\mc{M}}$ over $\OKa$} is a coherent
  $j^{\dagger}\O{\tube{\ov{X}}^{\mr{log}}_P}$-module $\E$  equipped with a $j^{\dagger}\O{\tube{\ov{X}}^{\mr{log}}_{P\op{1}}}$-module isomorphism
  $\epsilon: \pi^*_1\op{\E}\rightarrow \pi^*_0\op{\E}$ such that
  $\Delta^*\op{\epsilon}=\id_{\E}$
  and the cocycle condition $\pi_{0,1}^*\op{\epsilon}\circ\pi_{1,2}^*\op{\epsilon}=\pi_{0,2}^*\op{\epsilon}$ holds on $\tube{\ov{X}}^{\mr{log}}_{P\op{2}}$.
  $\mc{E}$ is \textit{locally free} if $\mc{E}$ is locally free as a $j^{\dagger}\O{\tube{\ov{X}}^{\mr{log}}_P}$-module.

  If $X=\ov{X}$, we call them \textit{convergent log isocrystals on $\lr{\ov{X},\mc{M}}$ over $\OKa$}.
\end{defn}

A locally free overconvergent log isocrystal naturally induces a locally free $j^{\dagger}\O{\tube{\ov{X}}^{\mr{log}}_P}$-module with an integrable connection.
So it defines a log $\nabla$-module on some strict neighborhood of $\tube{X}^{\mr{log}}_P$ in $\tube{\ov{X}}^{\mr{log}}_P$.


We can always take $\lr{P,\mc{L}}$ as in Proposition \ref{log isocrystal} \'etale locally on $\ov{X}$.
  Hence, for general $\lr{\ov{X},\mc{M}}$, we can define the category of overconvergent log isocrystals on $\lr{X,\ov{X},\mc{M}}$
  using an appropriate \'etale hypercovering.

\begin{defn}
  Let $\lr{\ov{X},\mc{M}}$ be a fine log variety over $\ka$ and $X\subseteq \ov{X}$ an open subscheme.
  We take an \'etale covering $\lr{\ov{X}^0,\mc{M}^0}\rightarrow \lr{\ov{X},\mc{M}}$ such that $\lr{\ov{X}^0,\mc{M}^0}$
  has an immersion $\lr{\ov{X}^0,\mc{M}^0}\hookrightarrow \lr{P^0,\mc{L}^0}$ satisfying the condition of Definition \ref{log isocrystal}.
  For n=1,2, let $\lr{\ov{X}^n,\mc{M}^n}$ be the $\lr{n+1}$-fold fiber product of $\lr{\ov{X}^0,\mc{M}^0}$ over $\lr{\ov{X},\mc{M}}$,
  let $\lr{\ov{P}^n,\mc{L}^n}$ be the $\lr{n+1}$-fold fiber product of $\lr{\ov{P}^0,\mc{L}^0}$ over $\OKa$
  and let $X^n\coloneqq X\fibpro{\ov{X}}\ov{X}^n$.
  For $n=0, 1, 2$, let $I^{\dagger,n}$ be the category of overconvergent log isocrystals on $\lr{X^n,\ov{X}^n,\mc{M}^n}$ over $\OKa$ defined by $\lr{\ov{X}^n,\mc{M}^n}\hookrightarrow\lr{P^n,\mc{L}^n}$.
  The category of \textit{overconvergent log isocrystals on $\lr{X,\ov{X},\mc{M}}$ over $\OKa$} is defined as the category of descent data with respect to
  \[\begin{tikzcd}
  I^{\dagger,0}
  \arrow[r,shift left]\arrow[r, shift right] &
  I^{\dagger,1}
  \arrow[r]\arrow[r,shift left=2]\arrow[r, shift right=2] &
  I^{\dagger,2}.
  \end{tikzcd}\]
\end{defn}

\begin{rem}
  The above definition does not depend on the choice of an \'etale covering $\lr{\ov{X}^0,\mc{M}^0}\rightarrow \lr{\ov{X},\mc{M}}$ and an immersion $\lr{\ov{X}^0,\mc{M}^0}\hookrightarrow \lr{P^0,\mc{L}^0}$ by Lemma 4.5 of \cite{Shi5}.
\end{rem}

To define the notion of convergent log $\nabla$-modules, we introduce the notion of charted standard small frame.

\begin{defn}\label{chart}(cf. Definition 3.3 of \cite{Shi})
  A \textit{charted standard small frame of $\lr{\ov{X},\mc{M}}$} is the following data:
  \begin{itemize}
  \item an exact closed immersion $\lr{\ov{X},\mc{M}}\hookrightarrow \lr{P,\mc{L}}$ over $\OKa$ where $P$ is an affine formally log smooth $p$-adic formal scheme
    over $\OKa$,
  \item a chart
    \[\begin{tikzcd}
    N\arrow{r}{\alpha} \arrow{d}{g} & \OK\arrow{d}\\
    M\arrow{r}{\beta} &\O{P}
    \end{tikzcd}\]
    of the structure morphism $\lr{P,\mc{L}}\rightarrow \OKa$,
  \item
    $\gamma_1,\ldots \gamma_d \in \O{P}$,
  \end{itemize}
  such that $P_k=\ov{X}$ and the morphism
  $P \rightarrow \AA{\OK}^d\times\A{M}{\OK}\fibpro{\A{N}{\OK}}\Spf\,\OK$
  induced by $\gamma_1,\ldots,\gamma_d$ and the chart is formally \'etale.

  A charted standard small frame is \textit{good at $\ov{x}$}
  if $M=\mc{M}_{\ov{x}}$.
\end{defn}

\begin{rem}\label{exist_chart}
  Assume that $\lr{\ov{X},\mc{M}}$ is log smooth over $\OKa$.
  Then, \'etale locally on $\ov{X}$, there exists a charted standard small frame.
  Indeed, \'etale locally we can take a chart
  \[\begin{tikzcd}
  N\arrow{r}{\alpha} \arrow{d} & k \arrow{d}\\
  M\arrow{r} &\O{\ov{X}}
  \end{tikzcd}\]
  such that the morphism $\ov{X}\rightarrow \A{M}{k}\fibpro{\A{N}{k}}\Spec\,k$
  is the composition of an \'etale morphism $\ov{X}\rightarrow \AA{k}^d\times\A{M}{k}\fibpro{\A{N}{k}}\Spec\,k$ and the projection
  $\AA{k}^d\times\A{M}{k}\fibpro{\A{N}{k}}\Spec\,k\rightarrow \A{M}{k}\fibpro{\A{N}{k}}\Spec\,k$ for some $d\in \N$.
  We can take a lift $P$ of $\ov{X}$ such that the following diagram is Cartesian
  \[\begin{tikzcd}
  \ov{X}\arrow{r} \arrow[hook]{d} & \AA{k}^d\times \A{M}{k}\fibpro{\A{N}{k}}\Spec\,k \arrow[hook]{d}\\
  P \arrow{r}& \AA{\OK}^d\times\A{M}{\OK}\fibpro{\A{N}{\OK}}\Spf\,\OK
  \end{tikzcd}\]
  and the bottom horizontal morphism is formally \'etale. We define the log structure $\mc{L}$ on $P$ over $\OKa$
  by $P\rightarrow \A{M}{\OK}\fibpro{\A{N}{\OK}}\Spf\,\OK$.
  The exact closed immersion $\lr{\ov{X},\mc{M}}\rightarrow \lr{P,\mc{L}}$ forms a charted standard small frame
  with $\gamma_1,\ldots \gamma_d \in \O{P}$ and the chart associated to the above diagram.

  Moreover, for any geometric point $\ov{x}$ of $\ov{X}$, if $N\rightarrow \ov{\mc{M}}_{\ov{x}}$ is injective and
  $\lr{\ov{\mc{M}}_{\ov{x}}/N}^{\mr{gp}}$ is torsion-free, there exists a charted standard small frame
  which is good at $\ov{x}$. Indeed, by Lemma \ref{3.1.1.A},
  we can take a good one as the above chart.
\end{rem}

In the rest of this subsection, we assume the existence of a charted standard small frame $\lr{P,M,\gamma_1,\ldots, \gamma_d}$.
Let $X$ be a open subscheme of $\ov{X}$.

Let $\lr{P\op{i},\mc{L}\op{i}}$ be the ($i+1$)-st fiber product of $\lr{P,\mc{L}}$ over $\OKa$.

\begin{defn}(cf. Definition 6.3.1 of \cite{Ked}, Definition 3.4 of \cite{Shi})
  A log $\nabla$-module $E$ on $P_K$ is \textit{convergent} if it the restriction of $E$ to some strict neighborhood
  of $\tube{X}_P$ in $P_K$ comes from a locally free overconvergent log isocrystal on $\lr{X,\ov{X},\mc{M}}/\OKa$.
\end{defn}

\begin{prop}\label{cat_equiv}(cf. Theorem 6.4.1 of \cite{Ked})
  If $X$ is dense in $\ov{X}$,
  there is an equivalence between the category $I_{\mr{conv}}\op{\lr{\ov{X},\mc{M}}/\OKa}^{\mr{lf}}$ of locally free convergent log isocrystals on $\lr{\ov{X},\mc{M}}/\OKa$
  and the category $\widehat{MIC}_{\mr{conv}}\op{\lr{\ov{X},\mc{M}}/\OKa}$ of convergent log $\nabla$-modules on $P_K$.
\end{prop}
\begin{proof}
Let $I_{\mr{inf}}\op{\lr{\ov{X},\mc{M}}/\OKa}$ be the category of isocrystals on the infinitesimal site $\lr{\lr{\ov{M},\mc{M}}/\OKa}_{\mr{inf}}$ which is defined in Definition 8 of \cite{diPr2}. By \S 3 of \cite{diPr2}, the category $\widehat{MIC}\op{\lr{\ov{X},\mc{M}}/\OKa}$ of log $\nabla$-modules on $P_K$ is equivalent to $I_{\mr{inf}}\op{\lr{\ov{X},\mc{M}}/\OKa}$. By Theorem 3 of \cite{diPr2}, the natural functor from $I_{\mr{conv}}\op{\lr{\ov{X},\mc{M}}/\OKa}^{\mr{lf}}$ to $I_{\mr{inf}}\op{\lr{\ov{X},\mc{M}}/\OKa}$ is fully faithful, so the natural functor from $I_{\mr{conv}}\op{\lr{\ov{X},\mc{M}}/\OKa}^{\mr{lf}}$ to $\widehat{MIC}\op{\lr{\ov{X},\mc{M}}/\OKa}$ is fully faithful.  We have to show that the essential image of this functor is equal to $\widehat{MIC}_{\mr{conv}}\op{\lr{\ov{X},\mc{M}}/\OKa}$.

It is clear that the image is contained by $\widehat{MIC}_{\mr{conv}}\op{\lr{\ov{X},\mc{M}}/\OKa}$. For $E \in \widehat{MIC}_{\mr{conv}}\op{\lr{\ov{X},\mc{M}}/\OKa}$, the restriction of $E$ to $X$ comes from a convergent log isocrystal on $\lr{X,\mc{M}|_X}/\OKa$. So by Proposition 8 of \cite{diPr2}, $E$ is contained in the image of $I_{\mr{conv}}\op{\lr{\ov{X},\mc{M}}/\OKa}^{\mr{lf}}$.

\end{proof}

\subsection{Exponents of convergent log isocrystals}

\begin{defn}
  Let $\lr{\ov{X},\mc{M}}$ be a log smooth variety over $\ka$.
  Let $\ov{x}$ be a geometric point of $\ov{X}$. Put $M\coloneqq \ov{\mc{M}}_{\ov{x}}$.
  Assume that $N\rightarrow M$ is injective and $\lr{M/N}^{\mr{gp}}$ is torsion-free.
  Let $\Sigma$ be a subset of $M^{\mr{gp}}\tens{\Z}\Kbar$.

  Let $\E$ be a locally free convergent log isocrystal on $\lr{\ov{X},\mc{M}}/\OKa$.
  Take a charted standard small frame $\lr{\lr{U,\mc{M}|_U}\hookrightarrow \lr{P,\mc{L}},M\rightarrow \O{P},\gamma_1,\ldots,\gamma_d}$ of some neighborhood $U$ of $\ov{x}$ in $\ov{X}$ which is good at $\ov{x}$.
  Let $E_P$ be the log $\nabla$-module on $P_K$ induced by $\E$. 
  We say that $\E$ \textit{has exponents in $\Sigma$ at $\ov{x}$} if
  the exponents of $E_P$ are contained in $\Sigma$.
\end{defn}

\begin{prop}\label{well-def_exp}
  The above definition is independent of the choice of charted standard small frames.
\end{prop}
\begin{proof}
  We assume that $\ov{X}=U$ has two charted standard small frames\[\lr{\lr{\ov{X},\mc{M}}\hookrightarrow \lr{P_1,\mc{L}_1},\beta_1:M\rightarrow \O{P_1},\gamma_{1,1},\ldots,\gamma_{1,d}}\] and \[\lr{\lr{\ov{X},\mc{M}}\hookrightarrow \lr{P_2,\mc{L}_2},\beta_2:M\rightarrow \O{P_2},\gamma_{2,1},\ldots,\gamma_{2,d}}\]
  which are good at $\ov{x}$.

  Let $\lr{P_{1,2},\mc{L}_{1,2}}\coloneqq \lr{P_1,\mc{L}_1}\fibpro{\OKa} \lr{P_2,\mc{L}_2}$.
  Let $\pi_1: \lr{P_{1,2},\mc{L}_{1,2}}\rightarrow \lr{P_1,\mc{L}_1}$ and $\pi_2: \lr{P_{1,2},\mc{L}_{1,2}} \rightarrow \lr{P_2,\mc{L}_2}$ be the projections.   Let $m_1,\ldots,m_r\in M^{\mr{gp}}$ be a set of lifts of free generators of $\lr{M/N}^{\mr{gp}}$.
  Put $\beta\op{m}_i\coloneqq \pi^{*}_i\op{\beta_i\op{m}}$ for $i=1,2$.
  Let $M\op{1}$ be the amalgamated sum of $2$ copies of $M$ over $N$ in the category of fine monoids.
  Let $M\op{1}'$ be the submonoid of $M\op{1}^{\mr{gp}}$ generated by $M\op{1}$ and the kernel of the codiagonal homomorphism $M\op{1}^{\mr{gp}}\rightarrow M^{\mr{gp}}$.
  There exists a chart $M\op{1}\rightarrow \O{P_{1,2}}$ of $\mc{L}_{1,2}$ induced by $\beta_1$ and $\beta_2$.
  Let $\lr{P_3,\mc{L}_3}\coloneqq \lr{P_{1,2},\mc{L}_{1,2}} \fibpro{\tA{M\lr{1}}}\tA{M\lr{1}'}$.
  We denote the natural map $\lr{P_3,\mc{L}_3}\rightarrow \lr{P_{1,2},\mc{L}_{1,2}}$ by $\iota$.
  The following diagram commutes:
  \[\begin{tikzcd}
  \lr{X,\mc{M}}\arrow{d}\arrow{r}&\tA{M}\arrow{r}&\tA{M\op{1}'}\arrow{d}\\
  \lr{P_{1,2},\mc{L}\op{j}}\arrow{rr}& &\tA{M\op{1}}
  \end{tikzcd}\]
  where the left vertical morphism is the diagonal map.
  Thus there exists the diagonal map $\lr{X,\mc{M}} \rightarrow \lr{P_3,\mc{L}_3}$.
  Since $\iota^*\op{\frac{\beta\op{m}_2}{\beta\op{m}_1}}$ is invertible on $P_3$ for any $m \in M$,
  the map $M \ni m \mapsto \iota^*\op{\beta\op{m}_1}$ is a fine chart of $\mc{L}_3$. Thus
  $\lr{\ov{X},\mc{M}} \rightarrow \lr{P_3,\mc{L}_3}$
  is an exact closed immersion. So
  \[\tube{\ov{X}}^{\mr{log}}_{P_{1,2}}=\tube{\ov{X}}_{P_3}\]

  On the other hand, we define the morphism
  \[f: \lr{P_3,\mc{L}_3}\rightarrow \tAA^{r+d} \times \lr{P_2,\mc{L}_2}\]
  by $\frac{\beta\op{m_1}_2}{\beta\op{m_1}_1}-1,\ldots,\frac{\beta\op{m_r}_2}{\beta\op{m_r}_1}-1,
  \pi^*_1\op{\gamma_{1,1}}-\pi^*_2\op{\gamma_{2,1}},\ldots,\pi^*_1\op{\gamma_{1,d}}-\pi^*_2\op{\gamma_{2,d}}\in \O{P_3}$,
  regarded as $P_3\rightarrow \tAA^{1}$, and the projection $\pi_2\circ\iota$.
  This is log \'etale and strict, so \'etale. So,
  \[\tube{\ov{X}}_{P_3}\cong\tube{\ov{X}}_{\tAA^{r+d} \times P_2}=\AA{K}^{r+d}\!\clop{0,1}\times\tube{\ov{X}}_{P_2}\]
  by Proposition 1.3.1 of \cite{Ber}.

  The following diagram commutes:
  \[\begin{tikzcd}
    & & \lr{\ov{X},\mc{M}} \arrow[hook]{lld} \arrow[hook]{ld} \arrow[hook]{d} \arrow[hook]{rd} &\\
    \lr{P_1,\mc{L}_1}&\lr{P_3,\mc{L}_3}\arrow{l}{\pi_1\circ\iota}\arrow[swap]{r}{f}&\arrow[shift left]{r}{\mr{pr}_2}\tAA^{r+d} \times \lr{P_2,\mc{L}_2}&\arrow[shift left]{l}{s}\lr{P_2,\mc{L}_2},
  \end{tikzcd}\]
  where $\mr{pr_2}$ is the projection and $s$ is the section to
  $\brace{\mb{0}} \times P_2$. We denote the morphism $\tube{\ov{X}}_{P_2}\rightarrow \tube{\ov{X}}_{\tAA^{r+d}\times P_2}$ induced by $s$
  also by $s$ and denote the morphism $\tube{\ov{X}}_{\tAA^{r+d}\times P_2}\cong\tube{\ov{X}}_{P_3}\rightarrow \tube{\ov{X}}_{P_1}$ induced by $\pi_1\circ\iota$ also by $\pi_1$. Then $E_{P_2}=\lr{\pi_1\circ s}^{*}\op{E_{P_1}}$. Since
  \[\lr{\pi_1\circ s}^{*}\op{\beta_1\op{m}}=s^{*}\op{\beta\op{m}_1}=\beta_2\op{m}\]
  on $\tube{\ov{X}}_{P_2}$ for any $m \in M$, the pullback of the sheaf induced by $\Sigma$ on $P_1$ by $\pi_1\circ\iota\circ s$ is equal to
  the sheaf induced by $\Sigma$ on $P_2$.
  Therefore if the exponents of $E_{P_1}$ are contained in $\Sigma$, the exponents of $E_{P_2}$ are contained $\Sigma$.
\end{proof}

\begin{defn}
  Let $\lr{\ov{X},\mc{M}}$ be a log smooth variety over $\ka$ such that
  $N\rightarrow \ov{\mc{M}}_{\ov{x}}$ is injective and $\lr{\ov{\mc{M}}_{\ov{x}}/N}^{\mr{gp}}$ is torsion-free
  at any geometric point $\ov{x}$ of $\ov{X}$.
  Let $\mc{S}\subseteq \ov{\mc{M}}^{\mr{gp}}\tens{\Z} \Kbar$ be a subsheaf.
  A locally free convergent log isocrystal $\E$ on $\lr{\ov{X},\mc{M}}/\OKa$ \textit{has exponents in $\mc{S}$} if for any geometric point $\ov{x}$,
  $\E$ has exponents in $\mc{S}_{\ov{x}}$ at $\ov{x}$.
\end{defn}

\subsection{Unipotence of overconvergent log isocrystals}\label{unipotence}

We consider the following situation:
\begin{situation}\label{situation}
  Let $\lr{\ov{X},\mc{M}_0\oplus\mc{M}}$ be a log variety over $\ka$ satisfying the following conditions:
  \begin{itemize}
  \item
    $\mc{M}$ and $\mc{M}_0$ are fine. 
  \item
    The map $N\rightarrow \mc{M}_0\oplus\mc{M}$ factors as $N\rightarrow  \mc{M}_0\rightarrow \mc{M}_0\oplus{}\mc{M}$.
  \item
    $\lr{\ov{X},\mc{M}_0\oplus\mc{M}}$ is log smooth over $\ka$.
  \item $N\rightarrow \lr{\overline{\mathcal{M}_0}}_{\overline{x}}$ is injective and $\lr{\lr{\ov{\mathcal{M}_0}}_{\overline{x}}/N}^{\mr{gp}}$ is torsion-free at any geometric point $\overline{x}$ of $\ov{X}$.
  \item
    $\overline{\mathcal{M}}^{\mr{gp}}_{\overline{x}}$ is torsion-free at any geometric point $\overline{x}$ of $\ov{X}$.    
  \end{itemize}
  Let $X$ be the trivial locus of $\mc{M}$ which is an open dense subset of $\ov{X}$ as shown below.
\end{situation}

  Let $\ov{x}$ be a geometric point of $\ov{X}$. Let $M_0\coloneqq\lr{\mc{M}_0}_{\ov{x}}$ and $M\coloneqq \mc{M}_{\ov{x}}$.
  Take a charted standard small frame $\lr{\lr{\ov{X},\mc{M}}\hookrightarrow\lr{P,\mc{L}},\beta:M_0\oplus M\rightarrow \O{P}, \gamma_1,\ldots,\gamma_d}$ which is good at $\ov{x}$. (We assume that it can be taken globally.) Then there exists a natural \'etale morphism
  \[f:\lr{P,\mc{L}} \rightarrow \tAA^d\times\lr{\tA{M_0}\fibpro{\tA{N}}\OKa}\times \tA{M}\]
  for some $d$.
  Let $Q$ be the inverse image of the vertex of $\tA{M}$ under the map $P\rightarrow \tA{M}$ and $Z\coloneqq Q_k \subseteq \ov{X}$. Then $\ov{x}\in Z$.
  Note that $Q\rightarrow \tAA^d\times\lr{\tA{M_0}\fibpro{\tA{N}}\Spf\,\OK}$ is \'etale.
  Let $\mc{L}_0$ be the log structure on $Q$ over $\OKa$ induced by $Q\rightarrow \tA{M_0}\fibpro{\tA{N}}\Spf\,\OK$.

  $X$ is dense in $\ov{X}$ since it is the inverse image of a dense subset $\tAA^d\times\lr{\tA{M_0}\fibpro{\tA{N}}\OKa}\times \tA{M^{\mr{gp}}}$ under the \'etale morphism $f$.

  Note that $\tube{Z}_Q=Q_K$.

  Note also that $M$ is semi-saturated. Indeed, for any face $F$ of $M$, there exists a geometric point $\ov{y}$ of $\ov{X}$
  such that $\ov{\mc{M}}_{\ov{y}}=M/F$. By the assumption, $\lr{M/F}^{\mr{gp}}$ is torsion-free.
  So $M$ is semi-saturated by Proposition \ref{semi-saturated}.
  
  \begin{lem}\label{tubeZ}
    In this situation,
    \[\tube{Z}_P\cong\tube{Z}_Q\times \A{M}{K}\!\clop{0,1}.\]
  \end{lem}
  \begin{proof}
  Let $\lr{P\times Q,\mc{L}_{P\times Q}}\coloneqq \lr{P,\mc{L}}\fibpro{\OKa}\lr{Q,\mc{L}_0}$.
  Let $\pi_1:\lr{P\times Q,\mc{L}_{P\times Q}}\rightarrow \lr{P,\mc{L}}$ and $\pi_2:\lr{P\times Q,\mc{L}_{P\times Q}}\rightarrow \lr{Q,\mc{L}_0}$ be the first and second projections. $\lr{\cdot}_i$ denotes $\pi^*_i\op{\cdot}$ for $i=1,2$.
  Let $\ov{m}'_1,\ldots,\ov{m}'_{r'}\in M_0^{\mr{gp}}/N^{\mr{gp}}$ be a set of free generators of $M_0^{\mr{gp}}/N^{\mr{gp}}$
  and let $m'_1,\ldots,m'_{r'}\in M_0^{\mr{gp}}$ be their lifts.
  Let $m_1,\ldots,m_{r}\in M^{\mr{gp}}$ be a set of free generators of $M^{\mr{gp}}$.
  Let $M_0\op{1}\coloneqq M_0\amsum{N}M_0$ and $M_0\op{1}'$ the submonoid of $M_0\op{1}^{\mr{gp}}$ generated by $M_0\op{1}$
  and the kernel of the codiagonal homomorphism $M_0\op{1}^{\mr{gp}}\rightarrow M_0^{\mr{gp}}$.
  Let $\lr{P',\mc{L}'}\coloneqq \lr{P\times Q,\mc{L}_{P\times Q}}\fibpro{\tA{M_0\op{1}}}\tA{M_0\op{1}'}$.
  Let $\iota: \lr{P',\mc{L}'}\rightarrow \lr{P\times Q,\mc{L}_{P\times Q}}$ be the projection.
  There exists a natural closed immersion $\lr{Z,\mc{M}|_Z}\hookrightarrow \lr{P',\mc{L}'}$.

  
  Since $\iota^*\op{\frac{\beta\op{m'}_2}{\beta\op{m'}_1}}$ is an invertible element
  of $\O{P'}$ for any $m' \in M_0$, $M_0\oplus M\ni m'\oplus m \mapsto \beta\op{m'\oplus m}_1$ is a fine chart of $\mc{L}'$. So $\lr{Z,\mc{M}|_Z}\hookrightarrow \lr{P',\mc{L}'}$ is an exact closed immersion.

  We define a morphism
  \[\lr{P',\mc{L}'}\rightarrow \lr{P,\mc{L}}\times \tAA^{d+r'}\]
  by the projection $\pi_1\circ\iota$ and 
  elements $\iota^*\op{\lr{\gamma_1}_2-\lr{\gamma_1}_1},\ldots,\iota^*\op{\lr{\gamma_d}_2-\lr{\gamma_d}_1}$,
  $\iota^{*}\op{\frac{\beta\op{m'_1}_2}{\beta\op{m'_1}_1}}-1,\ldots,\iota^{*}\op{\frac{\beta\op{m'_{r'}}_2}{\beta\op{m'_{r'}}_1}}-1$ of $\O{P'}$ regarded as morphisms
  $P'\rightarrow \tAA^1$. Then this is \'etale (since log \'etale and strict)
  and $\lr{Z,\mc{M}|_Z}\hookrightarrow \lr{P',\mc{L}'}\rightarrow \lr{P,\mc{L}}\times \tAA^{d+r'}$ is the closed immersion to $\lr{Z,\mc{M}|_Z}\times \mb{0}$ where $\mb{0}$ is the vertex of
  $\tAA^{d+r'}$.
  So,
  \[\tube{Z}_{P'}\cong \tube{Z}_{P\times \tAA^{d+r'}}
  =\tube{Z}_P\times \AA{K}^{d+r'}\!\clop{0,1}.\]
  
  On the other hand, we define a morphism
  \[\lr{P',\mc{L}'}\rightarrow \tA{M} \times \tAA^{d+r'}\times \lr{Q,\mc{L}_0}\]
  by $\lr{P',\mc{L}'}\xrightarrow{\pi_1\circ\iota} \lr{P,\mc{L}}\rightarrow \tA{M}$, elements $\iota^*\op{\lr{\gamma_1}_2-\lr{\gamma_1}_1},\ldots,\iota^*\op{\lr{\gamma_d}_2-\lr{\gamma_d}_1}$,
  $\iota^*\op{\frac{\beta\op{m'_1}_2}{\beta\op{m'_1}_1}}-1,\ldots,\iota^*\op{\frac{\beta\op{m'_{r'}}_2}{\beta\op{m'_{r'}}_1}}-1$ of $\O{P'}$ regarded as morphisms
  $P'\rightarrow \tAA^1$ and the projection $\pi_2\circ\iota$.
  Then this is also \'etale and $\lr{Z,\mc{M}|_Z}\hookrightarrow \lr{P',\mc{L}'}\rightarrow\tA{M} \times \tAA^{d+r'}\times \lr{Q,\mc{L}_0}$ is the closed immersion to
  $\mb{0}\times\mb{0}\times \lr{Z,\mc{L}_0|_Z}$ where the first $\mb{0}$ is the vertex of $\tA{M}$ and the second $\mb{0}$ is the vertex of $\tAA^{d+r'}$.
  So,
  \[\tube{Z}_{P'}\cong \tube{Z}_{\tA{M} \times \tAA^{d+r'}\times Q}
  =\A{M}{K}\!\clop{0,1}\times \AA{K}^{d+r'}\!\clop{0,1}\times \tube{Z}_Q.\]
  
  Comparing two isomorphisms, we have $\tube{Z}_P\cong\tube{Z}_Q\times \A{M}{K}\!\clop{0,1}$.
  \end{proof}
  \begin{lem}\label{6.3.4}(cf. Proposition 2.2.13 of \cite{Ber} or Lemma 6.3.4 of \cite{Ked})
    Let $E$ be a convergent log $\nabla$-module on $P_K$. Then the restriction of $E$ to $\tube{Z}_P\cong\tube{Z}_Q\times \A{M}{K}\!\clop{0,1}$ is log-convergent.
  \end{lem}
  \begin{proof}
    Let $\lr{P\op{1},\mc{L}\op{1}}$ be the fiber product of $2$ copies of $\lr{P,\mc{L}}$ over $\OKa$.
    Let $\pi_1,\pi_2: \lr{P\op{1},\mc{L}\op{1}}\rightarrow \lr{P,\mc{L}}$ be the two projections.  Put $\lr{\cdot}_i\coloneqq \pi^{*}_i\op{\cdot}$ for $i=1,2$.

    Let $\lr{M\oplus M_0}\op{1}\coloneqq \lr{M\oplus M_0}\amsum{N}\lr{M\oplus M_0}$ and $\lr{M\oplus M_0}\op{1}'$ the submonoid of $\lr{M\oplus M_0}\op{1}^{\mr{gp}}$ generated by $\lr{M\oplus M_0}\op{1}$
    and the kernel of the codiagonal homomorphism $\lr{M\oplus M_0}\op{1}^{\mr{gp}}\rightarrow \lr{M\oplus M_0}^{\mr{gp}}$.
    Let $\lr{P\op{1}',\mc{L}\op{1}'}\coloneqq \lr{P\op{1},\mc{L}\op{1}} \fibpro{\tA{\lr{M\oplus M_0}\op{1}}} \tA{\lr{M\oplus M_0}\op{1}'}$. Then $\tube{\ov{X}}_{\lr{P\op{1},\mc{L}\op{1}}}^{\mr{log}}=\tube{\ov{X}}_{P\op{1}'}$ as in the proof of Proposition \ref{well-def_exp}.

    Let $m_1,\ldots m_r$ and $m'_1,\ldots m'_{r'}$ be as in the previous proof and 
    let $m_{r+i}\coloneqq m'_{i}$ for $1\leq i \leq r'$.
    
    By the proof of Proposition \ref{well-def_exp},
    $\Omega^{\mr{log},1}_{\lr{P\op{1}',\mc{L}\op{1}'}/\OKa}$ is generated by $d\lr{\gamma_1}_1,\ldots,d\lr{\gamma_d}_1$, $d\lr{\gamma_1}_2,\ldots,d\lr{\gamma_d}_2$, $d\lr{\frac{\beta\op{m_1}_1}{\beta\op{m_1}_2}},\ldots,d\lr{\frac{\beta\op{m_{r+r'}}_1}{\beta\op{m_{r+r'}}_2}}$ and $d\log \lr{m_1}_2,\ldots,d\log \lr{m_{r+r'}}_2$.
    Let $u_k\coloneqq\frac{\beta\op{m_k}_1}{\beta\op{m_k}_2}$ for $1\leq k \leq r+r'$ and $\tau_l\coloneqq \lr{\gamma_l}_2-\lr{\gamma_l}_1$ for $1\leq l \leq d$.

    By the definition of convergent log $\nabla$-modules, there exists a locally free convergent log isocrystal $\mc{E}$ on $\lr{X,\ov{X},\mc{M}}$  over $\OKa$ which coincides with $E$ on some strict neighborhood of $\tube{X}_{P\op{1}'}$. We can regard $\mc{E}$ as a locally free module on some strict neighborhood $V$ of $\tube{X}_{P}$ equipped with an isomorphism $\epsilon:\pi_2^*\op{\mc{E}}\rightarrow \pi_1^*\op{\mc{E}}$ on some strict neighborhood $W$ of $\tube{X}_{P\op{1}'}$ contained in $\pi_1^{-1}\op{V}\cap\pi_2^{-1}\op{V}$.

    \begin{claim}
    The isomorphism $\epsilon$ is written by
    \[\epsilon\op{1\tens{}\mb{v}}=\sum_{i_1,\ldots,i_{r+r'+d}=0}^{\infty}\lr{\prod_{k=1}^{r+r'}\frac{\lr{u_k-1}^{i_k}}{i_k!}\prod_{l=1}^d\frac{\tau_l^{i_{r+r'+l}}}{i_{r+r'+l}!}
      \tens{}\lr{\prod_{k=1}^{r+r'}\prod_{x=0}^{i_k-1}\lr{\partial_i-x}\prod_{l=1}^d\lr{\frac{\partial}{\partial \gamma_l}}^{i_{r+l}}}\op{\mb{v}}}\]
    for any section $\mb{v}$ of $E$.
    \end{claim}
    \begin{proof}
    Let $\mc{I}\coloneqq \Ker\op{\mc{O}_{P\op{1}'}\rightarrow \mc{O}_P}$ and define the $n$-th log infitesimal neighborhood $P^n\op{1}$ of $P$ by $\mc{O}_{P^n\op{1}}\coloneqq \mc{O}_{P\op{1}'}/\mc{I}^{n+1}$ for $n \in \N$. Let $\epsilon^n$ be the restriction of $\epsilon$ to $W \cap P^n\op{1}$. We define $\varepsilon^n:\pi_2^*\op{\mc{E}}|_{W\cap P^n\op{1}}\rightarrow \pi_1^*\op{\mc{E}}|_{W\cap P^n\op{1}}$ as
    \[\varepsilon^n\op{1\tens{}\mb{v}}=\sum_{i_1+\cdots+i_{r+r'+d}\leq n}\lr{\prod_{k=1}^{r+r'}\frac{\lr{u_k-1}^{i_k}}{i_k!}\prod_{l=1}^d\frac{\tau_l^{i_{r+r'+l}}}{i_{r+r'+l}!}
      \tens{}\lr{\prod_{k=1}^{r+r'}\prod_{x=0}^{i_k-1}\lr{\partial_i-x}\prod_{l=1}^d\lr{\frac{\partial}{\partial \gamma_l}}^{i_{r+l}}}\op{\mb{v}}}.\]
    Then $\epsilon^1$ and $\varepsilon^1$ coincide on $W\cap P^1\op{1}$.

    By the cocycle condition, for $n \in \N$, the following diagram commutes:
    \begin{tikzcd}
      \mc{O}_{P^{n}\op{1}}\tens{\mc{O}_P} \mc{O}_{P^n\op{1}}\tens{\mc{O}_P}\mc{E}\arrow{r}{\id\tens{}\epsilon^n} \arrow[equal]{dd}& \mc{O}_{P^{n}\op{1}}\tens{\mc{O}_P}\mc{E}\tens{\mc{O}_P} \mc{O}_{P^n\op{1}}\arrow{d}{\epsilon^n\tens{}\id} \\
      & \mc{E}\tens{\mc{O}_P} \mc{O}_{P^n\op{1}}\tens{\mc{O}_P} \mc{O}_{P^n\op{1}}\arrow[equal]{d}\\
      \mc{O}_{P^{n}\op{1}}\tens{\mc{O}_P} \mc{O}_{P^n\op{1}}\tens{\mc{O}_{P^{2n}\op{1}}}\lr{\mc{O}_{P^{2n}\op{1}}\tens{\mc{O}_P}\mc{E}} \arrow{r}{\id\tens{}\epsilon^{2n}} &
      \mc{O}_{P^{n}\op{1}}\tens{\mc{O}_P} \mc{O}_{P^n\op{1}}\tens{\mc{O}_{P^{2n}\op{1}}}\lr{\mc{E}\tens{\mc{O}_P}\mc{O}_{P^{2n}\op{1}}}\\
    \end{tikzcd}
    where $\mc{O}_{P^{2n}\op{1}}\rightarrow \mc{O}_{P^{n}\op{1}}\tens{\mc{O}_P} \mc{O}_{P^n\op{1}}$ is the natural map. This map is injective.

    The above diagram also commutes even if we replace $\epsilon^n,\epsilon^{2n}$ by $\varepsilon^n,\varepsilon^{2n}$, respectively. Thus if $\epsilon^n$ and $\varepsilon^n$ coincide, $\epsilon^{2n}$ and $\varepsilon^{2n}$ coincide and so $\epsilon^{m}$ and $\varepsilon^{m}$ coincide for $m\leq 2n$.
    By induction, $\epsilon^n$ and $\varepsilon^n$ coincide for all $n\in \N$.
    \end{proof}

    By the proof of Proposition \ref{well-def_exp},
    \[\tube{\ov{X}}_{P\op{1}'}=\AA{K}^{r+r'+d}\clop{0,1}\times\tube{\ov{X}}_{P}\]
    where the coodinates of $\AA{K}^{r+r'+d}\clop{0,1}$ are $u_1-1,\ldots,u_{r+r'}-1,\tau_1,\ldots,\tau_{d}$.
    Thus for any $\eta < 1$ and any $\mathbf{v} \in \Gamma\op{\tube{\ov{X}}_{P},E}$, the multi-indexed series
    \[\lr{\frac{1}{i_1!\cdots i_r!}\lr{\prod_{k=1}^{r+r'}\prod_{x=0}^{i_k-1}\lr{\partial_i-x}}\op{\mb{v}}}_{i_1,\ldots,i_r}\]
    is $\eta$-null on $\tube{X}_P$.  Since $X$ is dense $\ov{X}$, the spectral seminorm on $\Gamma\op{\tube{X}_P,\mc{O}_P}$ restricts to the spectral seminorm on $\Gamma\op{\tube{\ov{X}}_P,\mc{O}_P}$. Thus this series is $\eta$-null on $\tube{\ov{X}}_P$ .

    For any $a<1$, sections of $E$ on $\tube{\ov{X}}_P$ generate all sections of $E$ on $\tube{Z}_Q\times \A{M}{K}\!\brack{0,a}\subseteq \tube{Z}_Q\times \A{M}{K}\!\clop{0,1}\cong\tube{Z}_P\subseteq \tube{\ov{X}}_P$.
    So, by Remark \ref{log-congergence-generator}, the restriction of $E$ to $\tube{Z}_P\cong\tube{Z}_Q\times \A{M}{K}\!\clop{0,1}$ is log-convergent.
  \end{proof}

  Put $R\coloneqq \tAA^d\times\lr{\tA{M_0}\fibpro{\tA{N}}\OKa}$
  and let $f_M:\lr{P,\mc{L}}\rightarrow \tA{M}$ be the composition of $f$ and the projection $R\times \tA{M}\rightarrow \tA{M}$.
    
  Let $F\subseteq M$ be a face. There exists a natural closed immersion $\tA{F}\hookrightarrow \tA{M}$
  of underlying formal schemes defined by
  $\OK\!\angbra{M} \ni \sum_{m\in M}c_mt^m\mapsto \sum_{m\in F}c_mt^m\in \OK\!\angbra{F}$.
  Put
  \begin{align*}
    Q_F&\coloneqq f_M^{-1}\op{\tA{F}}=\Set{p \in P| \forall m \in M\setminus F: \beta\op{m}\op{p}=0},\\
    \tilde{P}_F&\coloneqq f_M^{-1}\op{\tA{F^{-1}M}}=\Set{p\in P|\forall m\in F: \beta\op{m}\op{p}\neq 0},\\
    \tilde{Q}_F&\coloneqq Q_F\cap \tilde{P}_F.
  \end{align*}
  $Q_F$ is closed
  and $\tilde{P}_F$ is open in $P$. Let $\mc{L}_F$ be the log structure on $Q_F$ over $\OKa$ induced by $Q_F\rightarrow \lr{\tA{M_0}\fibpro{\tA{N}}\Spf\,\OK}\times \tA{F}$.
  Note that $P=\sqcup_{F} \tilde{Q}_F$ where $F$ runs through all faces of $M$. Put $Z_F\coloneqq \lr{Q_F}_k$,
  $\tilde{X}_F\coloneqq \lr{\tilde{P_F}}_k$ and $\tilde{Z}_F\coloneqq \lr{\tilde{Q}_F}_k$. Note that $X=\tilde{X}_M$.

  Applying the Lemma \ref{monoid_section} to $M\rightarrow M/F$, we have an isomorphism $F^{-1}M\cong M/F\oplus F^{\mr{gp}}$. Take an isomorphism $F^{\mr{gp}}\cong \Z^{r_F}$ where
  $r_F$ is the rank of $F^{\mr{gp}}$. Then the morphism $f$ and these isomorphisms
  induce an \'etale morphism
  \begin{equation}\label{induced chart}
    \lr{\tilde{P}_F,\mc{L}|_{\tilde{P}_F}}\rightarrow R\times \tAA^{r_F}\times \tA{M/F}=\tAA^{d+r_F}\times\lr{\tA{M_0}\fibpro{\tA{N}}\OKa}\times \tA{M/F}.
  \end{equation}
  The lifting $\lr{\tilde{X}_F,\lr{\mc{M}_0\oplus\mc{M}}|_{\tilde{X}_F}}\hookrightarrow \lr{\tilde{P}_F,\mc{L}|_{\tilde{P}_F}}$, the chart induced by $\lr{\tilde{P}_F,\mc{L}|_{\tilde{P}_F}} \rightarrow \lr{\tA{M_0}\fibpro{\tA{N}}\OKa}\times \tA{M/F}$ and the sections of $\O{\tilde{P}_F}$ inducing $\tilde{P}_F\rightarrow \tAA^{d+r_F}$ forms a charted standard small frame which is good at points of $\tilde{Z}_F$.
  Thus by Lemma \ref{tubeZ},
  \begin{equation}\label{ZPcong}
    \tube{\tilde{Z}_F}_P\cong\tube{\tilde{Z}_F}_{Q_F}\times \A{M/F}{K}\!\clop{0,1}.
  \end{equation}
  Note that $\tube{\tilde{Z}_F}_{Q_F}=\lr{\tilde{Q}_F}_K$.
  
  An overconvergent log isocrystal $\mc{E}$ on $\lr{X,\ov{X},\mc{M}_0\oplus\mc{M}}/\OKa$ defines a log $\nabla$-module $E_F$ on $\tube{\tilde{Z}_F}_{Q_F}\times \A{M/F}{K}\lr{a,1}$ for some $a \in \lr{0,1}\cap\Gamma^{*}$.
  \begin{defn}
    Let $\Sigma\subseteq M^{\mr{gp}}\tens{\Z}\Z_p$ be an ($\lr{\mr{NI}\cap\mr{NL}}$-D) subset. $\mc{E}$ is \textit{$\Sigma$-unipotent with respect to $\lr{P,\mc{L}}$} if for any face $F$ of $M$, the induced log $\nabla$-module on $\tube{\tilde{Z}_F}_{Q_F}\times \A{M/F}{K}\lr{a,1}$ is $\Sigma_F$-unipotent, where $\Sigma_F$ is the image of $\Sigma$ under the projection $M^{\mr{gp}}\tens{\Z}\Z_p\twoheadrightarrow \lr{M^{\mr{gp}}/F^{\mr{gp}}}\tens{\Z}\Z_p$.
  \end{defn}
  \begin{rem}
    By Proposition \ref{UULNM}, this definition is independent of the choice of $a$.
  \end{rem}
  \begin{rem}\label{Sigma_F unipotence}
    If $\E$ is $\Sigma$-unipotent with respect to a charted standard small frame $\lr{P,\mc{L}}$,
    then $\E$ is also $\Sigma_F$-unipotent with respect to the charted standard small frame of 
    $\lr{\tilde{X}_F,\mc{M}_0\oplus\mc{M}|_{\tilde{X}_F}}$ induced by (\ref{induced chart})
    for any face $F$ of $M$. Indeed, any face $\ov{F'}$ of $M/F$ corresponds with a face $F'$ of $M$ containing $F$
    and $\tube{\tilde{Z}_{\ov{F'}}}_{\tilde{P}_F}=\tube{\tilde{Z}_{F'}}_{P}$.
   \end{rem}
  \begin{rem}\label{unipotence/Z2}
    By Remark \ref{unipotence/Z}, $\Sigma$-unipotence of overconvergent log isocrystals only depends on
    the image of $\Sigma$ in $M^{\mr{gp}}\tens{\Z}\lr{\Z_p/\Z}$.
  \end{rem}
  
\begin{rem}
  In the subsection \ref{extension}, we will prove that the definition is independent of the choice of charted standard small frames
  good at $\ov{x}$ under some assumptions.
\end{rem}

\subsection{Overconvergent generization}

In this subsection, we adapt Proposition 3.5.3 of \cite{Ked} or Proposition 2.7 of \cite{Shi} to our situation.

\begin{prop}\label{overconvergent-generization}(overconvergent generization, cf. Proposition 2.7 in \cite{Shi})
  Let $\lr{P,\mc{L}_0\oplus\mc{L}}$ be an affine connected $p$-adic fine log formal scheme log smooth over $\OKa$
  such that the structure morphism $N\rightarrow \mc{L}_0\oplus\mc{L}$ factors through
  $\mc{L}_0\rightarrow \mc{L}_0\oplus\mc{L}$.
  Let $M'\rightarrow \O{P}$ be a fine chart of $\mc{L}$.
  Assume that $\mc{L}_0$ is trivial on $P_K$.
  Let $X$ be the trivial locus of $\mc{L}|_{P_k}$.
  Let $M$ be a fine sharp semi-saturated monoid.
  Let $\Sigma \subseteq M^{\mr{gp}}\tens{\Z}\Z_p$ be an ($(\mr{NI}\cap\mr{NL})$-D) subset
  and $I \subseteq \lr{0,1}$ a quasi-open subinterval of positive length.
  Let $V$ be a strict neighborhood of $\tube{X}_P$
  in $P_K$ and $E$
  an object of $\LNM_{V\times \A{M}{K}\op{I},\Sigma}$ whose restriction to $\tube{X}_P\times \A{M}{K}\op{I}$
  is $\Sigma$-unipotent.
  Then, for any closed aligned subinterval $\brack{b,c}\subseteq I$ of positive length, there exists a
  strict neighborhood $V'$ of $\tube{X}_P$ in $P_K$ contained in $V$, such that the restriction of $E$ to 
  $V'\times \A{M}{K}\!\brack{b,c}$ is $\Sigma$-unipotent.
\end{prop}

\begin{proof}
  This proof is essentially the same as the proof of Proposition 2.7 in \cite{Shi}.

  We may assume that $V$ is affinoid.
  Assume that $E$ is $\Sigma$-unipotent on $\tube{X}_P\times \A{M}{K}\op{I}$.
  Let $\brack{b',c'}\subseteq\brack{d,e}\subseteq I$ be aligned closed subintervals such that $d<b'< b$, $c <c'<e$.
  We define $D_l$ as in the proof of Lemma \ref{2.3}.
  As shown in the proof of Lemma \ref{2.3} and in the proof of Proposition \ref{transfer},
  for some $\xi\in \Sigma$ and some $\eta>1$,
  for any $\mb{v}\in \Gamma\op{V\times \A{M}{K}\!\brack{d,e},E}$,
  $\brace{\lr{D_{l+1}-D_l}\op{\mb{v}}}$ is $\eta$-null
  on $\tube{X}_P\times \A{M}{K}\!\brack{b',c'}$ and
  $\brace{D_l\op{\mb{v}}}$ converges to an element $f\op{\mb{v}}$ of
  $H^0_{\xi}\op{\tube{X}_P\times\A{M}{K}\!\brack{b',c'},E}$.
  Moreover, for some $\mb{v}\in \Gamma\op{V\times \A{M}{K}\!\brack{d,e},E}$,
  $f\op{\mb{v}}\neq 0$.

  Let $m'_1,\ldots,m'_g\in M'$ be a set of generators of $M'$.
  For $\lambda\in\lr{0,1}\cap \Gamma^*$, put $V_{\lambda}\coloneqq \Set{x\in P_K|\abs{t^{m'_1+\cdots+m'_g}\op{x}}\geq\lambda}$.
  Let $W\subseteq V\times \A{M}{K}\!\brack{b',c'}$ be an affinoid subspace
  such that $E$ is free on $W$. We can show that $\brace{\lr{D_{l+1}-D_l}\op{\mb{v}}}$ is $\rho$-null for some $\rho>0$ on $W \cap \lr{V_{\lambda}\times \A{M}{K}\!\brack{b',c'}}$ for some $\lambda\in\lr{0,1}\cap \Gamma^*$ by the same calculation as in the proof
  of Proposition 2.7 in \cite{Shi}. If $W\cap \lr{\tube{X}_P\times \A{M}{K}\!\brack{b',c'}}=\emptyset$, by the maximum modulus principle, $W \cap \lr{V_{\lambda}\times \A{M}{K}\!\brack{b',c'}}=\emptyset$ for some $\lambda\in\lr{0,1}\cap \Gamma^*$.
  Otherwise, by Proposition 3.5.2 of \cite{Ked},
  there exists $\lambda\in\lr{0,1}\cap \Gamma^*$ such that $\brace{\lr{D_{l+1}-D_l}\op{\mb{v}}}$ is $1$-null on $W\cap \lr{V_{\lambda}\times \A{M}{K}\!\brack{b',c'}}$. 

  So we can take $\lambda\in\lr{0,1}\cap \Gamma^*$ such that
  $\brace{D_l\op{\mb{v}}}$ converges over $V_\lambda\times \A{M}{K}\!\brack{b',c'}$.
  Put $H_E\coloneqq H^0_{\xi}\op{V_{\lambda}\times \A{M}{K}\!\brack{b',c'}, E}\neq 0$.
  Then 
  as shown in the proof of Proposition \ref{generization}, $H_E$ can be regarded as an object of $\LNM_{V_{\lambda}\times \A{M}{K}\!\brack{0,0}}$.
  Let $F\coloneqq \mc{U}_{\brack{b',c'}}\op{H_E}$. Then $F$ is a $\Sigma$-constant subobject of $E|_{V_{\lambda}\times \A{M}{K}\!\brack{b',c'}}$
  as proven in the proof of Proposition \ref{generization}.
  By induction of the rank of $E$, we can assume that $E/F$ is $\Sigma$-unipotent on $V_{\lambda'}\times \A{M}{K}\!\brack{b,c}$
  for some $\lambda'\leq\lambda$. Hence $E$ is $\Sigma$-unipotent on $V_{\lambda'}\times \A{M}{K}\!\brack{b,c}$.
\end{proof}

\subsection{Extension of overconvergent log isocrystals}\label{extension}

  In this section, we prove that unipotent overconvergent log isocrystals can be
  extended to convergent log isocrystals. The well-definedness of unipotence is proven by this extension property.
  
  Let $\lr{\ov{X},\mc{M}_0\oplus\mc{M}}$ and $X$ be as in Situation \ref{situation}.
  Let $\ov{x}$ be some geometric point of $\ov{X}$.
  Put $M_0\coloneqq\lr{\ov{\mc{M}_0}}_{\ov{x}}$ and $M\coloneqq\ov{\mc{M}}_{\ov{x}}$.
  We assume that there exists a charted standard small frame \[\lr{\lr{P,\mc{L}},M_0\oplus M\rightarrow \O{P},\gamma_1,\ldots,\gamma_d}\]
  which is good at $\ov{x}$.
  We continue to use the symbols defined in \S \ref{unipotence}.
  Note that $P_K=\bigsqcup_F \tube{\tilde{Z}_F}_P$ by definition,  where $F$ runs throught all faces of $M$. We show that this covering can be enlarged to an admissible covering.

\begin{lem}\label{admissible covering}
  For each $F$, there exists a strict neighborhood $V_F$ of $\tube{\tilde{Z}_F}_P$ in $\tube{Z_F}_P$
  and an isomorphism $V_F\cong V_F'$ which is an extension of (\ref{ZPcong}),
  where $V_F'$ is some strict neighborhood of $\tube{\tilde{Z}_F}_{Q_F}\times \A{M/F}{K}\!\clop{0,1}$
  in $\lr{Q_F}_K\times \A{M/F}{K}\!\clop{0,1}$. Moreover, $P_K=\bigcup_F V_F$ is an admissible covering.
\end{lem}
\begin{proof}
  If $V_F$ is a strict neighborhood of $\tube{\tilde{Z}_F}_P$ in $\tube{Z_F}_P$,
  the covering
  \[\tube{Z_F}_P=V_F\cup \bigcup_{F'\subsetneq F} \tube{Z_{F'}}_P\]
  is admissible, where $F'$ runs through proper faces of $F$.
  Thus, if we take $V_F$ for each face $F$ of $M$, the covering $P_K=\bigcup_{F}V_F$ is admissible.
  
  Let $F$ be a face of $M$.
  Take a section $s:M/F\rightarrow F^{-1}M$ of the natural projection $M\twoheadrightarrow M/F$
  as Lemma \ref{monoid_section} and let $M_s\coloneqq M+\Im\,s\subseteq F^{-1}M$.
  The morphism $\tA{M_s}\rightarrow \tA{F}\times\tA{M/F}$ which is defined by the inclusion $F\rightarrow M$ and $s$ is log \'etale
  since $F^{\mr{gp}}\oplus \lr{M/F}^{\mr{gp}}\cong M^{\mr{gp}}$.
  By Lemma \ref{monoid_section}, $M_s\cap F^{\mr{gp}}=F$. $F$ is face of $M_s$. Indeed, if $a+b\in F$ for $a,b\in M_s$, then
  $a,b\in F^{\mr{gp}}$ because $F^{\mr{gp}}$ is a face of $F^{-1}M$, so $a,b\in M_s$.
  Thus there exists a natural closed immersion $\tA{F}\rightarrow \tA{M_s}$
  of underlying formal schemes and $\tA{F}\rightarrow \tA{M_s}\rightarrow \tA{M/F}$ is the morphism to the vertex of $\tA{M/F}$.
  Let $\lr{P_s,\mc{L}_s}\coloneqq \lr{P,\mc{L}}\fibpro{\tA{M}}\tA{M_s}$ which is log \'etale over $\lr{P,\mc{L}}$ and contains $Q_F$. Let
  \[\lr{P_s',\mc{L}'_s}\coloneqq \lr{P_s,\mc{L}_s}\fibpro{R\times\tA{F}\times\tA{M/F}}\lr{\lr{Q_F,\mc{L}_F}\times\tA{M/F}}.\]

  Note that $F^{-1}M=F^{-1}M_s\cong F^{\mr{gp}}\oplus M/F$.
  Then the following diagram commutes:
  \[\begin{tikzcd}
  &\lr{Q_F,\mc{L}|_{Q_F}}\arrow[hook]{r}& \lr{P,\mc{L}}\arrow{r}{f} &R\times\tA{M}\\
  \lr{\tilde{Q}_F,\mc{L}|_{\tilde{Q}_F}}\arrow[hook]{ru}\arrow[hook]{r} \arrow[hook]{rd} \arrow[hook]{rdd}
  & \lr{Q_F,\mc{L}_s|_{Q_F}}\arrow{u}\arrow[hook]{r}& \lr{P_s,\mc{L}_s}\arrow{u}\arrow{r}{f} &R\times \arrow[hook]{u}\tA{M_s}\arrow{dd}\\
  & \lr{Q_F,\mc{L}_s'|_{Q_F}}\arrow{u}\arrow{d}\arrow[hook]{r} & \lr{P_s',\mc{L}'_s}\arrow{u}\arrow{ru}\arrow{d}&\\
  & \lr{Q_F,\mc{L}_{F\times M/F}}\arrow[hook]{r}& \lr{Q_F,\mc{L}_F}\times\tA{M/F}\arrow{r}{\lr{f,\id}}&R\times \tA{F}\times\tA{M/F},
  \end{tikzcd}\]
  where $Q_F\hookrightarrow Q_F\times \tA{M/F}$ is the closed immersion to the vertex of $\tA{M/F}$ and
  $\mc{L}_{F\times M/F}$ is the log structure on $Q_F$ which is the pullback by this morphism.
  
  $\lr{P_s',\mc{L}_s'}\rightarrow \lr{P,\mc{L}}$ and $\lr{P_s',\mc{L}_s'}\rightarrow \lr{Q,\mc{L}_F}\times \tA{M/F}$ are log \'etale, so \'etale on some neighborhood of $\tilde{Q}_F$.
  Thus, by the strong fibration lemma, there exists an isomorphism between some strict neighborhood of $\tube{\tilde{Z}_F}_P$
  in $\tube{Z_F}_Q$ and some strict neighborhood of $\tube{\tilde{Z}_F}_{\tilde{Q}_F\times \tA{M/F}}=\lr{\tilde{Q}_{F}}_{K}\times \A{M/F}{K}\!\clop{0,1}$
  in $\tube{Z_F}_{Q_F\times \tA{M/F}}=\lr{Q_F}_K\times \A{M/F}{K}\!\clop{0,1}$.
\end{proof}

\begin{prop}\label{local extension}
  We also assume that $N\rightarrow \lr{\ov{\mc{M}_0}}_{\ov{x}}$
  is vertical for any geometric point $\ov{x}$ of $\ov{X}$.
  Let $\Sigma\subseteq M^{\mr{gp}}\tens{\Z}\Z_p$ be an ($\lr{\mr{NI}\cap\mr{NL}}$-D) subset.
  The category of overconvergent log isocrystals
  over $\lr{X,\ov{X},\mc{M}_0\oplus\mc{M}}/\OKa$
  which is $\Sigma$-unipotent with respect to $\lr{P,\mc{L}}$ is equivalent to the
  category of convergent log $\nabla$-modules on $P_K$ which has exponents in $\Sigma$.
\end{prop}
\begin{proof}
    We continue to use the symbols defined in the proof of Lemma \ref{admissible covering}.

    Let $\E$ be a $\Sigma$-unipotent overconvergent log isocrystal on $\lr{X,\ov{X},\mc{M}_0\oplus\mc{M}}$.
    Let $\Sigma_F$ be the image of $\Sigma$ under $M^{\mr{gp}}\tens{\Z}\Z_p\rightarrow \lr{M/F}^{\mr{gp}}\tens{\Z}\Z_p$.
    $\E$ defines a log $\nabla$-module $E_F$ on $V_F' \times \A{M/F}{K}\!\clop{a,1}$ for some $a\in \Gamma^*\lr{0,1}$,
    where $V_F'$ is some strict neighborhood of $\lr{\tilde{Q}_F}_K$ in $\lr{Q_F}_K$.
    By the assumption, the restriction of $E_F$ on $\lr{\tilde{Q}_F}_K\times \A{M/F}{K}\!\clop{a,1}$ is $\Sigma_F$-unipotent.
    Since the log structure on $\Spm\,K$ induced by $\lr{\Spf\,\OK,N}$ is trivial and $N\rightarrow \lr{\ov{\mc{M}_0}}_{\ov{x}}$ is vertical for any geometric point $\ov{x}$ of $\ov{X}$, $\mc{M}_0$ is trivial on $\lr{Q_F}_K$.
    By overconvergent generization (Proposition \ref{overconvergent-generization}), for any $b\in\lr{a,1}\cap \Gamma^*$,
    there exists a strict neighborhood $V_b\subseteq V$ of $\lr{\tilde{Q}_F}_{K}$ in $\lr{Q_F}_K$ such that
    the restriction of $E_F$ to $V_b\times \A{M/F}{K}\!\brack{a,b}$ is $\Sigma_F$-unipotent.
    By Proposition \ref{UULNM}, this can be uniquely extended to a $\Sigma_F$-unipotent $\nabla$-module on $V_b\times \A{M/F}{K}\!\clop{0,b}$.
    Take an increasing sequence $\ov{b}=\lr{b_i}_{i\in \N}$
    such that $b_i\rightarrow 1$ when $i\rightarrow \infty$. Let $V_{\ov{b}}\coloneqq \bigcup_i V_{b_i}\times \A{M/F}{K}\!\clop{0,b_i}$.
    Then $E_F$ can be uniquely extended to $V_{\ov{b}}$, which is a strict neighborhood of $\lr{\tilde{Q}_F}_{F}\times \A{M/F}{K}\!\clop{0,1}$
    in $\lr{Q_F}_K\times \A{M/F}{K}\!\clop{0,1}$. So $\E$ can be uniquely extended to some strict neighborhood $V_F$ of $\tube{\tilde{Z}_F}_{P}$
    in $\tube{Z_F}_{P}$.

    To prove that $\E$ can be extended to $P_K$, we need to show that the extensions of $\E$ on $V_F$ for various faces $F$ can be glued together.
    It is enough to consider two faces. Indeed, if the following claim (\ref{twofaces}) is true, for any pairs of faces $F, F'$ of $M$, we can take a strict neighborhood $V_{F,F'}$ of $\tube{\tilde{Z}_{F}}_P$ in $\tube{Z_{F}}_P$ and a strict neighborhood $V_{F',F}$ of $\tube{\tilde{Z}_{F'}}_P$ in $\tube{Z_{F'}}_P$ such that
      $\E$ can be uniquely extended to $W_1 \cup W_2$ for any strict neighborhood $W_1\subseteq V_{F,F'}$  of $\tube{\tilde{Z}_{F}}_P$ in $\tube{Z_{F_1}}_P$ and any strict neighborhood $W_2\subseteq V_{F',F}$  of $\tube{\tilde{Z}_{F'}}_P$ in $\tube{Z_{F'}}_P$. For any face $F$ of $M$, let $W_F\coloneqq \cap_{F'} V_{F,F'}$ where $F'$ runs through all faces of $M$, then $W_F$ is a strict neighborhood of $\tube{\tilde{Z}_{F}}_P$ in $\tube{Z_{F}}_P$ because the number of faces of $M$ is finite. The extensions of $\E$ on $W_F$ for all faces $F$ of $M$ can be glued together.
    \begin{claim}\label{twofaces}
      For any two faces $F_1$, $F_2$ of $M$, for any sufficiently small strict neighborhood $V_1$ of $\tube{\tilde{Z}_{F_1}}_P$
      in $\tube{Z_{F_1}}_P$ and any sufficiently small strict neighborhood $V_2$ of $\tube{\tilde{Z}_{F_2}}_P$ in $\tube{Z_{F_2}}_P$,
      $\E$ can be uniquely extended to $V_1 \cup V_2$.
    \end{claim}
    \begin{proof}
      It is enough to show the claim in the case of $F_1$ and $F_1\cap F_2$ and in the case of $F_2$ and $F_1\cap F_2$.
      Indeed, $F_1\cap F_2$ is also face and $Z_{F_1\cap F_2}=Z_{F_1}\cap Z_{F_2}$.
      For any strict neighborhood $V_{1,2}$ of $\tube{\tilde{Z}_{F_1\cap F_2}}_P$ in $\tube{Z_{F_1\cap F_2}}_P$,
      there exist a strict neighborhood $V_{1}$ of $\tube{\tilde{Z}_{F_1}}_P$ in $\tube{Z_{F_1}}_P$
      and a strict neighborhood $V_2$ of $\tube{\tilde{Z}_{F_2}}_P$ in $\tube{Z_{F_2}}_P$
      such that $V_1\cap V_2\subseteq V_{1,2}$.
      So for any sufficiently small strict neighborhood $V_{1}$ of $\tube{\tilde{Z}_{F_1}}_P$ in $\tube{Z_{F_1}}_P$
      and any sufficiently small strict neighborhood $V_2$ of $\tube{\tilde{Z}_{F_2}}_P$ in $\tube{Z_{F_2}}_P$,
      $V_1\cap V_2$ is a sufficiently small strict neighborhood of $\tube{\tilde{Z}_{F_1\cap F_2}}_P$ in $\tube{Z_{F_1\cap F_2}}_P$.
      If the claim is true in the case of $F_1$ and $F_1\cap F_2$ and in the case of $F_2$ and $F_1\cap F_2$,
      $\E$ can be uniquely extended to $V_1 \cup \lr{V_1\cap V_2}$ and to $V_2 \cup\lr{V_1\cap V_2}$, so it can be also
      uniquely extended to $V_1 \cup V_2$.
      Hence we may assume that $F_1\supseteq F_2$.

      Take a section $s_2:M/F_2\rightarrow F_2^{-1}M$ of $F^{-1}_2M_2\twoheadrightarrow M/F_2$ and
      a section $\ov{s_1}: M/F_1\rightarrow F_1^{-1}M/F_2$ of $F_1^{-1}M/F_2\twoheadrightarrow M/F_1$.
      $s_2$ can be extended to $F_1^{-1}s_2: F_1^{-1}M/F_2\rightarrow F_1^{-1}M$.
      Let $s_1\coloneqq \lr{F_1^{-1}s_2}\circ \ov{s_1}: M/F_1\rightarrow F_1^{-1}M$ which is a section of ${F_1}^{-1}M\twoheadrightarrow M/F_1$
      
      Let $M_{s_i}=\Im\op{s_i}+M$,
      \[\lr{P_{s_i}',\mc{L}_{s_1}'}=\lr{\lr{P,\mc{L}}\fibpro{\tA{M}}\tA{M_i}}\fibpro{R\times \tA{F_i}\times\tA{M/F_i}}\lr{\lr{Q_{F_i},\mc{L}_{F_i}}\times \tA{M/F_i}}\]
      as above for $i=1,2$.
      
      Let $\underline{s_2}\coloneqq s_2|_{F_1/F_2}:F_1/F_2\rightarrow F_2^{-1}F_1$.
      Let $F_{1,s_2}\coloneqq \Im\op{\underline{s_2}}+F_1/F_2$,
      \[\lr{Q_{F_1,s_2}',\mc{L}_{F_1,s_2}'}\coloneqq\lr{\lr{Q_{F_1},\mc{L}_{F_1}}\fibpro{\tA{F_1}}\tA{F_{1,s_2}}}\fibpro{R\times \tA{F_2}\times\tA{F_1/F_2}}\lr{\lr{Q_{F_2},\mc{L}_{F_2}}\times \tA{M/F_2}}.\]

      Let $\lr{M/F_2}_{s_1}=\Im\op{\ov{s_1}}+M/F_2$.

      Then the following diagram commutes:
      \[\begin{tikzcd}
        &\lr{Q_{F_1},\mc{L}|_{Q_{F_1}}}\arrow{r}& \lr{P,\mc{L}}\\
        &\lr{Q_{F_1},\mc{L}_{s_1}'|_{Q_{F_1}}}\arrow{r}\arrow{u}\arrow{d}& \lr{P_{s_1}',\mc{L}_{s_1}'}\arrow{u}\arrow{d}\\
      \lr{\tilde{Q}_{F_1},\mc{L}|_{\tilde{Q}_{F_1}}}\arrow[hook]{ruu}\arrow[hook]{ru}\arrow[hook]{r}&\lr{Q_{F_1},\mc{L}_{F_1,M/F_1}}\arrow{r}&\lr{Q_{F_1},\mc{L}_{F_1}}\times \tA{M/F_1}\\
        &\lr{Q_{F_2},\mc{L}_{F_1,s_2,M/F_1}}\arrow{r}\arrow{u}\arrow{d}& \lr{Q_{F_1,s_2}',\mc{L}_{F_1,s_2}'}\times \tA{M/F_1}\arrow{u}\arrow{d}\\      
      &\lr{Q_{F_2},\mc{L}_{F_2,F_1/F_2,M/F_1}}\arrow{r}&\lr{Q_{F_2},\mc{L}_{F_2}}\times \tA{F_1/F_2}\times \tA{M/F_1}\\
      \lr{\tilde{Q}_{F_2},\mc{L}|_{\tilde{Q}_{F_2}}}\arrow[hook]{uuu} \arrow[hook]{ruu}\arrow[hook]{ru}\arrow[hook]{r}\arrow[hook]{rd}\arrow[hook]{rdd}\arrow[hook]{rddd}  & \lr{Q_{F_2},\mc{L}_{F_2,\lr{M/F_2}_{s_1}}}\arrow{r}\arrow{u}\arrow{d}& \lr{Q_{F_2},\mc{L}_{F_2}}\times \tA{\lr{M/F_2}_{s_1}}\arrow{u}\arrow{d}\\
        & \lr{Q_{F_2},\mc{L}_{F_2,M/F_2}}\arrow{r}& \lr{Q_{F_2},\mc{L}_{F_2}}\times \tA{M/F_2}\\
        & \lr{Q_{F_2},\mc{L}_{s_2}'|_{Q_{F_2}}}\arrow{r}\arrow{d}\arrow{u}& \lr{P_{s_2}',\mc{L}_{s_2}'}\arrow{d}\arrow{u}\\
        & \lr{Q_{F_2},\mc{L}|_{Q_{F_2}}}\arrow{r}& \lr{P,\mc{L}}
      \end{tikzcd}\]
      where $\mc{L}_{F_1,M/F_1}$,$\mc{L}_{F_1,s_2,M/F_1}$, $\mc{L}_{F_2,F_1/F_2,M/F_1}$, $\mc{L}_{F_2,\lr{M/F_2}_{s_1}}$ and $\mc{L}_{F_2,M/F_2}$
      are log structures which are defined as the pullbacks by the horizontal arrows in the above diagram.
      Any sufficiently small strict neighborhood $V_1$ of $\tube{\tilde{Z}_{F_1}}_P$ in $\tube{Z_{F_1}}_P$ is isomorphic to some strict neighborhood $V_1'$ of
      $\lr{\tilde{Q}_{F_1}}_{K}\times \A{M/F_1}{K}\!\clop{0,1}$ in $\lr{Q_{F_1}}_K\times \A{M/F_1}{K}\!\clop{0,1}$. We may assume $V_1'$
      is of the form of $\bigcup_{i\in\N}V_{1,i}\times \A{M/F_1}{K}\!\clop{0,b_i}$ where $V_{1,i}$'s are
      strict neighborhoods of $\lr{\tilde{Q}_{F_1}}_K$ in $\lr{Q_{F_1}}_K$ and
      $b_i\rightarrow 1$ because the set of strict neighborhoods of this form forms a fundamental system of strict neighborhoods.

      On the other hand, any sufficiently small strict neighborhood $V_2$ of $\tube{\tilde{Z}_{F_2}}_P$ in $\tube{Z_{F_2}}_P$ is isomorphic to some strict neighborhood $V_2'$ of
      $\lr{\tilde{Q}_{F_2}}_{K}\times \A{F_1/F_2}{K}\!\clop{0,1}\times\A{M/F_1}{K}\!\clop{0,1}$
      in $\lr{Q_{F_2}}_K\times \A{F_1/F_2}{K}\!\clop{0,1}\times \A{M/F_1}{K}\!\clop{0,1}$. Likewise, we may assume $V_2'$
      is of the form of $\bigcup_{i\in\N}V_{2,i}\times \A{F_1/F_2}{K}\!\clop{0,b_i}\times \A{M/F_1}{K}\!\clop{0,b_i}$ where $V_{2,i}$'s are
      strict neighborhoods of $\lr{\tilde{Q}_{F_2}}_{K}$ in $\lr{Q_{F_2}}_K$ and $b_i\rightarrow 1$.

      Any sufficiently small strict neighborhood $V_3$ of $\tube{\tilde{Z}_{F_2}}_{Q_{F_1}}$ in $\tube{Z_{F_2}}_{Q_{F_1}}$ is isomorphic to some strict neighborhood $V_3'$
      of $ \lr{\tilde{Q}_{F_2}}_{K}\times \A{F_1/F_2}{K}\!\clop{0,1}$ in $\lr{Q_{F_2}}_K\times \A{F_1/F_2}{K}\!\clop{0,1}$.
      Let $g:V_3\xrightarrow{\sim} V_3'$ be the isomorphism. We may also assume that $V_{2,i}\times \A{F_1/F_2}{K}\!\clop{0,b_i}\subseteq V_3'$ for each $i\in \N$.
      We have to prove that the extensions of $\mc{E}$ coincide on $\lr{V_{1,i}\cap g^{-1}\op{V_{2,i}\times \A{F_1/F_2}{K}\!\clop{0,b_i}}}\times \A{M/F_1}{K}\!\clop{0,b_i}$.
      Fix $i$. If $b_i$ is close enough to $1$, $\mc{E}$ defines a log $\nabla$-module on $V_{2,i}\times \A{F_1/F_2}{K}\op{a,b_i}\times \A{M/F_1}{K}\op{a,b_i}$ for some $a\in \lr{0,b_i}$. We may assume that it is $\Sigma$-unipotent by overconvergent generization.
      We may assume that \[g\op{V_{1,i}\cap g^{-1}\op{V_{2,i}\times \A{F_1/F_2}{K}\!\clop{0,b_i}}}\subseteq V_{2,i}\times\A{F_1/F_2}{K}\op{a,b_i}.\]
      By Proposition \ref{UULNM}, the extension of $\Sigma$-unipotent $\nabla$-module on \[g\op{V_{1,i}\cap g^{-1}\op{V_{2,i}\times \A{F_1/F_2}{K}\!\clop{0,b_i}}}\times \A{M/F_1}{K}\op{a,b_i}\] to \[g\op{V_{1,i}\cap g^{-1}\op{V_{2,i}\times \A{F_1/F_2}{K}\!\clop{0,b_i}}}\times \A{M/F_1}{K}\!\clop{0,b_i}\] is unique. So the extensions of $\mc{E}$ coincides.
    \end{proof}
    
    Conversely, let $E_P$ be a convergent log $\nabla$-module on $P_K$
    which has the exponents in $\Sigma$.
    Let $F$ be a face of $M$. The restriction $E_{\tilde{P}_F}$ of $E_P$ to
    $\tilde{P}_F$ is also a convergent log $\nabla$-module.
    Since $N\rightarrow \lr{\ov{\mc{M}_0}}_{\ov{x}}$ is vertical for any geometric point $\ov{x}$ of $\ov{X}$, the assumption of Lemma \ref{6.3.4} holds.
    By Lemma \ref{6.3.4}, the restriction of $E_{\tilde{P}_F}$ to
    $\tube{\tilde{Z}_F}_{P}=\tube{\tilde{Z}_F}_{Q_F}\times \A{M/F}{K}\!\clop{0,1}$ is log-convergent and has the exponents in $\Sigma_F$. By Proposition
    \ref{transfer}, this is $\Sigma_F$-unipotent. So a convergent log $\nabla$-module on $P_K$ which has the exponents in $\Sigma$ is restricted to
    a $\Sigma$-unipotent overconvergent log isocrystal.
\end{proof}
By this proposition, we have the well-definedness of $\Sigma$-unipotence.
\begin{prop}
  In Situation \ref{situation}, we also assume that $N\rightarrow \lr{\ov{\mc{M}_0}}_{\ov{x}}$ is vertical
  at a geometric point $\ov{x}$.
  Let $\Sigma\subseteq M^{\mr{gp}}\tens{\Z}\Z_p$ be an ($\lr{\mr{NI}\cap\mr{NL}}$-D) subset.
  The $\Sigma$-unipotence does not depend the choice of good charted standard small frames, i.e.,
  for two charted standard small frames \[\lr{\lr{P_1,\mc{L}_1},M_0\oplus M\rightarrow \O{P_1},\gamma_{1,1},\ldots \gamma_{1,d}}\] and \[\lr{\lr{P_2,\mc{L}_2},M_0\oplus M\rightarrow \O{P_2},\gamma_{2,1},\ldots \gamma_{2,d}}\] which are good at $\ov{x}$,
  if an overconvergent isocrystal $\mc{E}$ on $\lr{X,\ov{X},\mc{M}_0\oplus\mc{M}}/\OKa$ satisfies the condition of $\Sigma$-unipotence
  with respect to the first frame, then $\mc{E}$ also satisfies the condition of $\Sigma$-unipotence with respect to the second frame.
\end{prop}
\begin{proof}
  Take $P_3$ and $s$ as in the proof of Proposition \ref{well-def_exp}.
  If $\E$ is $\Sigma$-unipotent with respect to the first frame, $\E$ is extended to a convergent log $\nabla$-module $E_{P_1}$ on $\lr{P_1}_K$
  by Proposition \ref{local extension}.
  Let $E_{P_2}=\lr{\pi_1\circ s}^*\op{E_{P_1}}$. Then it is a convergent log $\nabla$-module which has exponents in $\Sigma$.
  Take a face $F$ of $M$. Let $\Sigma_F$ be as in the definition of $\Sigma$-unipotence.
  Put
  \begin{align*}
    Q_{2,F}&\coloneqq\Set{p \in P_2| \forall m \in M\setminus F: \beta\op{m}\op{p}=0},\\
    \tilde{P}_{2,F}&\coloneqq\Set{p\in P_2|\forall m\in F: \beta\op{m}\op{p}\neq 0},\\
    \tilde{Q}_{2,F}&\coloneqq Q_{2,F}\cap \tilde{P}_{2,F}.
  \end{align*}
  and $Z_{2,F}\coloneqq \lr{Q_{2,F}}_k$, $\tilde{Z}_{2,F}\coloneqq \lr{\tilde{Q}_{2,F}}_k$.
  The restriction of $E_{P_2}$ to $\tube{\tilde{Z}_{2,F}}_{P_2}\cong\tube{\tilde{Z}_{2,F}}_{Q_{2,F}}\times \A{M/F}{K}\!\clop{0,1}$ is a convergent log $\nabla$-module which has exponents in $\Sigma_F$.
  By Lemma \ref{6.3.4}, it is log-convergent. By Proposition \ref{transfer}, it is $\Sigma_F$-unipotent. The log $\nabla$-module on $\tube{\tilde{Z}_{2,F}}_{Q_{2,F}}\times \A{M/F}{K}\lr{a,1}$ for some $a\in\lr{0,1}\cap \Gamma^*$ defined by $\E$ is
  the restriction of $E_{P_2}$, so it is also $\Sigma_F$-unipotent. Thus $\E$ is $\Sigma$-unipotent with respect to the second frame.
\end{proof}

Finally, we consider the global case.

\begin{defn}
  Let $\lr{\ov{X},\mc{M}}$ be a log variety over $\ka$.
  Let $\mc{S}\subseteq \ov{\mc{M}}^{\mr{gp}}\tens{\Z}\Kbar$ be a subsheaf and $S\subseteq \Kbar$ a subset.
  $\mc{S}$ is called \textit{($S$-D)} if for any geometric point $\ov{x}$ of $\ov{X}$,
  $\mc{S}_{\ov{x}}\subseteq \ov{\mc{M}}^{\mr{gp}}_{\ov{x}}\tens{\Z}\Kbar$ is ($S$-D) and
  there exists a good chart $\ov{\mc{M}}_{\ov{x}}\rightarrow \mc{M}$ on some neighborhood $U$ of $\ov{x}$ such that
  $\mc{S}$ is equal to the image of $\mc{S}_{\ov{x}}$ under the map $\ov{\mc{M}}^{\mr{gp}}_{\ov{x}}\tens{\Z}\Kbar\rightarrow \ov{\mc{M}}^{\mr{gp}}\tens{\Z}\Kbar$ on $U$.
\end{defn}
\begin{defn}
  In Situation \ref{situation},
  we also assume that $N\rightarrow \lr{\ov{\mc{M}_0}}_{\ov{x}}$
  is vertical at any geometric point $\ov{x}$.
  For a geometric point $\ov{x}$ of $\ov{X}$ and an ($\lr{\mr{NI}\cap\mr{NL}}$-D) subset $\Sigma\subseteq \ov{\mc{M}}_{\ov{x}}^{\mr{gp}}\tens{\Z}\Z_p$,
  An overconvergent log isocrystal $\mc{E}$ on $\lr{X,\ov{X},\mc{M}_0\oplus\mc{M}}$ is \textit{$\Sigma$-unipotent at $\ov{x}$} if
  $\E$ is $\Sigma$-unipotent with respect to charted standard small frames which are good at $\ov{x}$.
  Let $\mc{S}\subseteq \ov{\mc{M}}^{\mr{gp}}\tens{\Z}\Kbar$ be an ($\lr{\mr{NI}\cap\mr{NL}}$-D) subsheaf.
  $\mc{E}$ is \textit{$\mc{S}$-unipotent} if, for any geometric point $\ov{x}$ of $\ov{X}$,
  $\E$ is $\mc{S}_{\ov{x}}$-unipotent at $\ov{x}$.
\end{defn}

\begin{rem}
  By Remark \ref{Sigma_F unipotence}, if $\mc{E}$ is $\mc{S}_{\ov{x}}$-unipotent at $\ov{x}$, there exists some neighborhood $U$ of $\ov{x}$ such that
  $\mc{E}$ is $\mc{S}_{\ov{y}}$-unipotent at any geometric point $\ov{y}$ of $U$.
\end{rem}

The following theorem is the main result of this paper.

\begin{thm}
  In Situation \ref{situation},
  we also assume that $N\rightarrow \lr{\ov{\mc{M}_0}}_{\ov{x}}$
  is vertical at any geometric point $\ov{x}$.
  Let $\mc{S}\subseteq \ov{\mc{M}}^{\mr{gp}}\tens{\Z}\Z_p$ be an ($\lr{\mr{NI}\cap\mr{NL}}$-D)
  subsheaf.
  There exists a equivalence between the category of $\mc{S}$-unipotent overconvergent log isocrystals on $\lr{X,\ov{X},\mc{M}_0\oplus\mc{M}}/\OKa$
  and the category of locally free convergent log isocrystals on $\lr{\ov{X},\mc{M}_0\oplus\mc{M}}/\OKa$ which have exponents in $\mc{S}$.
\end{thm}
\begin{proof}
  It is enough to show it \'etale locally, so we can assume the existence of a charted standard small frame.
  Then it is the consequence of Proposition \ref{local extension} and Proposition \ref{cat_equiv}.
\end{proof}
\begin{rem}\label{codimension 2}
  The definition of unipotence in this paper is stronger than that of \cite{Shi}, because only monodromies at codimension one components
  are considered in \cite{Shi}. In our situation, the consideration to codimension one components is not enough.

  For example, assume $p\neq 2$ and let $M\coloneqq \Set{\lr{a_1,a_2}\in \N^2| a_1+a_2\equiv 0 \pmod{2}}$. Then $M$ is a fine sharp semi-saturated monoid.
  Put $\lr{\ov{X},\mc{M}}\coloneqq \A{M}{k}$ and $\lr{P,\mc{L}}\coloneqq\tA{M}$.
  Let $x\coloneqq t^{\lr{2,0}}$, $y\coloneqq t^{\lr{0,2}}$ and $z\coloneqq t^{\lr{1,1}}$. The trivial locus $X$ of $\mc{M}$
  is $\ov{X}\setminus\lr{\brace{x=0}\cup\brace{y=0}}$.
  Let $\lr{E,\nabla}$ be a free $\nabla$-module on $P_K\setminus \lr{\brace{x=0}\cup\brace{y=0}}$ of rank $1$ such that
  \[\nabla\op{\mb{e}}=\mb{e}\frac{dx}{2x},\]
  where $\mb{e}$ is a basis of $E$.
  Let $\mc{E}$ be an overconvergent log isocrystal on $\lr{X,\ov{X},\mc{M}}/K$ associated to $\lr{E,\nabla}$.
  Clearly $\mc{E}$ is $\brace{0}$-constant at the points of $\brace{x\neq 0,y=0}$. Moreover, $\mc{M}$
  is also $\brace{0}$-constant at the points of $\brace{x=0,y\neq 0}$. Indeed,
  \[\nabla\op{z^{-1}\mb{e}}=-z^{-1}\mb{e}\frac{dy}{2y}\]
  and $z^{-1}\mb{e}$ is also a basis of $E$. But $\mc{E}$ is not $\brace{0}$-unipotent at the vertex of $\ov{X}$.

  This phenomenon is caused by the following fact on monoids: Let $F_1$ and $F_2$ be the two facets of $M$.
  Then $M\rightarrow M/F_1\oplus M/F_2$ is injective.
  But $\Sigma$-unipotence is only dependent on the image of $\Sigma$ in $M^{\mr{gp}}\tens{\Z}\lr{\Z_p/\Z}$ by Remark
  \ref{unipotence/Z2} and $M^{\mr{gp}}\tens{\Z}\lr{\Z_p/\Z} \rightarrow \lr{\lr{M/F_1}^{\mr{gp}}\tens{\Z}\lr{\Z_p/\Z}}\oplus \lr{\lr{M/F_2}^{\mr{gp}}\tens{\Z}\lr{\Z_p/\Z}}$ is not injective. Indeed,
  $\lr{2,0}\tens{}\ov{\frac{1}{2}}\in M^{\mr{gp}}\tens{\Z}\lr{\Z_p/\Z}$ is a non-zero element of the kernel of this map.
  So the $\brace{0}$-unipotence at $\tilde{Z}_{F_1}$ and $\tilde{Z}_{F_2}$ does not imply the $\brace{0}$-unipotence at $\tilde{Z}_{F_1\cap F_2}$.
\end{rem}

\end{document}